\documentclass[10pt,a4paper]{amsart}[2000]
\usepackage{geometry}
\usepackage{comment}
\usepackage{amsmath}
\usepackage{amssymb}
\usepackage{amsthm}
\usepackage{mathrsfs}
\usepackage[hang]{footmisc}

\usepackage[linktocpage]{hyperref}
\usepackage{tikz, tikz-cd}
\usepackage{color}
\usepackage{mathtools}
\usepackage{enumitem}
\usepackage{stackengine}

\newcommand{\bigand}{\quad \text{and} \quad }

\usepackage{graphicx}
\usepackage{esvect}
\makeatletter
\DeclareRobustCommand{\cev}[1]{%
  {\mathpalette\do@cev{#1}}%
}
\newcommand{\do@cev}[2]{%
  \vbox{\offinterlineskip
    \sbox\z@{$\m@th#1 x$}%
    \ialign{##\cr
      \hidewidth\reflectbox{$\m@th#1\vec{}\mkern4mu$}\hidewidth\cr
      \noalign{\kern-\ht\z@}
      $\m@th#1#2$\cr
    }%
  }%
}
\makeatother

\thanks{}

\allowdisplaybreaks
\newcommand{\dblsetminus}{\mathbin{{\setminus}\mspace{-5mu}{\setminus}}}

\theoremstyle{plain}

\newtheorem{Thm}{Theorem}[section]

\newtheorem{Claim}[Thm]{Claim}

\theoremstyle{definition}

\newtheorem{Rmk}[Thm]{Remark}

\theoremstyle{plain}
\newtheorem{thm}[Thm]{Theorem}
\newtheorem{lem}[Thm]{Lemma}
\newtheorem{cor}[Thm]{Corollary}
\newtheorem{prop}[Thm]{Proposition}

\theoremstyle{definition}
\newtheorem{defn}[Thm]{Definition}

\newtheorem{eg}[Thm]{Example}
\newtheorem{rmk}[Thm]{Remark}

\newtheorem{innercustomthm}{Theorem}
\newenvironment{customthm}[1]
{\renewcommand\theinnercustomthm{#1}\innercustomthm}
{\endinnercustomthm}

\newcommand{\B}{B}
\newcommand{\A}{A}
\newcommand{\J}{J}
\newcommand{\K}{\mathcal{K}}
\newcommand{\D}{D}
\newcommand{\Ch}{D}
\newcommand{\Zh}{\mathcal{Z}}
\newcommand{\E}{E}
\newcommand{\Oh}{\mathcal{O}}

\newcommand{\T}{{\mathbb T}}
\newcommand{\R}{{\mathbb R}}
\newcommand{\N}{{\mathbb N}}
\newcommand{\Z}{{\mathbb Z}}
\newcommand{\C}{{\mathbb C}}
\newcommand{\Q}{{\mathbb Q}}

\newcommand{\la}{{\langle}}
\newcommand{\ra}{{\rangle}}

\newcommand{\aut}{\mathrm{Aut}}
\newcommand{\supp}{\mathrm{supp}}

\newcommand{\eps}{\varepsilon}
\numberwithin{equation}{section}

\newcommand{\id}{\mathrm{id}}

\newcommand{\halpha}{\widehat{\alpha}}
\newcommand{\calpha}{\widehat{\alpha}}

\newcommand{\tih}{\widetilde {h}}

\newcommand\set[1]{\left\{#1\right\}}  
\newcommand\mset[1]{\left\{\!\!\left\{#1\right\}\!\!\right\}}

\newcommand{\IE}[0]{\mathbb{E}}

 \newcommand{\IX}[0]{\mathbb{X}}

\newcommand{\CA}[0]{\mathcal{A}} \newcommand{\CB}[0]{\mathcal{B}}
\newcommand{\CC}[0]{\mathcal{C}} \newcommand{\CD}[0]{\mathcal{D}}
 \newcommand{\CF}[0]{\mathcal{F}}
\newcommand{\CG}[0]{\mathcal{G}} \newcommand{\CH}[0]{\mathcal{H}}

\newcommand{\CO}[0]{\mathcal{O}} 
\newcommand{\CQ}[0]{\mathcal{Q}} 
\newcommand{\CS}[0]{\mathcal{S}} \newcommand{\CT}[0]{\mathcal{T}}
\newcommand{\CU}[0]{\mathcal{U}}

\newcommand{\Ra}[0]{\Rightarrow}
\newcommand{\La}[0]{\Leftarrow}
\newcommand{\LRa}[0]{\Leftrightarrow}

\newcommand{\quer}[0]{\overline}
\newcommand{\eins}[0]{\mathbf{1}}			
\newcommand{\diag}[0]{\operatorname{diag}}
\newcommand{\ad}[0]{\operatorname{Ad}}
\newcommand{\ev}[0]{\operatorname{ev}}
\newcommand{\fin}[0]{{\subset\!\!\!\subset}}
\newcommand{\diam}[0]{\operatorname{diam}}
\newcommand{\Hom}[0]{\operatorname{Hom}}
\newcommand{\dst}[0]{\displaystyle}
\newcommand{\spp}[0]{\operatorname{supp}}
\newcommand{\lsc}[0]{\operatorname{Lsc}}
\newcommand{\del}[0]{\partial}
\newcommand{\GU}[0]{\CG^{(0)}}

\newcommand{\ax}[0]{\operatorname{Ax}}

\theoremstyle{definition}

\numberwithin{equation}{Thm}

\newcommand{\ywy}[1]{{\color{red}{#1}}}

\title[Boundary actions]{Boundary actions of $\cat$ spaces: topological freeness and applications to $C^\ast$-algebras}	

\begin{document}
\global\long\def\floorstar#1{\lfloor#1\rfloor}
\global\long\def\ceilstar#1{\lceil#1\rceil}	

\global\long\def\B{B}
\global\long\def\A{A}
\global\long\def\J{J}
\global\long\def\K{\mathcal{K}}
\global\long\def\D{D}
\global\long\def\Ch{D}
\global\long\def\Zh{\mathcal{Z}}
\global\long\def\E{E}
\global\long\def\Oh{\mathcal{O}}

\global\long\def\T{{\mathbb{T}}}
\global\long\def\BR{{\mathbb{R}}}
\global\long\def\N{{\mathbb{N}}}
\global\long\def\Z{{\mathbb{Z}}}
\global\long\def\C{{\mathbb{C}}}
\global\long\def\Q{{\mathbb{Q}}}

\global\long\def\aut{\mathrm{Aut}}
\global\long\def\supp{\mathrm{supp}}

\global\long\def\eps{\varepsilon}

\global\long\def\id{\mathrm{id}}

\global\long\def\halpha{\widehat{\alpha}}
\global\long\def\calpha{\widehat{\alpha}}

\global\long\def\tih{\widetilde{h}}

\global\long\def\opFol{\operatorname{F{\o}l}}

\global\long\def\opRange{\operatorname{Range}}

\global\long\def\opIso{\operatorname{Iso}}
\global\long\def\opisom{\operatorname{Isom}}
\global\long\def\dimnuc{\dim_{\operatorname{nuc}}}

\global\long\def\set#1{\left\{  #1\right\}  }


\global\long\def\mset#1{\left\{  \!\!\left\{  #1\right\}  \!\!\right\}  }

\global\long\def\Ra{\Rightarrow}
\global\long\def\La{\Leftarrow}
\global\long\def\LRa{\Leftrightarrow}

\global\long\def\quer{\overline{}}
\global\long\def\eins{\mathbf{1}}
\global\long\def\diag{\operatorname{diag}}
\global\long\def\ad{\operatorname{Ad}}
\global\long\def\ev{\operatorname{ev}}
\global\long\def\fin{{\subset\!\!\!\subset}}
\global\long\def\diam{\operatorname{diam}}
\global\long\def\Hom{\operatorname{Hom}}
\global\long\def\dst{{\displaystyle }}
\global\long\def\spp{\operatorname{supp}}
\global\long\def\spo{\operatorname{supp}_{o}}
\global\long\def\del{\partial}
\global\long\def\lsc{\operatorname{Lsc}}
\global\long\def\GU{\CG^{(0)}}
\global\long\def\HU{\CH^{(0)}}
\global\long\def\AU{\CA^{(0)}}
\global\long\def\BU{\CB^{(0)}}
\global\long\def\CUU{\CC^{(0)}}
\global\long\def\DU{\CD^{(0)}}
\global\long\def\QU{\CQ^{(0)}}
\global\long\def\TU{\CT^{(0)}}
\global\long\def\CUUU{\CC'{}^{(0)}}
\global\long\def\dom{\operatorname{dom}}
\global\long\def\ran{\operatorname{ran}}
\global\long\def\AUl{(\CA^{l})^{(0)}}
\global\long\def\BUl{(B^{l})^{(0)}}
\global\long\def\HUp{(\CH^{p})^{(0)}}
\global\long\def\sym{\operatorname{Sym}}
\global\long\def\stab{\operatorname{Stab}}
\newcommand{\cat}[0]{\operatorname{CAT}(0)}
\global\long\def\properlength{proper}
\global\long\def\deg{\operatorname{deg}}
\global\long\def\isom{\operatorname{Isom}}
\global\long\def\interior#1{#1^{\operatorname{o}}}
	\global\long\def\ln{\operatorname{ln}}

\author{Xin Ma}
	
	\address{X. Ma:  Institute for Advanced Study in Mathematics, Harbin Institute of Technology, Harbin, China, 150001}
  \email{xma17@hit.edu.cn}


\author{Daxun Wang}
	\address{D. Wang: Yau Mathematical Sciences Center, Tsinghua University, Beijing, China}
\email{wangdaxun@mail.tsinghua.edu.cn}

\author{Wenyuan Yang}
\address{W. Yang: Beijing International Center for Mathematical Research, Peking University, Beijing 100871, China P.R.}
\email{wyang@math.pku.edu.cn}


\date{\today}

\begin{abstract}
	In this paper, we study topological dynamics on  the visual boundary and several combinatorial boundaries associated to $\operatorname{CAT}(0)$ spaces. Through verifying the freeness of Myrberg points on the boundaries,  we prove that a large class of these  boundary actions are topologically free strong boundary actions. These include certain visual boundary actions obtained from proper isometric actions of groups on proper $\operatorname{CAT}(0)$ spaces with  rank-one elements, horofunction boundary actions from  actions of irreducible finitely generated infinite non-affine Coxeter groups on the Caylay graphs, and Roller-type boundary actions from certain group actions on irreducible $\cat$ cube complexes. This  in particular leads to a new proof  of Kar-Sageev's topological freeness result for Roller boundary actions of $\operatorname{CAT}(0)$ cube complexes and  generalizes Klisee's topological freeness result on horofunction boundaries from hyperbolic and right angled Coxeter groups to the general case.  As applications to $C^*$-algebras, our work yields new examples of $C^*$-selfless groups and of exact, purely infinite, simple reduced crossed product $C^\ast$-algebras. 
\end{abstract}
\maketitle


\section{Introduction}
A discrete group $G$ is said to be \textit{$C^\ast$-simple} if the reduced group $C^\ast$-algebra $C^\ast_r(G)$ is simple, i.e., has no non-trivial two-sided closed ideals. The study of $C^\ast$-simplicity in discrete groups has developed into a major research direction at the intersection of geometric group theory and operator algebras, since  Powers' seminal work on  free groups \cite{Powers}. For a comprehensive overview of this property, we refer to the survey \cite{Harpe}.

\begin{defn}
A topological boundary action of $G$ on a compact space $Z$ is called \textit{topologically free} if  the set of points with trivial stabilizer called free points is  dense in $Z$. The action $G\curvearrowright Z$ is a \textit{strong boundary action} (or \textit{extreme proximal}) if for any compact set $F\neq Z$ and non-empty open set $O$ there is a $g\in G$ such that $gF\subset O$. 
\end{defn} 
Recent work of Kalantar and Kennedy   introduced a powerful boundary approach to the $C^\ast$-simplicity problem by establishing the equivalence between $C^\ast$-simplicity and the topological freeness of the group action on the Furstenberg boundary \cite[Theorem 1.5]{K-K}. By definition, the  abstract Furstenberg boundary serves as a universal boundary in a sense that {any} $G$-boundary is a continuous $G$-equivariant quotient of it (Definition \ref{defn: boundary}). It follows that topological freeness of the action on any $G$-boundary implies topological freeness on the Furstenberg boundary, and hence implies $C^\ast$-simplicity. We refer the reader to \cite{K-K, B-K-K-O,LB} for further related discussions. Moreover, a significant new property for groups, termed \textit{$C^\ast$-selfless}, has been recently introduced by Robert \cite{Ro}. This property is  stronger than $C^\ast$-simplicity and has found important applications in the structure theory of $C^\ast$-algebras, see the recent work of Amrutam-Gao-Elayavalli-Patchell \cite{AGKEP}. Very recently, Ozawa \cite{O} showed that groups admitting topologically free strong  boundary actions are $C^*$-selfless. It is worth emphasizing that,  by the work of Laca and Spielberg \cite{L-S}, topologically free strong boundary actions are also deeply connected to the pure infiniteness of the crossed product $C^\ast$-algebras. Pure infiniteness has long been recognized as a fundamental ingredient in the celebrated Kirchberg–Phillips classification of $C^\ast$-algebras by K-theory (see, e.g., \cite{Kir, Phillips}). These connections highlight topologically free strong boundary actions  as  important dynamical objects in the study of geometric group theory and  operator algebras.

Since their introduction by Gromov, $\cat$ spaces have constituted a fundamental class of metric spaces with non-positive Alexandrov curvature. This class encompasses a wide range of geometric objects, including simply connected Riemannian manifolds of non-positive sectional curvature, Euclidean buildings, and numerous CW complex examples from geometric group theory, such as the Davis complex of Coxeter groups and $\cat$ cube complexes. This paper first aims to investigate the topological free actions  of discrete groups on visual boundaries of $\cat$ spaces. Motivated by an analogous study by Kar and Sageev \cite{K-S} on the Roller boundary  of $\cat$ cube complexes, our work also establishes topological free actions on the Roller boundary of Coxeter groups using the  recent work by Ciobanu-Genevois \cite{CG25} and Lam-Thomas \cite{LT15}.

Let $X$ denote a proper $\cat$ metric space and $\del_\infty X$ the visual boundary for $X$ (see \textsection \ref{SSubVisualBoundary}). If $X$  has  a cube complex structure, then a combinatorial boundary  \textit{Roller boundary} $\del_R X$ could be defined using half-spaces in $X$ (see \textsection \ref{SSubRollerBoundary}). A  group $G$ acting properly by isometry on $X$ naturally acts by homeomorphism on the visual and Roller boundaries  $\del_\infty X$ and $\del_R X$. The group $G$ we are considering is always \textit{non-elementary}, meaning it is not a virtually cyclic group. In this paper, unless otherwise mentioned, we  always assume that $G$ is  countable and discrete.

We assume further that $G$ includes \textit{rank-one} hyperbolic elements in the sense of \cite{B95, B-B}; namely, any axis of these elements does not bound a flat half-plane (see \textsection \ref{SSubRankone}). Such elements are ubiquitous: according  to the celebrated rank rigidity theorem  by Ballmann and Burns-Spatzier, if a closed non-positively curved Riemannian manifold is neither a metric product nor  a locally symmetric space of higher rank, then   the action on the $\cat$ universal cover contains rank-one hyperbolic elements. We refer to \cite[Introduction, p5]{B95} for a detailed discussion. 

In what follows, we first describe in detail the topological free strong boundary action result of these boundary actions, and the new applications to $C^\ast$-algebras.  

\subsection{Topological free strong boundary actions for \texorpdfstring{$\cat$}\ \  groups}
The boundary action $G \curvearrowright \partial_\infty X$ is essentially supported on the \textit{limit set} $\Lambda G$, which consists of the accumulation points of some (or any) orbit $Go$ in $X$ within the visual boundary $\partial_\infty X$. It is well-known that a rank-one element exhibits north-south dynamics with respect to the pair of fixed points in $\partial_\infty X$. This results in the action $G \curvearrowright \Lambda G$ being a \textit{minimal} action. Moreover, there exists a unique maximal finite normal subgroup in $G$, known as the \textit{elliptic radical} $E(G)$, which is precisely the pointwise stabilizer of $\Lambda G$. For further details, see Definition \ref{ERdefn} and the accompanying remark. When the action is co-compact, the limit set $\Lambda G$ coincides with the entire visual boundary $\partial_\infty X$. 

Our first main result establishes $G\curvearrowright \Lambda G$ as topologically free strong boundary action.     

\begin{customthm}{A}[Proposition \ref{prop: pure inf on limit set}, Theorem \ref{thm: topo free rank one}, \ref{thm: topo free rank one2}]\label{thmA}
    Let $G\curvearrowright X$ be a proper isometric  action of a {non-elementary} group $G$ on a proper  $\cat$ space $X$ with a rank-one element. Suppose
    \begin{enumerate}[label=(\roman*)]
        \item the action $G\curvearrowright X$ is cocompact in which case $\Lambda G=\partial_\infty X$; or
        \item the space $X$ is geodesically complete.
    \end{enumerate}
    Assume that the elliptic radical $E(G)$ is trivial. Then  the  action $G\curvearrowright \Lambda G$ is a minimal, topologically free, strong boundary action. 
\end{customthm}
As above-mentioned, the $C^\ast$-simplicity could be characterized by a topological free action on \emph{some} $G$-boundary \cite{K-K}. Thus, Theorem \ref{thmA} offers alternative approach to the $C^\ast$-simplicity of groups with specific actions on $\cat$ spaces. Such groups are acylindrically hyperbolic groups with a trivial elliptic radical, with $C^\ast$-simplicity established in  \cite{DGO}. But we refer to Theorem \ref{thmC} below for new consequences of Theorem \ref{thmA} to $C^*$-algebras.

Before moving on, we wish to explain some crucial ingredients in the proof of the theorem, which we believe are of independent interest. 

The free points we found in the theorem  belong to a large class of limit points called \textit{Myrberg points}. This was introduced in his 1931 approximation theorem  by P. J. Myrberg \cite{Myr}  for Fuchsian groups. Myrberg points correspond to ``quasi-ergodic" geodesic rays on hyperbolic surfaces which approximate any direction with arbitrary accuracy. Myrberg proved that those points are generic in Lebesgue measure on the circle $\mathbb S^1$ when  Fuchsian groups have finite co-area. In 1983, Agard \cite{Aga}    generalized  this   to  higher dimensional Kleinian groups of divergent type, serving as a key ingredient in a different proof of Mostow rigidity theorem. 

Recently, the third-named author \cite{YangConformal} defined  the notion of Myrberg  points in the class of discrete groups with contracting elements on general metric spaces, where the limit set is taken as the set of accumulation points in  horofunction boundary. For $\cat$ spaces, Myrberg points $z\in \Lambda G$ are defined in the visual boundary so that $G(o,z)=\{(go,gz): g\in G\}$ is dense in the product $\Lambda G\times \Lambda G$. If we interpret it properly, this amounts to saying that the projected geodesic ray $[o,z]$ on the quotient $X/G$ returns to any closed rank-one geodesic with arbitrary accuracy. We refer to Definitions \ref{defn:Myrberg} and  \ref{defn: recurrence with arbitrary accuracy} for more details.  
The main technical result we prove in Theorem \ref{thmA} is as follows.
\begin{prop}\label{mainprop}
    Under the assumption of Theorem \ref{thmA},  there are uncountably many Myrberg points that are free. 
\end{prop}

Let us   compare with a very general result of Abbot-Dahmani in acylindrically hyperbolic groups. It is proved in \cite[Proposition 4.1]{A-D19} that if $G$ is an acylindrically hyperbolic group without non-trivial finite normal subgroup, then the action on the Gromov boundary of a hyperbolic space $X$ on which $G$ acts acylindrically and co-boundedly is topologically free.  
By a result of Sisto \cite{Sisto2}, groups with rank-one elements are acylindrically hyperbolic. We remark that the Gromov boundary here is not compact, as $X$ is generally not proper. However, the compactness is a crucial point to achieve $C^\ast$-simplicity and $C^*$-selflessness applications. 

Our proof of Proposition \ref{mainprop} uses crucially several important facts specific to $\cat$ geometry. Indeed, if  $X$ is Gromov hyperbolic and $G$ is torsion-free,  the conclusion of Proposition \ref{mainprop} is quite straightforward by the  classification of isometries  that any hyperbolic isometry has exactly two fixed points, while any parabolic isometry has a unique fixed point in the Gromov boundary. A similar trichotomy classification of isometries on $\cat$ spaces exists, but from a boundary point of view, the  fixed set of a parabolic or hyperbolic isometry  could be very large in the visual boundary. The difficulty  is thus to find a  free boundary point that is not fixed by any parabolic or hyperbolic isometry, as the argument for hyperbolic spaces based on the cardinality of fixed points is no longer valid. In the proof of Proposition \ref{mainprop}, we proved  the three facts of independent interest:  a Myrberg point could not be fixed by a hyperbolic isometry, and if a Myrberg point is  fixed by a parabolic isometry, the fixed set must be singleton by   a result of Fujiwara-Nagano-Shioya \cite{FNS06}, and there are uncountable Myrberg points (Lemma \ref{lem:uncountableMyrberg}). 

In the case that $X$ is a $\cat$ cube complex, there exists an equivariant bijection between Myrberg points in visual boundary and Roller boundary, so Theorem \ref{thmA}  implies  the topological freeness on (the limit set of) Roller boundary.

\begin{customthm}{B}[Theorem \ref{topo free Roller}]\label{thmB}
    Let $G\curvearrowright X$ be a proper essential  isometric action of a non-elementary group on a proper irreducible $\cat$ cube complex $X$ with a rank-one element. If $X$ is not geodesically complete in the $\cat$ metric, assume further that the action is co-compact.    Suppose the elliptic radical $E(G)$ is trivial. Then the  action $G\curvearrowright \Lambda_R G$ is topologically free.  
    \end{customthm}

This  gives a very different  proof of a pivotal result of Kar-Sageev \cite[Proposition 1.1]{K-S}, which is used to demonstrate the $C^\ast$-simplicity of groups with certain nice actions on $\cat$ cube complexes using strong separate boundaries.  Further discussions  are given in Remark \ref{comparing bdry}\eqref{comparing bdry-new proof}. Moreover, in the cocompact action case, one even obtains topological free strong boundary action result as demonstrated in Theorem \ref{thmD}.

\subsection{Topological free strong boundary actions for Coxeter groups}
Let $(W,S)$ be a Coxeter system of  finite rank. By Davis-Moussong theorem, $W$ admits a geometric action on the Davis complex $\Sigma(W,S)$ which is a proper $\cat$ space.  By definition, finite Coxeter groups are \textit{spherical} and virtually abelian Coxeter groups are \textit{affine}. If $W$ is  non-spherical   and non-affine, then $W$ contains rank-one elements by \cite{C-F}. If, in addition, $(W,S)$ is irreducible then the elliptic radical $E(W)$ is trivial. The following is an application of Theorem \ref{thmA}.

\begin{thm}\label{Coxeter top free on visul boundary}
The action of an irreducible non-spherical non-affine  Coxeter group $W$ on the visual boundary  $\partial_\infty \Sigma(W,S)$ is a topologically free and strong boundary action.    
\end{thm} 
See Corollary \ref{topo free for Coxter}  for a statement regarding reducible   Coxeter groups.

 
By construction, the 1-skeleton  of the Davis complex $\Sigma(W,S)$   is the Cayley graph of $W$ relative to $S$, denoted as $X(W,S)$.
Besides the visual and Tits boundaries of the Davis complex, there are several boundaries of combinatorial nature associated to  $X(W,S)$. The main result of this subsection is to establish the topological free   action on the following boundaries:
\begin{enumerate}
\item 
Caprace-Lécureux (minimal) combinatorial boundary $\partial_{sph} X(W,S)$ in \cite{C-L}, 
\item
Klisse's graph boundary $\partial X(W,S)$ in \cite{Kli20},
\item
Genevois' combinatorial boundary $\partial_c X(W,S)$ in \cite{Gen20},
\item
Roller boundary $\partial_R X(W,S)$ in \cite{Roller}. 
\end{enumerate}  
The primary fact is that  all these boundaries are  homeomorphic to  the horofunction boundary $\partial_h X(W,S)$. Indeed, the  homeomorphisms to the boundary in (1) and (2) are due to Caprace-Lécureux \cite{C-L}, and Klisse \cite{Kli20} accordingly, and   $\partial_c X(W,S)\cong \partial_R X(W,S)$ was proved by Genevois for $\cat$ cubical complexes in \cite[Proposition A.2]{Gen20}, which we shall generalize in the current paper to establish  the homeomorphisms between the last three boundaries; see also Remark \ref{rmk: genevois roller}. 

Indeed, we  develop and clarify  the above-mentioned  boundaries in \textsection \ref{subsec paraclique boundary} for  the class of paraclique graphs, which are recently introduced by Ciobanu-Genevois \cite{CG25}. This is a generalization of the (quasi-)median and mediangle graphs previously studied by Genevois \cite{Gen17, Gen22} from the perspective of cubical geometry. Let us briefly note that the 1-skeletons of $\cat$ cube complexes are precisely median graphs, while quasi-median graphs have been studied extensively in graph theory. Mediangle graphs include the Cayley graphs of  Coxeter groups (\cite[Proposition 3.24]{Gen22}). 

We say that Klisse's graph compactification of $X(W,S)$  is   visual if every point admits a geodesic representative (Definition \ref{def visual graph boundary}).  
\begin{thm}[Theorem \ref{allboundaryaresame}]\label{allboundaryaresame: intro}
Let $X$ be a connected paraclique graph. Then the Klisse's graph boundary, Genevois' combinatorial boundary and Roller boundary  are homeomorphic to the horofunction boundary of $X$, provided that the graph compactification is visual.
\end{thm}
\begin{rmk}
Klisse \cite{Kli20} deduced from the weak order that   Coxeter groups have visual graph boundary. We prove  that all quasi-median graphs have visual graph boundary (Lemma \ref{quasimedianisvisual}). Therefore, for these paraclique graphs, all  boundaries in the above theorem are homeomorphic to the horofunction boundary.    
\end{rmk}

From now on, although we are free to use any of the boundaries listed above as needed, the horofunction boundary serves as the natural setting for the main arguments in this work. Therefore, we will insist on the notation $\partial_h X(W,S)$. Analogous to Theorem \ref{Coxeter top free on visul boundary}, our next main result is as follows, which answered a question raised by Jean L\'ecureux and generalized topological freeness results (\cite[Proposition 3.25, Lemma 3.27]{Kli20}) on horofunction boundaries from hyperbolic and right angled Coxeter groups to the general case.
\begin{customthm}{C}\label{ThmE}(Theorem \ref{minimal action and Myrberg sets})
Let $(W,S)$ be an irreducible non-spherical  non-affine Coxeter group.  Then the action of $W$ on  $\partial_h X(W,S)$ is a  minimal, topologically free, topologically amenable and strong boundary action. 
\end{customthm}

The topological amenability of the action is a result established by Lécureux \cite{Lec}. To the best of our knowledge, the other properties listed above are new at this level of generality. The proof strategy mirrors that of Theorem \ref{thmA}, proceeding in two steps: we analyze contracting isometries of  $X(W,S)$ to obtain north-south dynamics on the horofunction boundary $\partial_h X(W,S)$, and then locate free points in the Myrberg limit set, which also lies in this boundary. To elaborate, we begin by describing a partition of  $\partial_h X(W,S)$.

We say two points $\xi,\eta\in \partial_h X(W,S)$ have \textit{finite difference} if the $L_\infty$-difference of their representative horofunctions $b_\xi,b_\eta$ is bounded. Declaring two points \textit{equivalent} when they have finite difference yields an equivalence relation. The resulting equivalence classes denoted as $[\xi]$ form the finite difference partition on $\partial_h X(W,S)$. Thanks to  $\partial_hX(W,S)\cong \partial_R X(W,S)$, we prove in Proposition \ref{finitedifferencesamefinitesymmetric} that this partition on  $\partial_h X(W,S)$ is exactly given by \emph{blocks} of infinite reduced words in $W$ introduced by Lam and Thomas in \cite{LT15}. In the proof of a minimal action, we shall make a crucial use of a partition of the Tits boundary  $\partial_T\Sigma(W,S)$ of Davis complex (studied in \cite{LT15}), which is known to be in one-to-one correspondence with the blocks of infinite reduced words on $\partial_hX(W,S)\cong \partial_R X(W,S)$; see Proposition \ref{LamThomasPartition}.

It has been proved in \cite{YangConformal} that any contracting isometry $g$ on $X(W,S)$ has the north-south dynamics relative to two fixed $[\cdot]$-classes  denoted as $[g^+], [g^-]$ in $\partial_hX(W,S)$; see Lemma \ref{NSDynamics on Horofunction boundary}. Therefore, to get the usual north-south dynamics, we are led to prove that $[g^+], [g^-]$ are singletons.  The recent work \cite{CG25} proves that   rank-one isometries on $\Sigma(W,S)$ in $\cat$ metric are exactly contracting isometries on $X(W,S)$ in combinatorial metric. Building on this work, we further prove the following.

\begin{lem}[Lemma \ref{NorthSouthDynamicsCoxeter}]
Any irreducible non-spherical  non-affine Coxeter group $W$  contains a contracting isometry on $X(W,S)$ with north-south dynamics relative to its \emph{two distinct fixed} points in $\partial_h X(W,S)$.
\end{lem}
\begin{rmk}
The irreducibility of $W$ is necessary, since the direct product of two nontrivial Coxeter groups  has nontrivial finite difference partition on the horofunction boundary: any $[\cdot]$-classes are non-singleton, so no contracting isometry could have singleton fixed points.  Moreover, we present an example (see Example \ref{example with non-minimal fixed points}) demonstrating that not all contracting isometries induce north-south dynamics on the Roller boundary of (irreducible) cubical hyperbolic groups. 
\end{rmk}

From north-south dynamics,  it is standard to derive a unique minimal $W$-invariant closed subset in  $\partial_h X(W,S)$, which may be a proper subset. To establish the minimal action in Theorem \ref{ThmE}, we need to prove the whole boundary is exactly such set.  For visual boundary, the  coincidence between limit set and the whole boundary is a trivial matter due to the geometric action on the Davis complex.  However, this becomes a non-trivial task when the the finite difference partition is nontrivial on $\partial_h X(W,S)$ (note it is trivial on the visual boundary by \cite[Lemma 11.1]{YangConformal}). Indeed, we present a cubical complex example (see Example \ref{example with non-minimal fixed points}) where the limit set is a proper subset of the Roller boundary.  We also note that the finite difference partition is always nontrivial for non-hyperbolic Coxeter groups (Lemma \ref{charactrize nontrivial finite difference}). Actually, Lam-Thomas gave a precise description of the blocks of infinite reduced words in \cite{LT15} which allows us to conclude the proof of  the minimal action on $\partial_h X(W,S)$.

With the ingredients just-mentioned,  we  prove the following    result on the Myrberg limit set for the action of $W$ on   $\partial_h X(W,S)$.
\begin{thm}[Proposition \ref{homeomorphicMyrbergset}]\label{homeomorphicMyrbergset Intro}
In the setup of Theorem \ref{ThmE},  
\begin{enumerate}[label=(\roman*)]

    \item The finite difference partition restricts to a trivial relation on the Myrberg limit set in $\partial_h X(W,S)$, and the Lam-Thomas partition also does trivially on the Myrberg limit set in $\partial_T \Sigma(W,S)$.
    \item 
    There exists a canonical $W$-equivariant homeomorphism between the Myrberg limit set in $\partial_h X(W,S)$ and the one in visual boundary $\partial_\infty \Sigma(W,S)$.    
\end{enumerate}
\end{thm}

The  topological freeness on $\partial_h X(W,S)$ in Theorem \ref{ThmE} now follows from the one on $\partial_\infty \Sigma(W,S)$ by Theorem \ref{Coxeter top free on visul boundary}.  

To conclude this subsection, let us summarize the above discussion into the following diagram.
\begin{equation}
\begin{tikzcd}
  W\curvearrowright (\partial_\infty \Sigma(W,S),[\cdot]) \arrow[r, leftrightarrow, "\phi"] \arrow[d, hookleftarrow, ""]
    & (\partial_h X(W,S),[\cdot]) \arrow[d, hookleftarrow, ""]  \curvearrowleft W \\
  W\curvearrowright \partial_\infty^{Myr} \Sigma(W,S) \arrow[r, leftrightarrow, "\Psi"]
& \partial_h^{Myr} X(W,S) \curvearrowleft W
\end{tikzcd}
\end{equation} 
The one-to-one correspondence in the top line between the Lam-Thomas partition and finite difference partition  restricts to a $W$-homeomorphism  in the bottom line. All the  actions are topological free and strong boundary actions by Theorems \ref{thmA} and \ref{ThmE}.

\subsection{Applications to \texorpdfstring{$C^*$}\ -algebras}\label{SSubPureInfinite}
We now present several applications to $C^*$-algebras following the preceding boundary action results. First, using \cite[Theorem 1]{O}, one has the following result on $C^*$-selfless groups. It  has been verified in \cite{MWY1} that many groups arising from Bass-Serre theory including $C^*$-simple generalized Baumslag-Solitar groups and certain non-acylindrically hyperbolic tubular groups admit topological free strong boundary actions, so are $C^*$-selfless. We now provide more examples, which are $\cat$ and are not covered in the literature. In addition, our results above also yield new examples of exact (not necessarily nuclear) simple purely infinite $C^\ast$-algebra arising from  geometric boundary actions. We refer to \cite{A-D, L-S, J-R, M2, G-G-K-N} for foundational and recent developments on the subject.

\begin{customthm}{D}[Theorem \ref{C star selfless} and \ref{pure infiniteness from visual}]\label{thmC}
    Let $G\curvearrowright X$ be a proper isometric  action of a   {non-elementary} group $G$ on a proper  $\cat$ space $X$ with a rank-one element and the elliptic radical $E(G)$ is trivial. Suppose
    \begin{enumerate}[label=(\roman*)]
        \item either the action $G\curvearrowright X$ is cocompact  in which case $\Lambda G=\partial_\infty X$, 
        \item or the space $X$ is geodesic complete.
    \end{enumerate}
Then the following are true.
\begin{enumerate}[label=(\roman*)]
    \item The group $G$ is $C^*$-selfless, i.e., the reduced group $C^*$-algebra $C_r^*(G)$ is selfless in the sense of \cite{Ro}. 
    \item The crossed product $C^\ast$-algebra $A=C(\Lambda G)\rtimes_r G$ for the induced action $G\curvearrowright \Lambda G$ is unital simple separable and purely infinite.  
    \end{enumerate}

    \end{customthm}

This in particular includes various boundary actions of Coxeter groups, discrete subgroups of automorphism groups of buildings, and right-angled Artin groups acting on specific $\cat$ spaces.

As a sample application of Theorem \ref{thmC}, let us examine the proper and cocompact action of a finitely generated Coxeter group $(W,S)$ on the Davis complex $\Sigma(W, S)$.   
Moreover, we obtain a similar pure infinite result for the crossed product from Coxeter group actions on the horofuncion boundaries, which is known to be topologically amenable by \cite{Lec}. Therefore, the $C^*$-algebra is a Kirchberg algebra satisfying the \textit{Universal Coefficient Theorem} (UCT) and thus classifiable by the K-theory due to the aforementioned Kirchberg-Phillips classification theorem.

\begin{cor}[Corollary \ref{pure inf irreducible Coxeter}]\label{cor11}
    Let $(W,S)$ be an irreducible non-spherical non-affine Coxeter group. The following is true.
    \begin{enumerate}[label=(\roman*)]
        \item The crossed product $C^\ast$-algebra $A=C(\partial_\infty \Sigma(W, S))\rtimes_r W$ of the visual boundary action of the Davis complex $\Sigma(W, S)$ is an exact unital simple separable purely infinite $C^\ast$-algebra.    
        \item the crossed product $C^*$-algebra $B=C(\partial_h X(W, S))\rtimes_r W$  of the horofunction boundary action of the Cayley graph $X(W, S)$ is a unital Kirchberg algebra satisfying the UCT and thus classifiable by the K-theory. 
        \end{enumerate}
 \end{cor}
    
It is worth noting that Corollary \ref{pure inf building} establishes a more general result for buildings of non-spherical, non-affine type. Importantly, unlike the spherical case, boundary actions associated with non-spherical non-affine irreducible Coxeter groups or buildings need not be topologically amenable. Consequently, the corresponding $C^\ast$-algebras in Corollaries \ref{cor11} and \ref{pure inf building} are non-nuclear. This places them in an interesting position within the $C^\ast$-classification program beyond the nuclear setting, as they nevertheless exhibit key regularity properties—notably pure infiniteness and exactness. Further discussion can be found in Remark \ref{coxeter and building}.

For groups acting on $\cat$ cube complexes $X$, we can apply Theorem \ref{thmB} to the action on the Nevo-Sageev boundary $B(X)$ in \cite{N-S}, which is a $G$-invariant closed subset of the Roller boundary $\partial_R X$. 

\begin{customthm}{E}[Theorem \ref{strong boundary cube complex} and \ref{pure infinite irreducible B(X)}]\label{thmD}
    Let $X$ be a locally finite essential irreducible non-Euclidean finite dimensional $\cat$ cube complex admitting a proper cocompact action of $G\leq \aut(X)$. Suppose the elliptic radical $E(G)$ is trivial. Then $G\curvearrowright B(X)$ is a topologically free and topologically amenable strong boundary action and the $C^\ast$-algebra $A=C(B(X))\rtimes_r G$ is a unital Kirchberg $C^\ast$-algebra satisfying the UCT and thus classifiable by the K-theory. 
    \end{customthm}

\subsection*{Organization of the paper} The paper is organized as follows. Section \ref{sec:preliminary} covers the necessary preliminaries. Section \ref{sect: visual bdry} is devoted to studying the topologically free action on the visual boundary and contains the proof of Theorem \ref{thmA}. We then investigate the geometry of paraclique graphs in Section \ref{sec: roller and NS bdry}, building on the work of Ciobanu-Genevois, and prove Theorem \ref{allboundaryaresame: intro} on the homeomorphism of various combinatorial boundaries. Building on these results, Section \ref{sec: Coxeter combinatorial boundary} proceeds to  study  the action on the horofunction boundary of Coxeter groups, proving Theorems \ref{thmB} and \ref{ThmE}. The final section, \ref{sec: application}, is dedicated to $C^\ast$-algebra applications, as stated in Theorems \ref{thmC} and \ref{thmD}.

\section{Preliminaries}\label{sec:preliminary}
This preliminary section first recalls  the basics of $\cat$ geometry, including the isometry classification, visual and Tits boundaries. These are standard materials covered in \cite{B-H} which we include for completeness. We then study rank-one hyperbolic isometries via their boundary actions and their property of being contracting in general metric spaces. 
\subsection{\texorpdfstring{$\cat$}\ \   metric geometry}\label{subsec} All metric spaces in the paper are assumed to be length spaces and actually geodesic spaces. 
\begin{defn}
    Let $(X, d)$ be a metric space and $\gamma: [0, 1
    ]\to X$ be a continuous path. The length of $\gamma$ is defined to be \[\ell(\gamma)=\sup\left\{\sum_{i=0}^nd(\gamma(t_i), \gamma(t_{i+1})): 0=t_0\leq t_1\leq \dots\leq t_n\leq t_{n+1}=1, n\in \N\right \}.\]
\end{defn}
A (connected) metric space $(X, d)$ is said to be a \textit{length space} if the distance $d(x, y)$ between any two points $x, y$ equals the infimum of the length of curves joining them.
Let $\gamma :[s,t]\subseteq \mathbb R\to X$ be a path parametrized by arc-length, from the initial point $\gamma^-:=\gamma(s)$ to the terminal point $\gamma^+:=\gamma(t)$. A path $\gamma$ is called a \textit{$c$-quasi-geodesic} for $c\ge 1$ if 
$\ell(\beta)\le c \cdot d(\beta^-, \beta^+)+c$ for   any  subpath $\beta$ such that $\ell(\beta)$ is finite, i.e., $\beta$ is rectifiable. Let $x, y$ be two points on $\gamma$, we denote by $[x, y]_\gamma$ the path from $x$ to $y$ through $\gamma$. 

A path $\gamma: [0, r]\to X$ is called a \textit{geodesic segment} if $\gamma$ is an isometry from the interval $[0, r]$ to $X$. A metric space $(X, d)$ is said to be a \textit{geodesic space} if for any $x\neq y\in X$ there exists a geodesic $\gamma: [0, r]\to X$ such that $\gamma(0)=x$ and $\gamma(r)=y$. We remark that all geodesic spaces are length spaces, but may not be uniquely geodesic spaces. By abuse of language, we often denote by $[x,y]$  some choice of a geodesic between $x$ and $y$ (which usually does matter in context). 

A metric space $(X, d)$ is said to be \textit{proper} if all closed bounded sets in $X$ are compact. The following result   known as the Hopf-Rinow theorem, characterizes  proper metric spaces via the familiar metric and topological terms.
\begin{thm}\cite[Corollary 3.8]{B-H}\label{Thm: Hopt-Rinow}
    A length space $(X, d)$ is proper if and only if $X$ is metrically complete and locally compact.
\end{thm}

We denote by $\IE^n$ the $n$-dimensional Euclidean space with $n\ge 1$. Let $(X, d)$ be a geodesic space. A geodesic triangle in $X$ with vertices $(p,q,r)$,  denoted by $\Delta(p, q, r)$,  means a closed loop composed with three geodesic sides $[p, q]$, $[q, r]$ and $[r, p]$.   A \textit{comparison triangle} for $\Delta(p, q, r)$ is a triangle $\bar{\Delta}(\bar{p}, \bar{q}, \bar{r})$ in $\IE^2$ with the same side lengths as $\Delta$. A \textit{comparison point} in $\bar{\Delta}$ is a point $\bar{x}$ in $[\bar{p}, \bar{q}]$  such that $d(x, p)=d_{\IE^2}(\bar{x}, \bar{p})$, etc. We say $\Delta(p, q, r)$ satisfies the $\cat$ \textit{inequality} if for any $x, y\in \Delta(p,q,r)$ and their comparison points $\bar{x}, \bar{y}\in \Delta(\bar{p}, \bar{q}, \bar{r})$, one has
\[d(x, y)\leq d_{\IE^2}(\bar{x}, \bar{y}).\]
This   is illustrated in the following Figure \ref{fig: comparison triangle}.
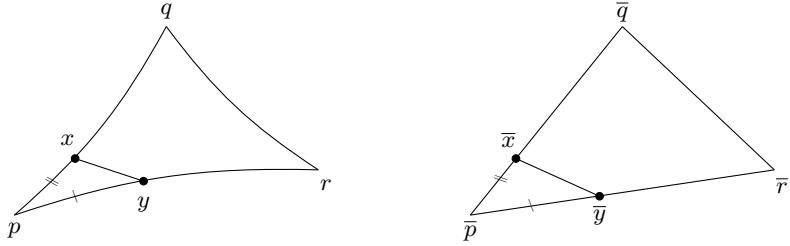
\begin{figure}[ht]
    \centering
\begin{tikzpicture}
      \node[label={below, yshift=0cm:}] at (0,-0.2) {\small$p$};
      \node[label={above, yshift=0cm:}] at (2,2.7) {\small$q$};
      \node[label={above, yshift=0cm:}] at (4.1,0.4) {\small$r$};
      \node[label={below, yshift=0cm:}] at (6,-0.2) {\small$\overline{p}$};
      \node[label={above, yshift=0cm:}] at (8,2.7) {\small$\overline{q}$};
      \node[label={above, yshift=0cm:}] at (10.1,0.4) {\small$\overline{r}$};
      \node[label={above, yshift=0cm:}] at (0.7,1) {\small$x$};
      \node[label={below, yshift=0cm:}] at (1.7,0.15) {\small$y$};
      \node[label={above, yshift=0cm:}] at (6.5,1) {\small$\overline{x}$};
      \node[label={above, yshift=0cm:}] at (7.7,0) {\small$\overline{y}$};
      \node[label={above, yshift=0cm:}] at (0.5,0.45) {\rotatebox{60}{\scalebox{0.5}{$\parallel$}}};
      \node[label={above, yshift=0cm:}] at (6.4,0.48) {\rotatebox{70}{\scalebox{0.5}{$\parallel$}}};
      \node[label={above, yshift=0cm:}] at (0.8,0.25) {\rotatebox{20}{\scalebox{0.5}{$|$}}};
      \node[label={above, yshift=0cm:}] at (6.8,0.13) {\rotatebox{20}{\scalebox{0.5}{$|$}}};

      \draw (0,0) to[bend left=10] (4,0.6);
      \draw (0,0) to[bend right=10] (2, 2.5);
      \draw (2,2.5) to[bend right=10] (4,0.6);

      \draw (6,0) -- (10,0.6);
      \draw (6,0) -- (8, 2.5);
      \draw (8,2.5) -- (10,0.6);
    
      \tikzset{enclosed/.style={draw, circle, inner sep=0pt, minimum size=.1cm, fill=black}}      
      \node[enclosed, label={right, yshift=.2cm:}] at (0.8,0.75) {};
      \node[enclosed, label={right, yshift=.2cm:}] at (1.7,0.45) {};
      \node[enclosed, label={right, yshift=.2cm:}] at (6.6,0.75) {};
      \node[enclosed, label={right, yshift=.2cm:}] at (7.7,0.25) {};

      \draw (0.8,0.75) -- (1.7,0.45);
      \draw (6.6,0.75) -- (7.7,0.25);

\end{tikzpicture}

    \caption{A comparison triangle for $\Delta(p,q,r)$}
    \label{fig: comparison triangle}
\end{figure}

A geodesic space $(X, d)$ is said to be a \textit{CAT(0) space} if any geodesic triangle $\Delta$ in $X$ satisfies the $\cat$ inequality. As $\IE^2$ is uniquely geodesic, it follows by triangle comparison that a $\cat$ space is uniquely geodesic. 

For points $x,y,z$ in a $\cat$ space $X$, we will let $\angle_x(y,z)$ denote the comparison angle at $x$ between $y$ and $z$. If $p_y : [0,a] \to X$ and $p_z : [0,b] \to X$ are the unique geodesics in $X$ from $x$ to $y$ and from $x$ to $z$ respectively, then the angle between $y$ and $z$ at $x$ is $\angle_x(y,z) = \lim_{t\to 0}\angle_x(p_y(t),p_z(t))$.

Let $\alpha,\beta$ be two geodesic rays starting at $x$. We define two notions of their angles: 
$$\begin{aligned}
\angle_x(\alpha,\beta) &:= \lim_{t\to 0} \angle_x ( \alpha(t),\beta(t)) = \inf \{\angle_x (\alpha(t),\beta(t'): t,t' > 0\}   \\
\angle(\alpha,\beta) &:= \lim_{t\to \infty} \angle_x (\alpha(t),\beta(t)) = \sup \{\angle_x(\alpha(t),\beta(t'): t,t' > 0\} 
\end{aligned}$$
which shall be used to define an angle metric on the visual boundary.

We now  classify the isometries on $\cat$ spaces. We denote by $\opisom(X)$ the group of all isometries on a metric space $(X,d)$. For any $\gamma\in \opisom(X)$, the \textit{displacement function} of $\gamma$ is the function $d_\gamma: X\to \R_+$ defined by $d_\gamma(x)=d(\gamma\cdot x, x)$. The \textit{translation length} of $\gamma$ is the number $|\gamma|=\inf\{d_\gamma(x): x\in X\}$. The following set $\operatorname{Min}(\gamma)=\{x\in X: d(\gamma\cdot x, x)=|\gamma|\}$ is a $\gamma$-invariant closed convex set by \cite[Proposition II 6.2]{B-H}.  Translation length provides the following classification of isometries on $X$. 

\begin{defn}\label{defn: class isometries}\cite[Definition 6.3]{B-H}
    Let $X$ be a metric space and $\gamma\in \opisom(X)$.
    \begin{enumerate}[label=(\roman*)]
        \item $\gamma$ is said to be \textit{elliptic} if $\gamma$ has a fix point in $X$.
        \item $\gamma$  is said to be \textit{hyperbolic} if $d_\gamma$ attains a strictly positive minimum.
        \item $\gamma$ is said to be \textit{parabolic} if $d_\gamma$ does not attain a minimum in $X$.
    \end{enumerate}
\end{defn}

\begin{rmk}\label{rmk: hyperbolic isometry property}
    Suppose $X$ is a $\cat$ space.
It is a standard fact (see, e.g., \cite[Theorem II 6.8]{B-H}) that an isometry $\gamma$ on  $X$ is hyperbolic if and only if there exists a bi-directional geodesic line $c: \R\to X$, which is translated by $\gamma$, namely, $\gamma\cdot c(t)=c(t+|\gamma|)$ for any $t\in \R$. Such a geodesic line $c$ is called an \textit{axis} of $\gamma$. Moreover, all the axes of $\gamma$ are parallel to each other in the sense of Remark \ref{rmk: CAT0 property}\eqref{parallel meaning} below and the union of them is exactly the set $\mathrm{Min}(\gamma)$, which is isometric to $K\times \mathbb R$ for some convex subset (and thus a $\cat$ space itself) $K\subset X$ by \cite[Theorem II 6.8(4)]{B-H}. In this picture, the image of an axis $c$ of $\gamma$ in $K\times \R$ is of the form $\{x\}\times \R$
\end{rmk}

Let $G$ be a group and $(X, d)$ a proper metric space, equipped with an action $\alpha$ of $G$ by isometry, i.e., there exists a group homomorphism from $G$ to $\opisom(X)$. We say the action $\alpha$ is \textit{proper} if for each $x\in X$, there exists a $r>0$ such that the set $\{g\in G: gB(x, r)\cap B(x, r)\neq \emptyset\}$ is finite, where $B(x, r)$ denotes the open $r$-ball in $X$. The action $\alpha$ is said to be \textit{cocompact} if there exists a compact set $K$ in $X$ such that $G\cdot K=X$. It is known that if $G\curvearrowright X$ cocompactly, then any isometry in $G$ is \textit{semi-simple} in the sense that the isometry is either elliptic or hyperbolic. See \cite[Proposition 6.10]{B-H}.  

\begin{rmk}\label{rmk: elliptic torsion}
    Let $G\curvearrowright X$ be a proper action. Then any elliptic isometry $g\in G$ is a torsion element. Indeed, let $o$ be a fixed point for $g$. Then by definition there exists a $r>0$ such that $\{h\in G: hB(o, r)\cap B(o, r)\neq \emptyset\}$ is finite, which implies that the stabilizer $\stab_G(o)$ has to be finite. Thus $g$ is of finite order.
\end{rmk}



\subsection{Visual and Tits boundary of \texorpdfstring{$\cat$}\ \  space}\label{SSubVisualBoundary}
Assume that $(X, d)$ is a $\cat$ space.  We begin with the definition of the visual boundary following \cite[Chapter II. 8]{B-H}.

We say two geodesic rays $c_1, c_2: [0, \infty)\to X$ are \textit{asymptotic} if there is a $C>0$ such that $d(c_1(t), c_2(t))<C$ for any $t\in [0, \infty)$. Being asymptotic is an equivalence relation for geodesic rays. Denote by $\partial_\infty X$ the set of equivalence classes, which is called the \textit{boundary set} of $X$. In addition, for any geodesic $c: [0, \infty)\to X$, we denote by $[c]$ the equivalence class containing $c$. 

\begin{rmk}\label{rmk: CAT0 property}
\begin{enumerate}
    \item 
    Let  $c_1, c_2: [0, \infty)\to X$ be two asymptotic geodesic  rays. The function $t\in [0, \infty)\mapsto d(c_1(t), c_2(t))$ is a bounded, non-negative, convex function by \cite[Proposition II 2.2]{B-H}. Thus, if they have  the same starting point (i.e. $c_1(0)=c_2(0)$), $c_1$ coincides with $c_2$.  
    \item 
    Using this fact, it is shown in \cite[Proposition II 8.2]{B-H} that for any point $z\in \partial_\infty X$ and any $x\in X$, there is a unique geodesic ray $\beta$ starting at $x$, denoted by $\beta=[x, z]$, with $[\beta]=z$. 
    \item\label{parallel meaning} 
    If $c_1, c_2: (-\infty, \infty)\to X$ are two bi-directional geodesics, then $d(c_1(t), c_2(t))\equiv C$ for some $C\ge 0$, and $c_1, c_2$ bound a flat strip by Flat Strip Theorem (see, e.g., \cite[Theorem II.2.13]{B-H}). In this case, we say that $c_1, c_2$  are \textit{parallel}. 
\end{enumerate}

\end{rmk}

Therefore, once a base point  $x_0\in X$ is chosen, $\partial_\infty X$ has one-to-one correspondence with the set of all geodesic rays starting at $x_0$. We shall endow a topology on $\partial_\infty X$ through this identification. For sake of simplicity,   we shall omit the bracket $[\alpha]$, as $[\alpha]$ contains a unique geodesic ray starting at $x_0$. 

\subsubsection*{\textbf{Cone topology}} Let $\alpha$ be a geodesic ray starting at $x_0$ and $r, \epsilon>0$. Consider the following set 
\begin{align}\label{NbhdBasis}
U({\alpha, r, \epsilon})=\{\beta\in \partial_\infty X: \beta(0)=x_0\ \text{and } d(\alpha(t), \beta(t))<\epsilon\ \text{for all }t<r\}.    
\end{align}
Note that such sets form a neighborhood basis for the geodesic $\alpha$ and generate a topology on $\partial_\infty X$, which is called the \textit{cone topology}. Finally, it was proved in \cite[Proposition II 8.8]{B-H}  that the cone topology is independent of the choice of the base point.

\begin{figure}[ht]
    \centering
\begin{center}
\begin{tikzpicture}
      \node[label={below, yshift=-0.3cm:}] at (3,2.25) {\small\}};
      \node[label={below, yshift=-0.3cm:}] at (1,2) {\small$x_0$};      \node[label={below, yshift=-0.3cm:}] at (4.6,2) {\small$\xi$};
      \node[label={below, yshift=-0.3cm:}] at (3,1.8) {\small$r$};

      \node[label={below, yshift=-0.3cm:}] at (3.3,2.25) {$\varepsilon$};

      \tikzset{enclosed/.style={draw, circle, inner sep=0pt, minimum size=.1cm, fill=black}}

      \node[enclosed, label={right, yshift=.2cm:}] at (1.4,2) {};

      \draw(1.4,2) -- (4.4,2);
      \draw(0, 2) arc(-180:180:1.4);
      \draw(1.4, 2) arc(90:30:2.8);
      \draw(1.4, 2) arc(-90:-30:2.8);
      
\end{tikzpicture}
\end{center}
    \caption{A neighborhood basis for the cone topology.}
    \label{fig: neighborhood basis}
\end{figure}
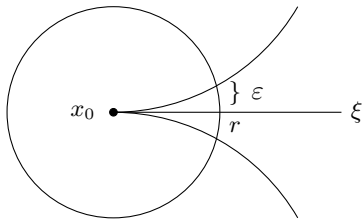

\begin{defn}\label{defn: visual bdry}
	Let $(X, d)$ be a proper $\cat$ space.  The boundary set $\partial_\infty X$ for $X$ is called the \textit{visual boundary} of $X$ when equipped with the cone topology. Denote by $\overline{X}=X\cup \partial_\infty X$, which is a compactification of $X$. 
\end{defn}

We recall basic properties on the visual boundary for a complete $\cat$ space.

\begin{rmk}\label{rem: visual bdry basic property}
\begin{enumerate}
    \item\label{visual bdry convergence} $\overline{X}$ equipped with the cone topology can also be described as the inverse limit of closed balls $\bar{B}(x, r)$ for a base point $x\in X$ and $r>0$. The space $X$ is open and dense in $\overline{X}$. Then a sequence $y_n\in X$ converges to a $z\in \partial_\infty X$ if and only if the geodesic segments $[x, y_n]$ converges to the geodesic ray $[x,z]$ uniformly on any compact sets by regarding geodesics as continuous functions on intervals. See \cite[Paragraph II. 8.5]{B-H}.

    \item If a complete $\cat$  space $X$ is not proper, the boundary $\partial_\infty X$ is not compact in general. For example, the set of ends of a locally infinite tree is not compact. However, it is well-known that if $X$ is additionally to be locally compact, then the visual boundary $\partial_\infty X$ is compact metrizable. Therefore, using Hopf-Rinow Theorem (Theorem \ref{Thm: Hopt-Rinow}), one has that for any proper $\cat$ space $X$, the visual boundary $\partial_\infty X$ is compact metrizable. Note also that the visual boundary is homeomorphic to the horofunction boundary in \textsection \ref{subsec horofunction compactfication}.

    \item If $X$ is Gromov hyperbolic, then $\partial_\infty X$ is exactly the classical Gromov boundary of $X$.\end{enumerate}
    \end{rmk}

We note the following elementary fact about the convergence in the cone topology. This shall be used in the proof of Lemma \ref{lem: characterizemyrberg}.
\begin{lem}\label{ConeConvLem}
Let $\alpha$ be a bi-directional geodesic in a $\cat$ space $X$ from $\alpha^-$ to $\alpha^+$. For some $R>0$, let $x_n,y_n$ be two unbounded sequences of points in $X$ so that $\alpha(-n), \alpha(n)$ lie in the $R$-neighborhood of $[x_n,y_n]$ for each $n\ge 1$.  Then $x_n\to \alpha^-$ and $y_n\to \alpha^+$ as $n\to\infty$.  
\end{lem}
\begin{proof}
We choose the base point $o$ on $\alpha$ (since the cone topology is independent of the basepoint). By the assumption, $d(\alpha(-n),[x_n,y_n])\le R$ and $d(\alpha(n),[x_n,y_n])\le R$, so when $d(o,x_n), d(o,y_n)\gg 0$, the convexity of $\cat$ metric shows  $d(o,[x_n,y_n])\le R$. If $z_n\in [x_n,y_n]$ is chosen so that $d(o,z_n)\le R$, then $[z_n,y_n]$ lies in the $R$-neighborhood of $[o,y_n]$ by the convexity again.  Since $d(\alpha(n),[x_n,y_n])\le R$, we deduce that  $d(\alpha(n),[o,y_n])\le 2R$. 
If $\alpha_1$ denotes the half-ray of $\alpha$ ending at $\alpha^+$, this implies $y_n\in U(\alpha_1,n,2R)$ (see Definition \ref{NbhdBasis}). As $n\to \infty$, we conclude that $y_n \to \alpha^+$. The proof for $x_n\to \alpha^-$  is symmetric. 
\end{proof}

Now, let $G\curvearrowright X$  be an action of a  group $G$ on a proper $\cat$ space $X$ by isometry. It follows from \cite[Corollary II 8.9]{B-H} that any isometry $g\in \isom(X)$ extends to  a homeomorphism $\bar{g}: \overline{X}\to \overline{X}$. First, set $\bar{g}=g$ on $X$.   Let $[c]\in \partial_\infty X$ be an equivalent class containing the geodesic $t\mapsto c(t)$. One defines $\bar{g}\cdot [c]=[g\cdot c]$, where the geodesic $g\cdot c$ is the geodesic  $t\mapsto g(c(t))$ with the starting point $g \cdot c(0)$. In this way, the isometric action $G\curvearrowright X$ thus induces a continuous action $ G\curvearrowright  \overline{X}$. Note that the visual boundary $\partial_\infty X$ is a $G$-invariant set in $\overline{X}$. 

The \textit{limit set} denoted by $\Lambda G$ of the action $G\curvearrowright X$ consists of the accumulation points  of an orbit $G\cdot x=\{gx: g\in G\}$ for some $x\in X$. By the construction of   cone topology, it is easy to verify that if $g_n x$ tends to $\xi\in \partial_\infty X$, then $g_n y\to \xi$ for any $y\in X$. Hence, the limit set $\Lambda G$ does not depend on the choice of the base point $x$.

\subsubsection*{\textbf{Tits topology}} There is  another   topology we could equip on the boundary set $\partial_\infty X$ called Tits topology, which is finer than the cone topology. This topology is necessary to understand the next Lemma \ref{lem:FixedPtsofPar} concerning the fixed points of parabolic isometry, which shall serve as a crucial ingredient in Theorem \ref{thm: topo free rank one2}.

For any $\alpha, \beta\in \partial_\infty X$, it is an exercise to show that $\angle(\alpha, \beta)$ defines a locally geodesic metric called the \textit{angle metric} on $\partial_\infty X$. The angle metric  induces a length metric on  $\partial_\infty X$ called the \textit{Tits metric}, denoted by $d_T(\cdot, \cdot)$. The \textit{Tits boundary} of $X$ is defined as $\partial_\infty X$ when with this metric and will be denoted by $\partial_T X$. Note that $d_T$ is an extended metric in the sense that it maps into $[0,\infty]$. The Tits distance between any two points in distinct path components of the angle boundary is infinity. We refer the interested reader to \cite[Chapter I. 1, II.9]{B-H} for complete details.

Recall that a geodesic metric space is said to be \textit{geodesically complete}, if any geodesic segment extends to a (possibly non-unique) bi-infinite geodesic. A smooth Hadamard manifold (i.e. simply connected and complete Riemannian manifold with non-positive sectional curvature) is geodesically complete.
\begin{lem}\cite[Theorem 1.1]{FNS06}\label{lem:FixedPtsofPar}
Let $X$ be a proper geodesically complete $\cat$ space. Let $p$ be a parabolic isometry. Then the fixed point set $\mathrm{Fix}(p):=\{\xi\in \partial_\infty X: p\xi=\xi\}$ is a subset with Tits diameter at most $\pi/2$.    
\end{lem}

\subsection{Rank-one   isometries and their dynamics on boundary}\label{SSubRankone}
We now focus on an important subclass of hyperbolic isometries called rank-one isometries, which shall play a crucial role later on. 

\begin{defn}\label{defn: rank one}
A bi-infinite geodesic $c: \R\to X$ is called \textit{rank one} if it does not bound a flat half-plane ({i.e.} the geodesic $c$ is not a boundary of  a totally geodesic embedded copy of an Euclidean half-plane in $X$). If $c'$ is a bi-infinite geodesic parallel to $c$ (i.e. $c$ and $c'$ together bound a flat strip by Remark \ref{rmk: CAT0 property}), then $c'$ is also rank-one.  A hyperbolic isometry $h$ is said to be \textit{rank-one} if $h$ has one (thus any) axis that is of rank-one.
\end{defn}

Note that each direction of the geodesic $c$, i.e., $c_1=c|_{[0, \infty)}$ and $c_2=c|_{(-\infty, 0]}$ represents different classes in $\partial_\infty X$, denoted by $c(\infty):=[c_1]$ and $c(-\infty):=[c_2]$.  It is clear that a hyperbolic isometry fixes the two  points $[c_1]$ and $[c_2]$ in the visual boundary. A key outstanding feature of a rank-one isometry $h$ is that $h$ has  only two fixed points, denoted by $h^-$ and $h^+$, outside which the subgroup $\langle h\rangle$ acts cocompactly with  north-south dynamics (see Definition \ref{NSDynamicDef}).  This property actually characterizes the rank-one ones among hyperbolic isometries.  

The direction $\Rightarrow$ which we shall use in this paper has been  obtained in earlier works \cite[Theorem A]{B-B} and \cite[Theorem 3.4]{B95}.
\begin{lem}\cite[Lemma 4.4]{Ham}\label{lem: rankone northsouth}
    Let $X$ be a proper $\cat$ space. A hyperbolic isometry $h\in \isom(X)$ is rank-one if and only if $h$ performs north-south dynamics on $\partial_\infty X$ with respect to the canonical fixed points $h^+$ and $h^-$. 
\end{lem}

\begin{rmk}\label{rmk: limit set for rank one isometry}
    In terms of limit set for $\la h\ra\curvearrowright X$,  the two fixed points $h^-, h^+$ exactly comprise its limit set by \cite[Lemma 5.1]{Ham}, which are respectively the accumulation points of $h^{-n}o$ and $h^{n}o$ as $n\to\infty$. 

\end{rmk}

A group is called \textit{elementary} if it is finite or virtually cyclic. 
In presence of a rank-one element, the induced action $G\curvearrowright \Lambda G$ of a non-elementary group $G$ has the following enjoyable dynamics.  

\begin{lem}\label{lem:rankone-bigthree}\cite[Theorem 1.1]{Ham}
Assume that a non-elementary group $G\curvearrowright X$ contains a rank-one element.  
\begin{enumerate}[label=(\roman*)]
    \item\label{rankone minimal}
    The limit set $\Lambda G$ of $G$ is the unique and minimal $G$-invariant closed subset and it is a perfect set in the sense that there is no isolated points. 
    \item 
    The  set $\{gh^\pm: g\in G\}$  of  fixed points of  elements    in the conjugacy class  of a rank-one element $h$ are dense in $\Lambda G$.  
    \item\label{rankone dense pair} 
    The fixed point  pairs of all rank-one elements are dense in $\Lambda G\times \Lambda G$.
\end{enumerate}
      
\end{lem}

\subsection{Rank-one  isometries as  contracting isometries}\label{SSubRankone2}
The notion of contracting isometries, encompassing rank-one isometries and many others,  is usually thought of as hyperbolic directions in general metric spaces and has been receiving many interests in recent years. We refer to \cite{YangGrowthtightness}, \cite{YangSCC}, \cite{YangGeneric}, and \cite{YangConformal} for more information on the following concepts.

Let $(X, d)$ be a proper geodesic space. The \textit{shortest projection} of a point $x\in X$ to a closed set $U\subset X$ is defined to be $\pi_U(x)=\{y\in U: d(x, y)=d(x, U)\}$. Given a subset $A\subset X$, we write $\pi_U(A)=\bigcup_{a
\in A}\pi_U(a)$. 

Recall that a map $f:X\to Y$ between two metric spaces $X, Y$ is said to be a \textit{quasi-isometric embedding} if there exists a $c>0$ such that 
\[c^{-1}d_X(x_1, x_2)-c\leq d_Y(f(x_1), f(x_2))\leq cd_X(x_1, x_2)+c\]
for any $x_1, x_2\in X$.
\begin{defn}\label{defn: contracting}
    A subset $U\subset X$ is said to be $C$-\textit{contracting} for  $C\ge 0$ if for any geodesic $\sigma$ satisfying $d(\sigma, U)\geq C$ one has the diameter $\diam(\pi_U(\sigma))\leq C$. 
    
    An isometry $g$ of infinite order is called \textit{contracting}  if the orbital map $n\mapsto g^n o$ (for a base point $o$) is quasi-isometric embedding and the image $\langle g\rangle o$ is a contracting subset.
\end{defn}

\begin{rmk}
Sometimes in the literature, the definition of contracting isometry does not require the orbital map $n\mapsto g^n o$ above is a quasi-isometric embedding. However, the quasi-isometric embedded image would make the $g^no$ form a quasi-geodesic in $X$ and yield a nice characterization of stabilizer group $E(g)$ of $\{g^+, g^-\}$ below. Therefore, we follow \cite{YangGeneric} to add the assumption in the definition.    
\end{rmk} 

Contracting property has the following well-known  characterization (see, e.g. \cite[Lemma 2.2]{YangConformal}).

\begin{lem}\label{char contracting property}
Let $U$ be a $C$-contracting subset. Then there exists $C'=C'(C)>0$ such that 
\begin{enumerate}[label=(\roman*)]
\item
If $d(\gamma,
U) \ge C'$ for a geodesic  $\gamma$, we have
$\mathrm
{diam}(\pi_U (\gamma))  \le C'.$
\item
If $\mathrm
{diam}(\pi_U (\gamma))  \ge  C'$ then $d(\pi_U(\gamma^-),\gamma)\le C', \;d(\pi_U(\gamma^+),\gamma)\le C'$.
\end{enumerate}
\end{lem}

Contracting subsets enjoys the following  Morse property.

\begin{lem}\label{lem:Morse}\cite[Proposition 2.2(1)]{YangGeneric}
Let $U$ be a $C$-contracting subset in $X$. Then $U$ is Morse in the following sense:  any $c$-quasi-geodesic with two endpoints in $U$ is contained in an $R$-neighborhood of $U$ for some $R$ depending only on $c$.  
\end{lem}

We recall two more facts on contracting subsets. First, it is easy exercise that the contracting property is preserved up to   a  finite Hausdorff distance.  
\begin{lem}\cite[Lemma 2.11]{YangSCC}\label{rmk: contracting set permance} 
Let $U$ be a $C$-contracting subset in $X$. Let $V$ is another subset in $X$ with bounded Hausdorff distance of $U$, i.e., $d_H(U,V)<\infty$. Then  $V$ is $D$-contracting for some constant $D$ depending on $C$ and $d_H(U,V)$. 
\end{lem} 
Second, any subpath of a contracting quasi-geodesic is again contracting in the following quantitative way.
\begin{lem}\cite[Proposition 2.2]{YangGeneric}\label{rmk: subpath is contracting} 
If $\gamma$ is a $C$-contracting $c$-quasi-geodesic, then any subpath of $\gamma$ is still contracting with a contracting constant $D$ for a constant $D=D(C, c)$.
\end{lem}

For $\cat$ spaces, it turns out that rank-one isometries are exactly contracting isometries. This is the cornerstone of our study on contracting isometries with applications given later on.
\begin{lem}\label{lem:rankone is contracting}\cite[Theorem 5.4]{BF02}
Let $X$ be a proper $\cat$ space. Then a hyperbolic isometry is rank-one if and only if any geodesic axis of it is contracting.    
\end{lem}

The minimal set $\mathrm{Min}(h)$ as defined before Definition \ref{defn: class isometries} of a rank-one isometry $h$ is quasi-isometric to a line by the following remark.
\begin{rmk}\label{rmk: rank one minimal set}
    The minimal set $\mathrm{Min}(h)$  is isometric to a metric product $K\times \mathbb R$ for some convex subset $K$  in $X$ (Remark \ref{rmk: hyperbolic isometry property}). If $h$ is rank-one, then each axis (that is, a geodesic line $\{x\}\times \mathbb R$ in $\mathrm{Min}(h)$) is rank-one, so we see that $K$ has bounded diameter and $\mathrm{Min}(h)$ is quasi-isometric to a line. 
    
    Indeed,  let $f: \mathrm{Min}(h)\to K\times \R$ be the isometry. Fix a $c\in K$ and then $\xi=f^{-1}(\{c\}\times \R)$ is a rank-one axis for $h$ and thus $C$-contracting  for some $C>0$ in the sense of Definition \ref{defn: contracting}. Let $c'\in K$ such that $c'\neq c$. Define $\eta=f^{-1}(\{c'\}\times \R)$, which is another axis of $h$ and thus parallel to $\xi$ by \cite[Theorem II 6.8(3)]{B-H}. Then Flat strip theorem (\cite[Theorem II, 2.13]{B-H}) shows that the convex hull $\mathrm{Conv}(\xi\cup \eta)$ is isometric to a Euclidean strip $[0, L]\times \R$. This implies that the projection $\pi_{\xi}(\eta(t))=\xi(t)$ for any $t\geq 0$. Then the $C$-contraction of $\xi$ implies that $\eta$ has to lie in the $C$-neighborhood of $\xi$ because the diameter of the projection $\pi_\xi(\eta(t))$ is unbounded.
\end{rmk}

We continue to list a few more standard facts for rank-one isometry. These actually hold for any contracting isometry in a general metric space. 

\subsubsection*{\textbf{Maximal elementary subgroups for rank-one isometries}} 
Recall that a group is called elementary if it is finite or virtually $\Z$. Let $h$ be a contracting isometry in a proper isometric action of a group $G$ on $X$.  We introduce the following useful group $E(h)$.
We denote by $N_R(A)$ the $R$-neighborhood of a subset $A$, i.e., $N_R(A)=\{x\in X: d(x, A)\le R\}$. 
\begin{defn}\label{Defn:E(h)}    
Define a group
    \[E(h)=\{g\in G: \text{there exits }r>0 \text{ such that }g\langle h \rangle\cdot o\subset N_r(\langle h \rangle o), \langle h\rangle o\subset N_r(g\langle h \rangle o)\}. \]
 \end{defn}

By \cite[Lemma 2.11]{YangSCC}, $E(h)$ is a maximal elementary subgroup containing $h$ in $G$.  Moreover, $E(h)$ can be characterized as the following group
$$
E(h)=\{g\in G: \exists n\in \mathbb N_{> 0},\;(gh^ng^{-1}=h^n)\; \lor\;  (gh^ng^{-1}=h^{-n})\}.
$$

If $h$ is rank-one isometry on a $\cat$ space,  $E(h)$ is exactly the set-stabilizer of the two fixed points  $\{h^-,h^+\}$ in the visual boundary $\partial_\infty X$. 
Let $E^+(h)$ be the subgroup of $E(h)$ with possibly index 2 whose elements  fix pointwise  $h^-,h^+$. Then we have 
$$
E^+(h)=\{g\in G: \exists n\in \mathbb N_{> 0}, \;gh^ng^{-1}=h^n\}.
$$

A rank-one isometry $h$ admits  a (non-uniquely in general) genuine axis (e.g. in $\mathrm{Min}(h)$) on which $h$ acts by translation. For general metric spaces, the following notion of a quasi-axis for a contracting isometry shall be particularly useful. 

\begin{defn}\label{defn: quasi-axis}
Define the \textit{quasi-axis} of $h$ to be
\[\mathrm{Ax}(h):=E(h)\cdot o\]
for a fixed base point $o\in X$.
\end{defn}

\begin{rmk}(Quasi-axis is contracting)\label{rmk: quasi-axis basic}
Let $h$ be a rank-one element. By Remark \ref{rmk: limit set for rank one isometry}, the limit set for $\{g^no: n\in \Z\}$  consists of exactly  two fixed points $h^+, h^-$, which   also belongs to the limit set of $\mathrm{Min}(h)$ by definition.  By Remark \ref{rmk: rank one minimal set}, the quasi-axis $\ax(h)$ is a quasi-geodesic with a finite Hausdorff distance to any geodesic axis in $\mathrm{Min}(h)$ (which may depend on the base point $o$). So by Lemma \ref{rmk: contracting set permance},  $\ax(h)$ is a contracting subset.
 \end{rmk}

The following roughly says that the converse of Lemma \ref{ConeConvLem} is also true when the bi-directional geodesic is an axis of a rank-one element.
\begin{lem}\label{Conv2Rank1ElemLem}
Let $h$ be a rank-one isometry on a proper $\cat$ space $X$ with a geodesic axis $\alpha$ in $\mathrm{Min}(h)$. Let $w_n,z_n\in X$ be two sequences of points so that $w_n\to h^-$ and $z_n\to h^+$.  Then the sequence of the shortest projection points of $w_n$ (resp. $z_n$) to $\alpha$ tends to $h^-$ (resp. $h^+$). Moreover, there exists some $C>0$ so that $[w_n,z_n]$ intersects $N_C(\alpha)$ unboundedly as $n\to \infty$.      
\end{lem}
\begin{proof}
We choose the base point $o\in \alpha$. Let $C$ be the contracting constant of $\alpha$.
Let $u_n$ be the shortest projection point of $w_n$ to $\alpha$. If $w_n'$ is the first entry point of $[w_n,o]$ in $N_C(\alpha)$, then $d(w_n',\alpha)= C$ and  $\pi_{\alpha}([w_n,w_n'])$ has diameter at most $C$ by the $C$-contracting property of $\alpha$. Thus, $d(u_n, [o,w_n])\le d(w_n, w_n')\le 2C$. 


The cone topology implies $u_n$ and $w_n$ lie in the same neighborhood $U(\alpha^-,L_n, 4C)$ for $L_n:=d(o,u_n)$, thus $u_n\to \alpha^-$ follows as $w_n\to \alpha^-$. The same holds for the shortest projection point $v_n$ of $z_n$ tending to $\alpha^+$.
Since the shortest projection points $u_n,v_n$ of $w_n, z_n$ have the distance tending to $\infty$,  we deduce that $[w_n,z_n]$ must intersect $N_C(\alpha)$ by the $C$-contracting property of $\alpha$. By a similar argument looking at the entry point $x_n$ and the exit point $y_n$ of $[w_n, z_n]$ in $N_C(\alpha)$ as above, we see  $d(u_n,[w_n,z_n]), d(v_n,[w_n,z_n])\le 2C$ and hence the ``Moreover" statement follows.     
\end{proof}

\subsubsection*{\textbf{Independent contracting elements}} At last, we discuss a notion of independence between contracting isometries. This relies on the following terminology  in \cite{YangGrowthtightness}. 

\begin{defn}
    We say that a family of subsets $\mathbb X$ in $X$ has \textit{bounded projection property} if  there exists $D>0$ such that the projection of any $Y\ne Z\in \mathbb X$ to $Z$ has diameter at most $D$.
The family $\IX$ is said to have \textit{bounded intersection property} if for any $R>0$ and any two different $X, X'\in \IX$, there exists $D=D(R)$ such that $\diam( N_R(X)\cap N_R(X'))\leq D$.
  \end{defn}

Let $h\in G$ be a contracting isometry on a  metric space $X$. Then $\ax(h)$ is $D$-contracting for some $D>0$ by Remark \ref{rmk: quasi-axis basic} and thus so is any isometric $G$-translate $g\ax(h)$ for some $g\in G$. The next result is well-known  (e.g. \cite[Lemma 2.3]{YangGrowthtightness}).

\begin{lem}\label{lem: Bddprojection}
If $G$ acts properly on a proper metric space $X$, then the collection $\{g\mathrm{Ax}(h): g\in G\}$ of quasi-axis of a contracting isometry $h$ has {bounded projection} and bounded intersection property.  
\end{lem}

\begin{defn}\label{defn: independence}
Let $h, k$ be two contracting isometries in a discrete group $G$ on a proper metric space $X$.  We say that $h, k$ are  \textit{independent} if $E(h)\ne E(k)$ and the system 
\[\{g\mathrm{Ax}(h),g\mathrm{Ax}(k): g\in G\}\]
has $\tau$-bounded intersection for some $\tau>0$ depending on the base point $o\in X$.
\end{defn}

\begin{lem}\cite[Lemma 2.12]{YangSCC}
If a non-elementary group $G$ acts properly on a proper metric space $X$ with one contracting isometry,  then $G$ contains infinitely many pairwise independent contracting isometries.    
\end{lem}

\subsection{Horofunction compactification}\label{subsec horofunction compactfication}
Fix a base point $o\in X$. For  each $y \in  X$, we define a Lipschitz map $b_y:  X\to \mathbb R$ by $$\forall x\in X:\quad b_y(x)=d(x, y)-d(o,y)$$ which sits in the set $C(X,o)$ of continuous functions on $ X$ vanishing at $o$. We equip $C(X,o)$ with the compact-open topology. By  Arzela-Ascoli Lemma, the closure  of $\{b_y: y\in  X\}$  gives a compact metrizable space $\overline X_h$. The complement   $\overline X_h\setminus X$ is called  the \textit{horofunction boundary} of $ X$ and is denoted by $\partial_h X$. 
\begin{rmk}
If we  equip  $C(X,o)$ with the  topology of uniform convergence over bounded sets, the closure of $\{b_y: y\in  X\}$ might not be compact and the new-added points which we denote by $\partial_h^\infty X_h$ in this remark  is also called horofunction boundary. If $X$ is proper, then these two topologies are the same, and  $ X$ is open and dense in $\overline X_h$. In general, this is not true. For example, the horofunction boundary $\partial_h X_h$ of a locally infinite tree  $X$ is the union of the space of ends and the set of vertices with infinite valence, while $\partial_h^\infty X_h$ is homeomorphic to the space of ends. Note that the visual boundary of a non-proper CAT(0)  space $X$  is homeomorphic to  $\partial_h^\infty X_h$ (not $\partial_h X_h$). See \cite[Chapter II 8.12]{B-H}. 
\end{rmk}

 
Every isometry $\phi$ of $X$ induces a homeomorphism on  $\overline X_h$:  
$$
\forall y\in X:\quad\phi(\xi)(y):=b_\xi(\phi^{-1}(y))-b_\xi(\phi^{-1}(o)).
$$
So $\isom(X)$ acts by  homeomorphism on  $\overline X_h$.
Depending on the context, we may use both $\xi$ and $b_\xi$ to denote a point in the horofunction boundary.

\subsubsection*{\textbf{Finite difference  relation on $\partial_h X$}}
Two horofunctions $b_\xi, b_\eta$ have   \textit{finite difference} if the $L_\infty$-norm of their difference is finite: $$\|b_\xi-b_\eta\|_\infty =\sup_{x\in X} |b_\xi(x)-b_\eta(x)| < \infty.$$ 
The   \textit{locus} of     $b_\xi$ consists of  horofunctions $b_\eta$ so that $b_\xi, b_\eta$ have   finite difference.  The loci   $[b_\xi]$  of    horofunctions $b_\xi$ form a \textit{finite difference  relation} $[\cdot]$ on $\partial_h X$. The \textit{locus} $[\Lambda]$ of a subset $\Lambda\subseteq \partial_h X$ is the union of loci of all points in $\Lambda$.
If $x_n\in X\to \xi\in \partial_h X$ and  $y_n\in X\to\eta\in \partial_h X$ are sequences with $\sup_{n\ge 1}d(x_n, y_n)<\infty$, then  $[\xi]=[\eta]$. 

Let $g\in \isom(X)$ be a contracting isometry. Let $\ax(g)=\cup_{n\in\mathbb Z}g^n[o,go]$ denote the contracting quasi-geodesic by definition, which we shall refer to quasi-axis of $g$. Let $[g^+]$ (resp. $[g^-]$) denote the $[\cdot]$-class of accumulation points of $\{g^no:n\ge 1\}$ (resp. $\{g^{-n}o:n\ge 1\}$) in $\partial_h X$. One may verify    that any two horofunctions in $[g^-]$ (and in $[g^+]$) have a finite difference bounded by a constant depending on the contracting constant $\ax(g)$ (\cite[Corollary 5.2]{YangConformal}). Thus,  $[g^-],[g^+]$ are  closed subsets in $\partial_h X$ which are also easily seen disjoint (\cite[Lemma 5.5]{YangConformal}). We write $[g^\pm]=[g^-]\cup[g^+]$. Note that  $[g^-],[g^+]$, which are  fixed setwise by $g$, will serve as \textit{repelling} and \textit{attracting} fixed points of $g$ in the following sense.

\begin{lem}\cite[Lemma 3.27]{YangConformal}\label{NSDynamics on Horofunction boundary}
The action of $g$ on $\partial_h X$ has north-south dynamics relative to $[g^-]$ and $[g^+]$. Namely, for any two open sets  $[g^-]\subset U$ and  $[g^+]\subset V$ containing $[g^+]$,  there exists $N>0$ so that $g^n(\partial_h X\setminus V)\subset U$ and $g^{-n}(\partial_h X\setminus U)\subset V$ for any $n>N$.  
\end{lem}
\begin{rmk}
In \cite[Lemma 11.1]{YangConformal}, it is proved  that the finite difference relation on the horofunction boundary of a proper $\cat$ space is trivial: that is, any $[\cdot]$-class is singleton. Thus, in this case, the above result recovers north-south dynamics of rank-one elements in Lemma \ref{lem:rankone-bigthree}.    
\end{rmk}

Assume that $G$ acts properly on $X$. For some $o\in X$, the \textit{limit set} denoted as $\Lambda_h(Go)$ of $Go$ 
consists of accumulation points of the orbit $Go$ in $\partial_h X$. Note that $\Lambda_h(Go)$ depends on the choice of the base point $o$, but the $[\cdot]$-closure will not: $[\Lambda_h(Go)]=[\Lambda_h(Go')]$ for any two $o,o'\in X$ (by \cite[Lemma 2.29]{YangConformal}).

The following lemma will be used in the proof of Lemma \ref{MinimalBdryActionCoxeter}.
\begin{lem}\cite[Lemma 3.11]{YangConformal}\label{unique limitset on Horofunction boundary}
Assume that $G$ acts properly on $X$. Let $h\in G$ be a contracting element. Take any point $\xi\in \Lambda_h(E(h)o)$. Then $\overline{G\xi}\subset \Lambda_h(Go)\subset [\overline{G\xi}]$. In particular, $[\overline{G\xi}]=[\Lambda_h(Go)]$. 
\end{lem}

\subsection{Topological dynamical systems}
Throughout this subsection, we adopt the following notation: $G$ denotes a countable discrete group, $Z$ a compact Hausdorff space, and $G \curvearrowright Z$ a continuous action of $G$ on $Z$ by homeomorphisms.

\begin{defn}
We say that an action $G\curvearrowright Z$ has the following property: 
    \begin{enumerate}
        \item 
        \textit{minimal} if all orbits are dense in $Z$;
        \item 
        \textit{topologically free} if the set $\{z\in Z: \stab_G(z)=\{e\}\}$, is dense in $Z$. This is equivalent to that the fixed point set $\{z\in Z: tz=z\}$ of each nontrivial element $t$ of $G$ is nowhere dense. 
    \end{enumerate}
\end{defn}

\begin{rmk}\label{rmk: minimal imply perfect}
Let  $G\curvearrowright Z$ be a minimal action. If $Z$ is an infinite set,  then $Z$ is a perfect set. Indeed, if not, there exists $z\in Z$ such that $\{z\}$ is open and thus $G\cdot z$ is an open set.  The minimality  implies that $Z=G\cdot z$, and then $Z$ is finite by compactness. This is a contradiction.

Moreover,  the topological freeness of a minimal action is equivalent to the existence of one free point $z\in Z$ for the action by looking at the orbit $G\cdot z$.
\end{rmk}

\begin{defn}\label{topo amenable action}
    We say the action $G\curvearrowright Z$ is \textit{topological amenable} if there exists a sequence of continuous maps $\mu_n: Z\to \operatorname{Prob}(G)$ such that for any $g\in G$, one has $\sup_{z\in Z}\|\mu_n(g\cdot z)-g\cdot\mu_n(z)\|_1\to 0$ as $n\to \infty$, where $\operatorname{Prob}(G)$ denotes the space of all probability measures on $G$ and $(g\cdot \mu_n(z))(h)=\mu_n(z)(g^{-1}h)$. 
    We say that $G$ is \textit{boundary amenable} if $G$ admits a topological amenable action $G\curvearrowright Z$ on a compact Hausdorff space $Z$. 
\end{defn}

\begin{rmk}\label{subspace topological amenable}
   Suppose $G\curvearrowright Z$ is topological amenable. Let $Y\subset Z$ be a closed $G$-invariant subset. Then it is straightforward to see that the restriction of all $\mu_n$ to $Y$ yielding that the restricted action $G\curvearrowright Y$ is also topological amenable.   
\end{rmk}

A type of topological dynamical system of particular interest is the so-called $G$-boundary actions in the sense of Furstenburg. Now, denote by $P(Z)$ the set of all probability measures on $Z$. Furstenberg introduced the following definition in \cite{F}.
\begin{defn}\label{defn: boundary}
	\begin{enumerate}
		\item A $G$-action  on $Z$ is called \textit{strongly proximal} if for any probability measure $\eta\in P(Z)$, the closure of the orbit $\{g\eta: g\in G\}$ contains a Dirac mass $\delta_z$ for some $z\in Z$ 
		\item A $G$-action on a compact Hausdorff space $Z$ is called a $G$-\textit{boundary action} if $\alpha$ is minimal and strongly proximal.
	\end{enumerate}
\end{defn}

We recall the following definition appeared in \cite{L-S}.  See also \cite{Gla2}.

\begin{defn}\cite[Definition 1]{L-S}\label{defn: strong boundary}
We say an action $G\curvearrowright Z$ is a \textit{strong boundary action} (or \textit{extreme proximal}) if for any compact set $F\neq Z$ and non-empty open set $O$ there is a $g\in G$ such that $gF\subset O$.  
\end{defn}

\begin{rmk}\label{rmk: strong bdry is bdry}
If $G\curvearrowright Z$ is a strong boundary action and  $Z$ contains more than two points,  Glasner \cite{Gla1} showed that $Z$ is a $G$-boundary in the sense of Definition \ref{defn: boundary}.
\end{rmk}

\begin{defn}\label{NSDynamicDef}
	Let $g$ be a homeomorphism of $Z$. We say $g$ has \textit{north-south dynamics} with respect to two fixed points $x, y\in Z$ if for any open neighborhoods $U$ of $x$ and $V$ of $y$, there is an $m\in \N$ and such that $g^m(Z\setminus V)\subset U$ and $g^{-m}(Z\setminus U)\subset V$. The points $x,y$ shall be referred to as  \textit{attracting} and \textit{repelling} fixed points of $g$, respectively.
\end{defn}

A similar argument for the following proposition also appeared in, e.g., \cite{A-D} and \cite{L-S}. 
\begin{prop}\label{prop: north-south visual}
	Let $G\curvearrowright Z$ be a minimal action. Suppose there is a $g\in G$  performing the north-south dynamics. Then the action is a strong boundary action and thus a $G$-boundary action by Remark \ref{rmk: strong bdry is bdry}.
\end{prop}
\begin{proof}
	Let $F$ be a proper compact set in $Z$ and $O$ a non-empty open set in $Z$.
 Write $O_1=Z\setminus F$, and $O_2=O$, which are non-empty open sets in $Z$. Suppose $x, y$ are attracting and repelling fixed points of $g$, respectively. First, by minimality of the action, one can find two open neighborhoods $U, V$ of $x, y$, respectively, small enough such that  there are $g_1, g_2\in G$ such that $g_1V\subset O_1$ and $g_2U\subset O_2$.
	Now our assumption on $g$ implies that there is an $m\in \N$ such that $g^m(Z\setminus V)\subset U$, which implies $g_2g^m(Z\setminus V)\subset O_2$. Then one observes that $Z=(g_2g^m)^{-1}O_2\cup g_1^{-1}O_1$. write $h=g_2g^mg_1^{-1}$ for simplicity. This shows that $hF=h(Z\setminus O_1)\subset O$. Therefore $\alpha$ is a strong boundary action and thus is a $G$-boundary action.
\end{proof}

\section{Actions on visual boundaries of \texorpdfstring{$\cat$}\ \  spaces}\label{sect: visual bdry}
In this section, we study the action $ G\curvearrowright X$ of a discrete non-elementary group $G$ on a proper $\cat$ space $X$ and  establish Theorem \ref{topo free strong boundary on visual summary}.
First of all, we will prove that such action on the limit set is a strong boundary action  by Proposition \ref{prop: pure inf on limit set}. Then, the topological freeness in  Theorem \ref{topo free strong boundary on visual summary} is the main difficulty, which shall be achieved by verifying that Myrberg points are usually free points for the actions on the visual boundary.

\subsection{Strong boundary actions on visual boundaries}
We start by proving that the visual boundary action is a strong boundary action. Recall that the limit set $\Lambda G$ of $G$ consists of all accumulation points of some (or any)  $G$-orbit $Go$ with $o\in X$ in the visual boundary $\partial_\infty X$.
 
\begin{prop}\label{prop: pure inf on limit set}
    Let $G\curvearrowright X$ be an isometric action of a non-elementary discrete group $G$ on a proper $\cat$ space $X$ with a rank-one element. Then the restricted action $G\curvearrowright \Lambda G$ is a strong boundary action and thus a $G$-boundary action.
\end{prop}
\begin{proof}
    First Lemma \ref{lem:rankone-bigthree}\ref{rankone minimal} shows that $G\curvearrowright \Lambda G$ is minimal.   Then by Proposition \ref{prop: north-south visual}, it suffices to show there exists a $g\in G$ performing north-south dynamics on $\Lambda G$. To show this, let $g\in G\leq \isom(X)$ be a hyperbolic rank-one isometry with an axis $\gamma$. Remark \ref{rmk: limit set for rank one isometry}  implies that the fixed points $g^+=\gamma(\infty)$ and $g^-=\gamma(-\infty)\in \Lambda G$. Let $U, V$ be open neighborhood of $g^+$ and $g^-$ in $\Lambda G$. Choose open sets $U'$ and $V'$ in $\partial_\infty X$ such that $U'\cap \Lambda G=U$ and
 $V'\cap \Lambda G=V$. This implies that $\Lambda G\setminus U\subset \partial_\infty X\setminus U'$ and $ \Lambda G\setminus V\subset \partial_\infty X\setminus V'$. Then Lemma \ref{lem: rankone northsouth} and the fact that $\Lambda G$ is $G$-invariant entail that there exists a $N\in \N$ such that for any $n>N$ one has 
 \[g^n(\Lambda G\setminus U)\subset g^n(\partial_\infty X\setminus U')\subset V'\cap \Lambda G=V\]
  and
  \[g^{-n}(\Lambda G\setminus V)\subset g^{-n}(\partial_\infty X\setminus V')\subset U'\cap \Lambda G=U.\]
  This shows that $g$ performs the north-south dynamics on $\Lambda G$.  Now, Proposition \ref{prop: north-south visual} implies that the restricted action $G\curvearrowright \Lambda G$ is a strong boundary action and thus a $G$-boundary action.
  \end{proof}

\subsection{Preliminary on Myrberg points}
According to the definition, the topological free action on the limit set consists in finding a dense subset of free points (i.e. with trivial stabilizer). In our proof given in next subsection, such points are coming from a class of so-called Myrberg points defined as follows.

\begin{defn}\label{defn:Myrberg}
A limit point $z\in \Lambda G$ is called \textit{Myrberg point} if for any point $w\in X$, the orbit $G(z,w)=\{(gz,gw): g\in G\}$ is dense in the set $\Lambda G\times \Lambda G$ of pairs of points in the following sense:
\begin{itemize}
    \item for any $(a,b)\in \Lambda G\times \Lambda G$ there exists $g_n\in G$ so that $g_nz\to a$ and $g_nw\to b$ in the cone topology.
\end{itemize}
\end{defn}

\begin{rmk}\label{Myrberg translate permanace}
    It is direct to see that a translate of any Myrberg point by a group element is still a Myrberg point. Indeed, suppose $z$ is a Myrberg point and $g\in G$. Let $(a, b)\in \Lambda G\times \Lambda G$. Then for $w\in X$, since $z$ is Myrberg, there exists $h_n\in G$ such that $h_nz\to a$ and $h_n\cdot g^{-1}w\to b$. This implies that $h_ng^{-1}\cdot gz\to a$ and $h_ng^{-1}\cdot w\to b$. Therefore, $gz$ is still a Myrberg point.
\end{rmk}

We shall give a characterization of Myrberg points using the geometry inside. This is based on the following technical notion.

\begin{defn}\label{defn: recurrence with arbitrary accuracy}
Let $h$ be a rank-one isometry with quasi-axis $\mathrm{Ax}(h)$. We say that a geodesic ray $\gamma$ is \textit{recurrent to $h$ with arbitrary accuracy}, if there exists a constant $R$ depending only on  $\mathrm{Ax}(h)$ with the following property. For any large $L\ge 1$, the geodesic ray $\gamma$ contains a segment of length $L$ in  $N_R(g\mathrm{Ax}(h))$ for some $g\in G$.    
\end{defn}

\begin{rmk}\label{rmk: strengthened arbitrary accuracy}
Here are a few remarks in order:
\begin{enumerate}
    \item One may replace the existence of a segment of length $L$ in $N_R(g\mathrm{Ax}(h))$ with the following seemingly weaker property: 
$N_R(g\mathrm{Ax}(h))\cap \gamma$ has diameter at least $L$. 

Indeed, let $x,y\in N_R(g\mathrm{Ax}(h))\cap \gamma$ be the entry and exit points so that $d(x,y)>L$.  Lemma \ref{lem:Morse} implies that there exists  $R'>0$ depending only on the contracting constant of $g\ax(h)$ and $R$ so that $[x,y]$ is contained in $N_{R'}(g\mathrm{Ax}(h))$. That is, we are able to find a segment of length at least $L$ in $N_{R'}(g\mathrm{Ax}(h))$. 
    \item 
    One may equally formulate the  definition using any geodesic axis $\alpha$ in $\mathrm{Min}(h)$ instead of the  quasi-axis $\mathrm{Ax}(h)=E(h)\cdot o$. If some $R$ works in Definition \ref{defn: recurrence with arbitrary accuracy}, so does any larger  $R$. As $\mathrm{Ax}(h)$ has finite Hausdorff distance to $\alpha$, we could replace $\mathrm{Ax}(h)$ with $\alpha$.
    \item
    This terminology is perhaps better justified  in terms of geodesic flow. Let $\pi: X\to X/G$ be the natural projection. Then the projected geodesic ray $\pi(\gamma)$ on  $X/G$ will enter into  the $R$-neighborhood of the axis $\pi(\ax(h))$ and wrap around it with an arbitrarily long time $L$.   
\end{enumerate}
\end{rmk}

\begin{lem}\label{lem: characterizemyrberg}
A point $z\in \Lambda G$ is a Myrberg point if and only if for some (or any) $w\in X$, the geodesic ray $[w,z]$ is recurrent to any rank-one isometry $h$ in $G$ with arbitrary accuracy. 
\end{lem}
\begin{rmk}\label{rmk: sub-class rank one for Myrberg}
In some application, we could restrict to a subclass $\mathcal C$ of rank-one isometries in $G$ with the following property: the pairs of fixed points of all $h\in\mathcal C$ are dense in $\Lambda G\times \Lambda G$. This is the  property used in the proof below.     
\end{rmk}
\begin{proof}
This is essentially proved in \cite[Lemma 4.12]{YangConformal} in the setting of  general metric spaces with contracting elements. Here we provide a proof to the interested reader, using the more specific facts due to the $\cat$ geometry.

(1). The $\Leftarrow$ direction. By Lemma \ref{lem:rankone-bigthree}\ref{rankone dense pair},  the pairs of fixed points of rank-one elements are dense in $\Lambda G\times \Lambda G$. To prove that $z$ is a Myrberg point, it suffices to find a sequence of elements $g_n\in G$ with  $g_n\cdot(w,z)\to (h^-,h^+)$ for any rank-one element $h$. By assumption, the geodesic ray $\gamma:=[w,z]$ is recurrent to $h$ with arbitrary accuracy: there exist a constant $R>0$ depending on the contracting constant of $\ax(h)$ and a sequence of elements $g_n\in G$ such that  $g_n\gamma$ contains a segment $\beta_n=[x_n,y_n]$ in $N_{R}(\mathrm{Ax}(h))$  with length $L_n\to\infty$. Since $E(h)$ contains $\langle h\rangle$ as a finite index subgroup, $\langle h\rangle$ acts co-compactly on  $\mathrm{Ax}(h)= E(h)\cdot o$. If  $o_n$ denotes the middle point on $\beta_n$, then  up to increasing $R$, there exists an $h_n\in \langle h\rangle$ such that $d(o_n, h_no)<R$. Substituting $g_n$ by $h^{-1}_ng_n$, one may assume the middle point $o_n$ of $\beta_n$ satisfies $d(o_n, o)<R$. Since $h$ fixes $h^-,h^+$, it remains to show that the two endpoints $g_nw, g_nz$ of $g_n\gamma$ tend to $h^-$ and $h^+$ respectively.

Let $\alpha$ be a geodesic axis on which the rank-one element $h$ acts by translation. By Lemma \ref{lem: rankone northsouth}, we know that the two endpoints of $\alpha$ are exactly the fixed points of $h$. As $\langle h\rangle $ also acts co-compactly on $\mathrm{Ax}(h)= E(h)\cdot o$, $\alpha$ has a bounded Hausdorff distance to $\mathrm{Ax}(h)$ which may depend on $o$ but not on the sequence of $g_n$. Up to raising $R$ by a finite amount depending on $o$, we may assume that $\beta_n$ lies in $N_R(\alpha)$. Recall that $\beta_n$ forms a sequence of segments of $g_n\gamma$ with length $L_n\to \infty$ and $\beta_n$ intersects a fixed ball around $o$ as $d(o, o_n)<R$. We deduce by Lemma \ref{ConeConvLem} that the two endpoints of $g_n\gamma$ converge to $h^-, h^+$ respectively:
$g_n w\to h^-$ and $g_nz\to h^+$. This proves the  $\Leftarrow$ direction. 

(2). The  $\Rightarrow$ direction follows from Lemma \ref{Conv2Rank1ElemLem}. Indeed, let $\gamma$ be a geodesic ray starting at $w$ and ending at  a Myrberg point  $z\in \Lambda G$. Let  $h$ be a rank-one element with a geodesic axis $\alpha$. By definition of Myrberg points, there exists a sequence of $g_n\in G$ such that $g_n\cdot (w, z)\to (h^-, h^+)$. As the visual compactification $\overline{X}$ with cone topology is metrizable, we may choose $z_n\in \gamma$ with $z_n\to z$ so that   $g_n\cdot (w, z_n)\to (h^-, h^+)$. If  $C$ is given by Lemma \ref{Conv2Rank1ElemLem} depending on the contracting constant of $\alpha$, $[g_nw,g_nz_n]$ (and thus $g_n\gamma$ containing it)  intersects $N_C(\alpha)$ unboundedly as $n\to \infty$. Hence, $\gamma$ is recurrent to the axis of $h$ with arbitrary accuracy. As $h$ is arbitrary rank-one element, the proof of the $\Rightarrow$ direction is complete.
\end{proof} 
\begin{lem}\label{MyrbergVisual}
Let $\alpha=[o,\xi]$ be a geodesic ray ending at a boundary point $\xi\in \partial_\infty X$. Assume that $p_n$ is a sequence of $C$-contracting segments for some $C>0$ so that  $N_C(p_n)\cap \alpha$ leaves every compact set and has diameter at least $8C$. Then   there exists a geodesic from   $\xi$ to every $\eta\in \partial_\infty X\setminus \xi$.    
\end{lem}
\begin{proof}
In \cite[Lemma 4.11]{YangConformal}, a proof is given when $z$ is assumed to be a conical point. We explain how the proof works for any boundary point. 

Let $u_n,v_n\in \alpha$ the entry and exit point of $\alpha$ in $N_C(p_n)$. By assumption, $d(u_n,v_n)\ge 8C$. If $u_n', v_n'\in p_n$ denote the projection points of $u_n, v_n$ to $p_n$ respectively, then $d(u_n,u_n'), d(v_n,v_n')\le C$ and thus $d(u_n',v_n')>d(u_n,v_n)-2C\ge 6C$. 

Let $\beta=[o,\eta]$ be a geodesic ray  from $o$ to $\eta$. Observe   that $\beta$ only intersects finitely many $N_C(p_n)$. If not, 
since $N_C(p_n)\cap \alpha$  escapes to infinity, the intersection $\beta\cap N_C(p_n)$ do so. That is, $\beta$ intersects  $N_C(\alpha)$ in an unbounded set of points. This implies that $\alpha=\beta$, contradicting $\xi\ne\eta$. Up to passing a subsequence of $p_n$, assume that $\beta\cap N_C(p_n)=\emptyset$ for any $n\ge 1$.

Let us fix one $p:=p_n$ with $\mathrm{diam}(N_C(p)\cap \alpha)>6C$.   Denote $u:=u_n, v:=v_n$ and $u':=u_n', v':=v_n'$ accordingly. Take two sequences of points $x_m\in \alpha$  and   $y_m\in \beta$ so that $x_m\to \xi$ and $y_m\to \eta$. We shall prove that  every $[x_m,y_m]$ intersects the $C$-neighborhood of $p$. 

By way of contradiction,  assume that  $[x_m,y_m]$  are disjoint with $N_C(p)$ for each $m\ge 1$, up to passing to subsequence. The $C$-contracting property of $p_n$ shows that the projection of $[x_m,y_m]$ to $p$ has  diameter at most $C$: $\mathrm{diam}(\pi_{p}([x_m,y_m]))\le C$. Similarly, as $[o,u]$ and $[v,x_m]$ are disjoint with $N_C(p)$, we have $\mathrm{diam}(\pi_{p}([v,x_m]))\le C$ and $\mathrm{diam}(\pi_{p}([o,u_n]))\le C$. As in the first paragraph,  $\beta\cap N_C(p_n)=\emptyset$ and thus $\mathrm{diam}(\pi_p([o,y_m]))\le C$.  Recalling that $u',v'\in p$ are projection points of $u,v$, we bound   from above their distance via previous projections : $$
\begin{aligned}
d(u',v') &\le \mathrm{diam}(\pi_p([v,x_m]))+\mathrm{diam}(\pi_p([x_m,y_m]))+\\
&\;\; +\mathrm{diam}(\pi_p([o,y_m]))+\mathrm{diam}(\pi_p([o,u])) \\
&\le 4C    
\end{aligned}$$  This contradicts $d(u',v')>6C$ in the first paragraph, so   $[x_n,y_n]\cap N_C(p)\ne\emptyset$ is proved.

Now, as $p$ is finite segment and $X$ is a proper metric space, a Cantor diagonal argument via Arzela-Ascoli Lemma extracts a subsequence of  $[x_m,y_m]$ that converges locally uniformly to a bi-infinite geodesic $\gamma$. By $\cat$ geometry, the two rays of $\gamma$ are terminating  at $\xi$ and $\eta$ respectively. The proof is complete. 
\end{proof}

The second statement of the next result uses crucially Lemma \ref{lem:FixedPtsofPar}. 
\begin{lem}\label{lem: myrbergpoints}
Let $G\curvearrowright X$ be a proper   isometric action of a non-elementary group $G$ on a proper $\cat$ space $X$. Then 
\begin{enumerate}[label=(\roman*)]
    \item\label{rankone fix no Mybrberg} The fixed points of a rank-one isometry  are not   Myrberg points. 
     
    \item Assume, in addition, that $X$ is geodesically complete. If a Myrberg point is fixed by a parabolic isometry $p$, then the fixed point set $\mathrm{Fix}(p):=\{z\in \partial_\infty X: pz=z\}$   is singleton. 
\end{enumerate}  
\end{lem}
\begin{proof}
\textbf{(1)}. This is proved in \cite[Lemma 4.15]{YangConformal}. Roughly speaking, since $G$ is non-elementary, any rank-one isometry is independent with  another rank-1 element, their axis have bounded intersection. However,  this gives a contradiction since a Myrberg geodesic ray is recurrent to any rank-one axis with arbitrary accuracy.

\textbf{(2)}. By Lemma \ref{lem:FixedPtsofPar}, $\mathrm{Fix}(p)$ is a subset with Tits diameter at most $\pi/2$. On the other hand, any boundary point $z$ is visible from a Myrberg point $w$: there exists a bi-infinite geodesic between $z$ and $w$ (Lemma \ref{MyrbergVisual}). This implies the angle metric $\angle (z,w)$ from $z$ to $w$ is $\pi$, so the  Tits distance $d_T(z,w)$ (induced by the angle metric) is at least $\pi$. Thus, if $\mathrm{Fix}(p)$ contains a Myrberg point, then $\mathrm{Fix}(p)$ consists of only one point.
\end{proof}

Let $G\curvearrowright X$ be a proper   isometric action of a non-elementary group $G$ on a proper $\cat$ space $X$ with rank-one elements. 
\begin{defn}\label{ERdefn}
The \textit{elliptic radical} of the action is defined to be the subgroup of elements fixing the limit set pointwise: $$E(G)=\{g\in G: gz=z,\; \forall z\in \Lambda G\}$$    
\end{defn} 
\begin{rmk}
Equivalently, by \cite[Proposition 7.14]{WXY}, it is characterized as the following two intersections  
$$
E(G)=\cap_{h\in \mathcal R} E(h) = \cap_{h\in \mathcal R} E^+(h)
$$
where $\mathcal R$ is the set of all rank-one elements in $ G$ and $E^+(h)$ is the subgroup fixing pointwise the two fixed points of $h$. Indeed, this follows from the fact that the fixed points of rank-one isometries are dense in $\Lambda G$ (Lemma \ref{lem:rankone-bigthree}). It can be also algebraically characterized as the unique maximal finite normal subgroup of $G$, i.e., \textit{the finite radical} of $G$ (see \cite{DGO}).  In other words, $E(G)$ is a subgroup of $G$ independent of the (proper) action $G\curvearrowright X$ we start with. We refer to \cite[Section 7.2]{WXY} for more details and relevant discussions.    
\end{rmk}

The last ingredient we need is a class of \textit{special} rank-one elements in the following sense. The result also holds in a general proper action, with a more involved proof given in Lemma \ref{lem: specialrank1_general}.
\begin{lem}\label{lem: specialrank1}
Let $G\curvearrowright X$ be a proper co-compact  isometric action on a proper $\cat$ space $X$ with a rank-one element. Assume that the elliptic radical of the action is trivial. Then there exists a rank-one isometry $h\in G$ so that the maximal elementary subgroup $E(h)=\langle h\rangle$   is an infinite  cyclic group.    
\end{lem}
\begin{proof}
The case of  co-compact actions follows as a combination of various well-known results. Recall that non-elementary group acting properly on a proper CAT(0) space with rank-one elements is acylindrical hyperbolic (\cite{Sisto2}). Indeed, it is proved in \cite[Corollary 5.7]{Hull}  that if $E(G)$ is trivial, then any acylindrical action of $G$ on a hyperbolic space contains a loxodromic element $h$ with the maximal elementary group $E(h)=\langle h\rangle$. By the main result of \cite{Sisto}, $h$ is a Morse element in the Cayley graph of $G$; that is to say, any $c$-quasi-geodesic with endpoints in $\langle h\rangle$ is contained in a $R$-neighborhood of $\langle h\rangle$ for some $R$ depending only on $c$. If the action on $X$ is co-compact, Morse elements are exactly rank-1 elements by \cite{CS1}, so  $h$  is a rank-one element as we wanted.  
\end{proof}

\subsection{Uncountable Myrberg points}
In this subsection, we shall prove that there are uncountable many Myrberg points in $\Lambda G$, provided that the limit set  $\Lambda G$ is uncountable (equivalently,  $|\Lambda G|\ge 3$). In \cite{YangConformal}, if the action of $G$ on $X$ is of divergent type, then the Patterson-Sullivan measures are supported on the Myrberg limit set. Consequently, in this case, Myrberg points are uncountable. Thus, our goal is to prove this fact only assuming the action is proper.   

To that end, more ingredients are needed from the metric  geometry of contracting elements. The following concept was introduced in \cite{YangGrowthtightness}.

\begin{defn}[Admissible Path]\label{AdmDef} Given $L,\tau\geq0$ and a family $\CF$ of $C$-contracting sets for some $C>0$, a path $\gamma$ is called $(L,\tau)$-\textit{admissible} in $\CF$, if $\gamma$ is a concatenation of geodesics $p_0q_1p_1\cdots q_np_n$ $(n\in\mathbb{N})$, where the two endpoints of each $p_i$ lie in some $X_i\in \mathcal F$, and   the following   \textit{Long Local} and \textit{Bounded Projection} properties hold:
\begin{enumerate}
\item[(LL)] Each $p_i$  for $1\le i< n$ has length bigger than $L$, and  $p_0,p_n$ could be trivial;
\item[(BP)] For each $X_i$, we have $X_i\ne X_{i+1}$ and $\max\{\diam(\pi_{X_i}(q_i)),\diam(\pi_{X_i}(q_{i+1}))\}\leq\tau$, where $q_0:=\gamma_-$ and $q_{n+1}:=\gamma_+$ by convention.
\end{enumerate} 
The subcollection $\CF(\gamma)=\{X_i: 1\le i\le n\}\subset \CF$ is referred to as contracting subsets associated with the admissible path.
\end{defn} 

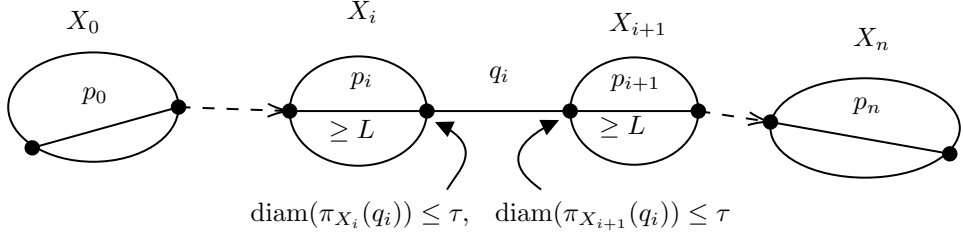
\begin{figure}
    \centering

\tikzset{every picture/.style={line width=0.75pt}} 

\begin{tikzpicture}[x=0.75pt,y=0.75pt,yscale=-1,xscale=1]

\draw   (80.5,98.5) .. controls (80.5,83.31) and (99.53,71) .. (123,71) .. controls (146.47,71) and (165.5,83.31) .. (165.5,98.5) .. controls (165.5,113.69) and (146.47,126) .. (123,126) .. controls (99.53,126) and (80.5,113.69) .. (80.5,98.5) -- cycle ;
\draw   (221,100.5) .. controls (221,85.31) and (236.33,73) .. (255.25,73) .. controls (274.17,73) and (289.5,85.31) .. (289.5,100.5) .. controls (289.5,115.69) and (274.17,128) .. (255.25,128) .. controls (236.33,128) and (221,115.69) .. (221,100.5) -- cycle ;
\draw   (461,109) .. controls (461,95.19) and (482.15,84) .. (508.25,84) .. controls (534.35,84) and (555.5,95.19) .. (555.5,109) .. controls (555.5,122.81) and (534.35,134) .. (508.25,134) .. controls (482.15,134) and (461,122.81) .. (461,109) -- cycle ;
\draw   (361,100.5) .. controls (361,85.59) and (375.33,73.5) .. (393,73.5) .. controls (410.67,73.5) and (425,85.59) .. (425,100.5) .. controls (425,115.41) and (410.67,127.5) .. (393,127.5) .. controls (375.33,127.5) and (361,115.41) .. (361,100.5) -- cycle ;
\draw  [dash pattern={on 4.5pt off 4.5pt}]  (165.5,98.5) -- (219,100.43) ;
\draw [shift={(221,100.5)}, rotate = 182.06] [color={rgb, 255:red, 0; green, 0; blue, 0 }  ][line width=0.75]    (10.93,-3.29) .. controls (6.95,-1.4) and (3.31,-0.3) .. (0,0) .. controls (3.31,0.3) and (6.95,1.4) .. (10.93,3.29)   ;
\draw  [dash pattern={on 4.5pt off 4.5pt}]  (425,100.5) -- (459.02,105.7) ;
\draw [shift={(461,106)}, rotate = 188.69] [color={rgb, 255:red, 0; green, 0; blue, 0 }  ][line width=0.75]    (10.93,-3.29) .. controls (6.95,-1.4) and (3.31,-0.3) .. (0,0) .. controls (3.31,0.3) and (6.95,1.4) .. (10.93,3.29)   ;
\draw    (289.5,100.5) -- (361,100.5) ;
\draw [shift={(361,100.5)}, rotate = 0] [color={rgb, 255:red, 0; green, 0; blue, 0 }  ][fill={rgb, 255:red, 0; green, 0; blue, 0 }  ][line width=0.75]      (0, 0) circle [x radius= 3.35, y radius= 3.35]   ;
\draw [shift={(289.5,100.5)}, rotate = 0] [color={rgb, 255:red, 0; green, 0; blue, 0 }  ][fill={rgb, 255:red, 0; green, 0; blue, 0 }  ][line width=0.75]      (0, 0) circle [x radius= 3.35, y radius= 3.35]   ;
\draw    (299.5,140) .. controls (305.29,130.35) and (319,118.84) .. (296.13,107.26) ;
\draw [shift={(293.5,106)}, rotate = 24.36] [fill={rgb, 255:red, 0; green, 0; blue, 0 }  ][line width=0.08]  [draw opacity=0] (8.93,-4.29) -- (0,0) -- (8.93,4.29) -- cycle    ;
\draw    (346.5,140) .. controls (338.7,134.15) and (322.34,121.65) .. (353.04,106.19) ;
\draw [shift={(355.5,105)}, rotate = 154.8] [fill={rgb, 255:red, 0; green, 0; blue, 0 }  ][line width=0.08]  [draw opacity=0] (8.93,-4.29) -- (0,0) -- (8.93,4.29) -- cycle    ;
\draw    (92.5,119) -- (165.5,98.5) ;
\draw [shift={(165.5,98.5)}, rotate = 344.31] [color={rgb, 255:red, 0; green, 0; blue, 0 }  ][fill={rgb, 255:red, 0; green, 0; blue, 0 }  ][line width=0.75]      (0, 0) circle [x radius= 3.35, y radius= 3.35]   ;
\draw [shift={(92.5,119)}, rotate = 344.31] [color={rgb, 255:red, 0; green, 0; blue, 0 }  ][fill={rgb, 255:red, 0; green, 0; blue, 0 }  ][line width=0.75]      (0, 0) circle [x radius= 3.35, y radius= 3.35]   ;
\draw    (221,100.5) -- (289.5,100.5) ;
\draw [shift={(221,100.5)}, rotate = 0] [color={rgb, 255:red, 0; green, 0; blue, 0 }  ][fill={rgb, 255:red, 0; green, 0; blue, 0 }  ][line width=0.75]      (0, 0) circle [x radius= 3.35, y radius= 3.35]   ;
\draw    (361,100.5) -- (425,100.5) ;
\draw [shift={(425,100.5)}, rotate = 0] [color={rgb, 255:red, 0; green, 0; blue, 0 }  ][fill={rgb, 255:red, 0; green, 0; blue, 0 }  ][line width=0.75]      (0, 0) circle [x radius= 3.35, y radius= 3.35]   ;
\draw    (461,106) -- (550.5,122) ;
\draw [shift={(550.5,122)}, rotate = 10.14] [color={rgb, 255:red, 0; green, 0; blue, 0 }  ][fill={rgb, 255:red, 0; green, 0; blue, 0 }  ][line width=0.75]      (0, 0) circle [x radius= 3.35, y radius= 3.35]   ;
\draw [shift={(461,106)}, rotate = 10.14] [color={rgb, 255:red, 0; green, 0; blue, 0 }  ][fill={rgb, 255:red, 0; green, 0; blue, 0 }  ][line width=0.75]      (0, 0) circle [x radius= 3.35, y radius= 3.35]   ;

\draw (108,48.4) node [anchor=north west][inner sep=0.75pt]    {$X_{0}$};
\draw (248,43.4) node [anchor=north west][inner sep=0.75pt]    {$X_{i}$};
\draw (379,49.4) node [anchor=north west][inner sep=0.75pt]    {$X_{i+1}$};
\draw (501,57.4) node [anchor=north west][inner sep=0.75pt]    {$X_{n}$};
\draw (200,144.4) node [anchor=north west][inner sep=0.75pt]    {$\diam (\pi _{X_{i}}( q_{i}))\leq \tau, $};
\draw (322,144.4) node [anchor=north west][inner sep=0.75pt]    {$\diam(\pi _{X_{i+1}}( q_{i})) \leq \tau $};
\draw (250,79.4) node [anchor=north west][inner sep=0.75pt]    {$p_{i}$};
\draw (116,87.4) node [anchor=north west][inner sep=0.75pt]    {$p_{0}$};
\draw (380,80.4) node [anchor=north west][inner sep=0.75pt]    {$p_{i+1}$};
\draw (319,76.4) node [anchor=north west][inner sep=0.75pt]    {$q_{i}$};
\draw (501,92.4) node [anchor=north west][inner sep=0.75pt]    {$p_{n}$};
\draw (239,103.4) node [anchor=north west][inner sep=0.75pt]    {$\geq L$};
\draw (374,101.4) node [anchor=north west][inner sep=0.75pt]    {$\geq L$};

\end{tikzpicture}
    \caption{Admissible path}
    \label{fig:admissiblepath}
\end{figure}


\begin{Rmk}\label{ConcatenationAdmPath}
     The path $q_i$ could be allowed to be trivial, so by the (BP) condition, it suffices to check $X_i\ne X_{i+1}$. It will be useful to note that admissible paths could be concatenated as follows: Let $p_0q_1p_1\cdots q_np_n$ and $p_0'q_1'p_1'\cdots q_n'p_n'$ be $(L,\tau)$-admissible. If $p_n=p_0'$ has length bigger than $L$, then the concatenation $(p_0q_1p_1\cdots q_np_n)\cdot (q_1'p_1'\cdots q_n'p_n')$ has a natural $(L,\tau)$-admissible structure.  
      
\end{Rmk}

\begin{rmk}\label{rmk: long admissible path}
    We frequently construct a path labeled by a word $(g_1, g_2,\cdots,g_n)$ for $g_i\in G$, which by convention means the following concatenation
\[[o,g_1o]\cdot g_1[o,g_2o]\cdots (g_1\cdots g_{n-1})[o,g_no]\]
where the basepoint $o$ is understood in context. With this convention, the paths labeled by $(g_1,g_2,g_3)$ and $(g_1g_2, g_3)$ may differ, depending on whether  $[o,g_1o]g_1[o,g_2o]$ is a geodesic or not. \end{rmk}

A sequence of points $x_i$ in a path $p$   is called \textit{linearly ordered} if $x_{i+1}\in [x_i, p^+]_p$ for each $i$. Compared with Morse property as in Lemma \ref{lem:Morse}, we have the following weaker sense of fellow travel property around certain points. 

\begin{defn}[Fellow travel]\label{Fellow}
Let   $\gamma = p_0 q_1 p_1 \cdots q_n p_n$ be an $(L, \tau)$-admissible
path. We say $\gamma$ has \textit{$r$-fellow travel} property for some $r>0$   if for any geodesic  
$\alpha$  with the same endpoints as $\gamma$,   there exists a sequence of linearly ordered points $z_i,
w_i$ ($0 \le i \le n$) on $\alpha$ such that  
$$d(z_i, p_{i}^-) \le r,\quad d(w_i, p_{i}^+) \le r.$$
\end{defn}
\begin{rmk}\label{rmk CAT fellow travel}
If $X$ is a $\cat$ space, then the convexity of the metric as in Remark \ref{rmk: CAT0 property} shows that $\alpha$ lies in the $r$-neighborhood of $\gamma$ and vice versa.  This is the usual sense of fellow travel as asserted in Morse lemma in hyperbolic geometry. 
\end{rmk}

The following result  says that   a local long admissible path has the fellow travel property.

\begin{prop}\label{admisProp}\cite[{Proposition 3.3}]{YangGrowthtightness}
For any $\tau>0$, there exist $L,  r,c>0$ depending only on $\tau,C$ such that  any $(L, \tau)$-admissible path   has $r$-fellow travel property. Moreover, it is a $c$-quasi-geodesic.
\end{prop}

If $G$ is non-elementary and contains one rank-one element, then it contains infinitely many pairwise independent rank-one elements.  See \cite[Lemma 2.12]{YangSCC}. Let $h_{1}, h_{2}, h_{3} \in G$ be any three independent contracting elements. By definition, the $C$-contracting system $\mathcal F=\{g\ax(h_i): g\in G, 1\le i\le 3\}$ has $\tau$-bounded intersection property for some $C, \tau>0$.

The next lemma gives a way to build admissible paths (\cite[Lemma 2.14]{YangSCC}) in $\mathcal F$.
\begin{lem}[Extension Lemma]\label{extend3}
There exist constants  $L, r, B>0$ depending only on $C,\tau$ as in Proposition \ref{admisProp} with the following property.  

Choose any element $f_i\in \langle h_i\rangle$ for each $1\le i\le 3$  to form the set $F$ satisfying $\min\{d(o,fo): f\in F\}\ge L$. Let $g,h\in G$ be any two elements. 
There exists an element $f \in F$ such that   the path  $$\gamma:=[o, go]\cdot(g[o, fo])\cdot(gf[o,ho])$$ is an $(L, \tau)$-admissible path relative to $\mathcal F$. 
\end{lem}

We are ready to prove the main result of this subsection.
\begin{lem}\label{lem:uncountableMyrberg}
Let $G\curvearrowright X$ be a proper   isometric action of a non-elementary group $G$ on a proper $\cat$ space $X$.  Assume that $G$ contains rank-one elements. Then the set of Myrberg points in $\Lambda G$ is uncountable.    
\end{lem}
We remark that the Hausdorff dimension of Myrberg limit set for discrete groups on Gromov hyperbolic spaces has been recently computed in \cite{MY25}. In this case, the above conclusion follows from the positive Hausdorff dimension.
\begin{proof}
We shall construct a map which  embeds  the uncountable set $ \mathbb N^\infty$ into the set of Myrberg points. To this end, let us first list all rank-one elements in $G$ as follows: $$\mathcal R(G):=\{h_1, h_2, \cdots, h_i, \cdots\}$$ ordered by lengths $N_i:=d(o,h_i)\le d(o,h_{i+1}o)$. Note that, given $h\in \mathcal R(G)$,  all nontrivial power $h^n$ for  $n\in \mathbb Z\setminus 0$ are counted in $\mathcal R(G)$.  

Let $L=L(\tau), r=r(\tau)$ be given by Proposition \ref{admisProp}. Choose any element $f_i\in \langle h_i\rangle$ for each $1\le i\le 3$  to form the set $F$ satisfying $\min\{d(o,fo): f\in F\}\ge L$. By Lemma \ref{extend3}, for any two consecutive $(h_i, h_{i+1})$ in $\mathcal R(G)$, we choose $f_i\in F$ with the following properties:
\begin{enumerate}
    \item  $[o,h_io]$ and $[o,h_{i+1}o]$ have $\tau$-bounded projection to $\mathrm{Ax}(f_i)$:
    \[\mathrm{diam}(\pi_{\mathrm{Ax}(f_i)}([o,h_io]))\le \tau \text{ and } \mathrm{diam}(\pi_{\mathrm{Ax}(f_i)}([o,h_{i+1}o]))\le \tau.\]
    \item 
    thus, for any $n\in \mathbb Z\setminus 0$, the word $(h_i, f_i^n, h_{i+1})$ labels an $(L,\tau)$-admissible path (as in Remark \ref{rmk: long admissible path}).  
\end{enumerate}

We shall use the ordered set $\mathcal R(G)$ of rank-one elements and the sequence  $\{f_i\}$ with the above property (1)  to build the map $\Phi:   \mathbb N^\infty\to \Lambda G$. 

Let $\omega=(n_1,n_2,\cdots, n_i,\cdots)\in \mathbb N^\infty$ be a sequence of positive integers. Then define a sequence of group elements $g_j$ for $j=1,\dots, \infty$ such that 
$g_1=1_G$ and $g_j=\prod_{i=1}^{j-1} h_if_i^{n_i}$ for $j\geq 2$.  Using the notation in Remark \ref{rmk: long admissible path},  the finite sequence $w_j=(h_1, f^{n_1}_1, h_2, f^{n_2}_2,\dots, h_j, f^{n_j}_j)$ labels a path
\[\gamma_i= [o, h_1o]h_1[o, f_1^{n_1}o]g_2[o, h_2o]g_2h_2[o, f_2^{n_2}]\dots g_i[o, h_io]g_ih_i[o, f_i^{n_i}o].\]

Consider the (formal) infinite word  $W=\prod_{i=1}^\infty h_i\cdot f_i^{n_i}$ and  associated admissible path $\gamma$ labeled by $W$.
The path is obtained by concatenating paths labeled by $h_i f_i^{2n_i}$ as follows:
$$\gamma=\cup_{i=1}^{\infty} \left(g_i [o, h_io][h_io, h_if_i^{2n_i}]\right).$$
For each $\gamma_i$ above, since $X$ is a geodesic metric space, there exists a geodesic $\alpha_i$ such that $\alpha^-_i=o$ and $\alpha^+_i=\gamma^+_i$. By Proposition \ref{admisProp}, $\alpha_i$ is an $r$-travel fellow of $\gamma_i$. By the convexity of  $\cat$  metric, $\alpha_i$ lies in the $r$-neighborhood of $\gamma_i$. By Ascoli-Arzela Lemma, since any metric ball around $o$ is compact, we may extract subsequence of $\alpha_i$ so that it converges locally uniformly to a  geodesic ray $\alpha$. Moreover, $\alpha$ lies in the $r$-neighborhood of $\gamma$, so that $g_i[o,h_io]$ is $r$-close to $\alpha$. 

We first note that $\gamma$ ends at a Myrberg point, denoted by $\Phi(\omega)$. Indeed, by construction, $\gamma$ is recurrent to any rank-one isometry $h\in \mathcal R(G)$ with arbitrary accuracy.  Thus, $\Phi(\omega)$ is a Myrberg point by Lemma \ref{lem: characterizemyrberg}.  Hence, it remains to prove  that the assignment 
$$
\begin{aligned}
\Phi: \;  \mathbb N^\infty&\longrightarrow \Lambda G\\
\omega &\longmapsto \Phi(\omega)
\end{aligned}
$$
is injective, which then concludes the proof that the limit set contains uncountable many Myrberg points.  

Indeed, assume by contradiction that $\Phi(\omega_1)= \Phi(\omega_2)=:\xi$ for two distinct $\omega_1=(n_i)$ and $ \omega_2=(n_i')$.  Let us assume that $\omega_1$ and $\omega_2$ differ at $n_k\ne n_k'$ for the minimal $k$; that is, $n_i=n_i'$ for each $1\le i< k$ and $n_k<n_k'$ for concreteness. Then the associated $(L,\tau)$-admissible paths  $\gamma_1$ and $\gamma_2$ labeled by  $W=\prod_{i=1}^\infty h_i\cdot f_i^{n_i}$ and  $W'=\prod_{i=1}^\infty h_i\cdot f_i^{n_i'}$ begin to depart at $g_k h_k[o,f_k^{n_k}o]$ and $g_k h_k[o,f_k^{n_k'}o]$ respectively. 

Let $\alpha$ be the geodesic ray starting at $o$ and ending at $\xi$.  On one hand, using Proposition \ref{admisProp} as above, we  see that $\alpha$ lies in the $r$-neighborhood of $\gamma_1$ and $\gamma_2$ and vice versa. Thus, $\gamma_1$ lies in the $2r$-neighborhood of  $\gamma_2$.  On  the other hand, we drop the common subwords for $1\le i< k$ from $W$ and $W'$, and concatenate via $f_k^{n_k'-n_k}$ the two remaining subwords $\prod_{i=k+1}^\infty h_i\cdot f_i^{n_i}$  and $\prod_{i=k+1}^\infty h_i\cdot f_i^{n_i'}$. Precisely, we get a bi-infinite word formed by
$$
U = \left(\prod_{i=k+1}^\infty h_i\cdot f_i^{n_i}\right)^{-1} \cdot f_k^{n_k'-n_k} \cdot \left(\prod_{i=k+1}^\infty h_i\cdot f_i^{n_i'}\right)
$$
where the inverse $(\cdot)^{-1}$ means reversing the word letter by letter. Moreover, recalling that the $(L,\tau)$-admissible paths as labeled by $W, W'$  are defined as  local conditions (LL) and (BP),  we see that $U$ labels a bi-infinite $(L,\tau)$-admissible path $\beta$:  the only thing is to check the (BP) condition for $[o, f_k^{n_k'-n_k}o]$, which holds by the $\tau$-bounded projection by the choice of $f_k$ as above. 

Now to conclude by Proposition \ref{admisProp},  the $(L,\tau)$-admissible path $\beta$ is a bi-infinite quasi-geodesic, which contains as subpaths $\gamma_1$ and $\gamma_2$ except for their finite subpaths.  However, this is impossible, as $\gamma_1$ and $\gamma_2$ lie within the $2r$-neighborhood of each other.   Therefore,  the map $\Phi$ is proven to be injective.      
\end{proof}

\subsection{Topological free actions on visual boundaries}
We are ready to prove the main theorem of this section. The proof uses $\cat$ geometry in a crucial way, though the results on Myrberg points in previous subsection is valid in general metric spaces (\cite{YangConformal}). 

\begin{lem}\cite[Theorem 1.4]{Hos}\label{lem:FixedPtsofHyp}
Let $h$ be a hyperbolic isometry on a proper $\cat$ space $X$. Then the fixed points of $h$ at the visual boundary $\partial_\infty X$ are exactly the set of all accumulation points of $\mathrm{Min}(h)$ in $\partial_\infty X$.   
\end{lem}
\begin{proof}
The proof is short and we include it at the convenience of the reader.

Let $\xi\in \partial_\infty X$ so that $h\xi=\xi$. Choose a geodesic ray $\gamma$ starting at a point in $\mathrm{Min}(h)$ and ending at $\xi$ by Remark \ref{rmk: CAT0 property}. Then $h\gamma$ and $\gamma$ are asymptotic with initial points in $\mathrm{Min}(h)$. The $\cat$ geometry by Remark \ref{rmk: CAT0 property} implies that $d(\gamma(t),h\gamma(t))\le d(\gamma(0),h\gamma(0))$ for every $t>0$. As $\mathrm{Min}(h)$  is the set of points realizing the minimal displacement of $h$ by definition, we see that $\gamma(t)$ is contained in $\mathrm{Min}(h)$ for any $t>0$, so $\xi$ is an accumulation point of $\gamma$ and of $\mathrm{Min}(h)$.

For the converse direction, let $\xi$ be an accumulation point of $x_n\in \mathrm{Min}(h)$.  Recall that $\mathrm{Min}(h)$ is a $h$-invariant non-empty closed convex set in $X$. This implies that the geodesic $\gamma_n=[x_0, x_n]\subset \mathrm{Min}(h)$ for any $n\geq 0$. Then Remark \ref{rem: visual bdry basic property}\ref{visual bdry convergence} implies that for any $\epsilon>0$ and $t\geq 0$, there exists an $N>0$ such that $d(\gamma_n(t), \xi(t))<\epsilon$ for any $n>N$. Then since $\mathrm{Min}(h)$ is $h$-invariant, one has $h\gamma_n\subset \mathrm{Min}(h)$. Furthermore,
because $h$ acts  by isometry, one has $d(h\gamma_n(t), h\xi(t))=d(\gamma_n(t), \xi(t))<\epsilon$. This implies
\[d(\xi(t), h\xi(t))\leq d(\xi(t), \gamma_n(t))+ d(\gamma_n(t), h\gamma_n(t))+d(h\gamma_n(t), h\xi(t))\leq d_h+2\epsilon\] 
holds for any $t\geq 0$ because $\gamma_n(t)\in \mathrm{Min}(h)$. Hence, $h\xi=\xi$.
\end{proof}

\begin{rmk}\label{rmk: geodesic in product spaces}
    Let $Z=X\times Y$ be a product of metric spaces $X$ and $Y$, equipped with $d_Z=\sqrt{d_X^2+d_Y^2}$. Then \cite[Proposition 5.3(2)]{B-H} implies that $Z$ is a geodesic space if and only if $X$ and $Y$ are so. Now suppose $Z$ is a geodesic space and $\gamma:I\to Z$ is a geodesic, where $I$ is a closed interval in $[0, \infty)$. It is a standard fact, following a similar argument \cite[Proposition 5.3(3)]{B-H}, that there exists geodesics $\gamma_1: I\to X$ and $\gamma_2: I\to Y$ and numbers $a, b\geq 0$ with $a^2+b^2=1$, such that $\gamma(t)=(\gamma_1(at), \gamma_2(bt))$. 
\end{rmk}

\begin{lem}\label{lem: flat quadrant}
    Let $f$ be a hyperbolic isometry on a proper $\cat$ space $X$ that is not rank one. Suppose $\gamma$ is a geodesic ray contained in $\mathrm{Min}(f)$. Then $\gamma$ is contained in a flat quadrant $E$, i.e., an isometric copy of $\{(x, y)\in \IE^2: x\geq 0, y\geq 0\}$.
\end{lem}
\begin{proof}
    Recall that $\mathrm{Min}(f)$ can be decomposed as a product space $\mathrm{Min}(f)=K\times \R$ in Remark \ref{rmk: hyperbolic isometry property} for a convex subspace $K$. Denote by $I=[0, \infty)$ for simplicity. Then Remark \ref{rmk: geodesic in product spaces} implies that $\gamma: I\to \mathrm{Min}(f)$ satisfies $\gamma(t)=(\gamma_1(at), \gamma_2(bt))$ for some geodesics $\gamma_1: I\to K$, $\gamma_2: I\to \R$, and $a, b\geq 0$ such that $a^2+b^2=1$. Then if $a, b\neq 0$ holds, then $\gamma\subset \gamma_1\times \gamma_2$, which is a flat quadrant. If $a=0$, then $\gamma(t)=(\gamma_1(0), \gamma_2(t))$ for $t\in I$, i.e., $\gamma$ is a part of an axis $c$ of $f$ by Remark \ref{rmk: geodesic in product spaces}. Since $f$ is not rank-one, by Definition \ref{defn: rank one}, such the axis $c$ has to be a boundary of a flat half-plane. This implies that $\gamma$ belongs to a flat quadrant.  The final case is that $b=0$, i.e., $\gamma(t)=(\gamma_1(t), \gamma_2(0))$. Note that $\gamma_1$ is an isometric copy of $I$ and thus $\gamma\subset I\times \R$ and in particular, $\gamma$ belongs to a flat quadrant.
\end{proof}

We remark that in the following theorem the cocompact assumption could be removed, if $G$ contains no parabolic elements, which is the only place in the proof using the cocompactness.  
\begin{thm}\label{thm: topo free rank one}
    Let $G\curvearrowright X$ be a proper cocompact isometric action of a non-elementary group on a proper $\cat$ space $X$ with a rank-one element.  Suppose the elliptic radical $E(G)$ is trivial. Then there exists a Myrberg point $z\in \Lambda G\subset \partial_\infty X$ for the  action $G\curvearrowright \Lambda G$ {so that $z$} is a free point in the sense that $\stab_G(z)=\{e\}$. Therefore, the  action $G\curvearrowright \Lambda G$ is topologically free. \end{thm}

\begin{proof}
Since the action $G\curvearrowright X$ is proper and cocompact, there exists no parabolic isometry in $G$ by \cite[Proposition 6.10]{B-H}.
We write $Z=\Lambda G$ for simplicity.  

Let $z\in Z$ be any Myrberg point of $G$. We are going to prove that $\stab_G(z)=\{e\}$. Arguing by contradiction, let $f$ be a non-trivial isometry fixing $z$. According to the classification of isometries in Definition \ref{defn: class isometries} and there are no parabolic isometries in $G$, we divide our discussion into the following two cases.

\textbf{Case 1:} Assume that $f$ is an elliptic isometry, so fixes a point $o\in X$ by definition. Let $\gamma=[o,z]$ be the geodesic ray from $o$ to $z$ (Remark \ref{rem: visual bdry basic property}(i)). Since $f$ fixes $o$ and $z$, Remark \ref{rmk: CAT0 property} implies that $f$ fixes pointwise the geodesic ray $\gamma$, i.e. $f(x)=x$ for any $x\in \gamma$. 

Let $h$ be a rank-one isometry given by Lemma \ref{lem: specialrank1} such that $E(h)=\langle h \rangle$. By Lemma \ref{lem: Bddprojection}, the set of axis $\{g\cdot\mathrm{Ax}(h): g\in G\}$ has bounded intersection.  That is to say, for any $R>0$, there exists $L(R)>0$ depending on $R$ so that the diameter $\diam(gN_R(\mathrm{Ax}(h))\cap N_R(\mathrm{Ax}(h)))\leq L$ for any $g\in G$ with $g\ax(h)\neq \ax(h)$. On the other hand, since $\gamma=[o,z]$ represents a Myrberg point, Lemma \ref{lem: characterizemyrberg} implies that there exists a $R_0>0$ such that for any $L>0$ there are a $g\in G$ and a geodesic segment $p\subset \gamma$ with length $\ell(p)\geq L$ and $p\subset N_{R_0}(g\ax(h))$.

Now, defining $L_0=L(R_0)$ and then for this $L_0$, there exists a $g\in G$ and a segment $p\subset \gamma$ with length $\ell(p)\geq L_0$ and $p\subset N_{R_0}(g\ax(h))$. Recall that $f$ fix $\gamma$ pointwise and therefore $fp=p$. But this entails that $p\subset N_{R_0}(fg\ax(h))\cap N_{R_0}(g\ax(h))$ and thus one has
\[\diam(N_{R_0}(fg\ax(h))\cap N_{R_0}(g\ax(h))\geq L_0.\]
Therefore, by our choice of $L_0=L(R_0)$, the bounded intersection property above implies  $fg\mathrm{Ax}(h)=g\mathrm{Ax}(h)$. In another word, one has $g^{-1}fg\ax(h)=\ax(h)$, which shows that \[g^{-1}fg\in E(h)=\langle h \rangle\simeq \Z\]
by definition of $E(h)$ (see Definition \ref{Defn:E(h)})
and thus $f$ is a torsion free element. But this is a contradiction as any elliptic element in a discrete group must be torsion (Remark \ref{rmk: elliptic torsion}).

\textbf{Case 2:} Assume that $f$ is a hyperbolic isometry such that $fz=z$. First, $f$ can not be rank-one by  Lemma \ref{lem: myrbergpoints}(\ref{rankone fix no Mybrberg}).
Recall that $\mathrm{Min}(f)$ decomposes as a metric product of a convex subset $K$ and the real line $\mathbb R$, on which $f$ acts by 
translation on the real line (Remark \ref{rmk: hyperbolic isometry property}). Then Lemma \ref{lem:FixedPtsofHyp} implies that the Myrberg point $z$ as a fixed point of $f$ in the visual boundary belongs to the boundary of $\mathrm{Min}(f)$, that is an accumulation point of $\mathrm{Min}(f)$ in $\partial_\infty X$. 
In what follows, we could choose  a geodesic ray $\gamma\subset \mathrm{Min}(f)$ that ends at $z$. Since $f$ is not of rank one, Lemma \ref{lem: flat quadrant} implies that $\gamma$ lies in a flat quadrant $E$.

Now, fix a rank-one element $h$ and  its quasi-axis $\mathrm{Ax}(h)$ is $C$-contracting by Lemma \ref{lem:rankone is contracting}.  Without loss of generality, we may assume that $\mathrm{Ax}(h)$ is a geodesic axis in $\mathrm{Min}(h)$. Indeed, as $\mathrm{Ax}(h)$ and $\mathrm{Min}(h)$ are both acted upon cocompactly by $h$, they have finite Hausdorff distance. The contracting property and bounded intersection property are preserved up to a finite Hausdorff distance. 

The geodesic ray $\gamma$ ends at the Myrberg point $z$, so by Lemma \ref{lem: characterizemyrberg} and Remark \ref{rmk: strengthened arbitrary accuracy}, there exists a $R_0>0$ such that for any $n>0$ there is a $g_n\in G$ such that  $\gamma\cap N_R(g_n\ax(h))$ is a geodesic segment with length at least $n$ (since $\ax(h)$ was assumed to be a geodesic axis). For notational simplicity, denote $\gamma_n=\gamma\cap N_R(g_n\ax(h))$. Thus, for each $n$, we can choose a segment $p_n$ of the geodesic axis $g_n\ax(h)$  so that the $R$-neighborhood of $p_n$ contains  $\gamma_n$ and conversely, the $R$-neighborhood of $\gamma_n$ contains $p_n$ by using convexity of $\cat$ metric.  In other words, $p_n$ and $\gamma_n$ has Hausdorff distance at most $R$. By Lemma \ref{rmk: subpath is contracting}, $p_n$ is $C_1$-contracting for some $C_1$ depending on $C$, and then  Lemma \ref{rmk: contracting set permance} thus implies that $\gamma_n$ is a $C_2$-contracting subset for some $C_2$ depending on $C_1$ and $R$. Note that $\gamma_n$ has length at least $n$.

Recall that each $\gamma_n\subset\gamma$  is contained in a flat quadrant $E$ by Lemma \ref{lem: flat quadrant}. The Euclidean geometry of $E$ tells us that  we could choose a geodesic $\beta_n$ in $E$ that is parallel to $\gamma$, so that   the projection of $\beta_n$ to $\gamma$ contains $\gamma_n$ of length $\ge n$: its diameter is at least $n$. Since $C_2$ does not depend on $n$,  this contradicts to the $C_2$-contracting property of $\gamma_n$.

Summarizing the above two cases, we have proved that no nontrivial elements fix the Myrberg point $z$. The proof is concluded by the fact that the action of $G$ on the limit set $\Lambda G$ is minimal as the orbit $Gz$ is thus a dense $G_\delta$ set in $\Lambda G$.\end{proof}

\begin{rmk}
The trivial elliptic radical assumption is necessary. For example,   a direct product of $G$ with any finite group $F$ acts on $X$ with a trivial action of $F$ on $X$ and $\partial_\infty X$. Then $G\times F$ satisfies all the other assumptions, but the induced action on the limit set is by no means topologically free.    
\end{rmk}

\subsection{Non geometric action case}
We shall prove in this subsection that if the $\cat$ space $X$ is additionally assumed to be geodesically complete, then Theorem \ref{thm: topo free rank one} can be proven without assuming the co-compactness of the action. In light of its   proof, it suffices to find a Myrberg point that is not fixed by any parabolic isometry in $G$.  

The cocompact case of the following result was given in Lemma \ref{lem: specialrank1} with a more direct proof.   For the  case of a general proper action, we shall appeal to the techniques of projection complex developed by Bestvina-Bromberg-Fujiwara \cite{BBF}, as we explain below. 
\begin{lem}\label{lem: specialrank1_general}
Let $ G\curvearrowright X$ be a proper   isometric action on a proper $\cat$ space $X$ with a rank-one element. Assume that the elliptic radical of the action is trivial. Then there exists a rank-one isometry $h\in G$ so that the maximal elementary subgroup $E(h)=\langle h\rangle$   is an infinite  cyclic group.    
\end{lem}
\begin{proof}
Let $f\in G$ be a rank-one element. By the work of \cite{BBF}, we build the projection complex $\mathcal P$ from the collection $\{g\mathrm{Ax}(f): g\in G\}$ of axis of $f$. Note that $\mathcal P$  is a quasi-tree on which $G$ acts acylindrically and non-elementarily (\cite{BBFS}). Let  $h$ be a loxodromic element on $\mathcal P$, given by \cite[Corollary 5.7]{Hull} (see also \cite[Theorem 6.14]{DGO}),  with $E(h)=\langle h\rangle$. Note that $E(h)$ is characterized algebraically as the maximal elementary subgroup containing $\langle h\rangle$. Thus, it remains to show that $h$ is a rank-one element on the original space $X$; then $h$ is the desired element in the conclusion.  To this end,  we need  show that  $h$ acts by translation on a contracting quasi-geodesic in $X$. In the remainder of the proof, we shall explain {such contracting quasi-geodesic} is obtained by lifting the axis on $\mathcal P$  to the space $X$. 

It is well-known fact that a loxodromic element $h$ admits a bi-infinite quasi-geodesic path $\alpha$ between the two fixed points in the Gromov boundary of $\mathcal P$, on which $h$ acts by translation. Moreover, $\alpha$ could be chosen so that any subpath is a standard path. Here is the construction of $\alpha$. We connect $h^{-n}\bar o$ and $h^n\bar o$ by a standard path $\alpha_n$ in $\mathcal P$. The triangle formed by three standard paths has the tripod-like property: any side is contained in the union of the other two sides, up to an exception of at most two vertices. We apply this property to the two triangles of the quadrangle formed by $\alpha_n$ and $\alpha_m$. This shows that $\alpha_n$ and $\alpha_m$ has large overlap whose length tends to $\infty$ as $n,m\to\infty$. The limit of the overlap gives the   path $\alpha$ connecting two fixed points of $h$ in the Gromov boundary. According to \cite[Lemma 3.6]{BBFS}, the almost triangle property in projection complex  further shows that $\alpha$ is invariant under $h$. Indeed, if not, $h\alpha\ne \alpha$ has   the same endpoints in the boundary, this contradicts to almost tripod property. Hence, we produced a standard path which is invariant under $h$.    

By  \cite[Lemma 2.25]{YangConformal}, we can transform a standard path $\alpha $ via lifting operation to get an admissible path $\gamma$. Moreover,  as $h$ acts on $\gamma$ by translation,  $\gamma$ is a contracting quasi-geodesic. By Theorem \ref{lem:rankone is contracting},  $h$ is the desired rank-one element, so the proof in non-geometric actions is complete. \end{proof}

We are ready to remove the cocompact assumption of Theorem \ref{thm: topo free rank one}, when the $\cat$ space is geodesically complete. 
\begin{thm}\label{thm: topo free rank one2}
    Let $ G\curvearrowright X$ be a proper isometric action of a non-elementary group on a proper geodesically complete $\cat$ space $X$ with a rank-one element.  Suppose the elliptic radical $E(G)$ is trivial. Then there exists a Myrberg point $z\in \Lambda G\subset \partial_\infty X$ for the  action $G\curvearrowright \Lambda G$ {so that $z$} is a free point in the sense that $\stab_G(z)=\{e\}$. Therefore, the  action $ G\curvearrowright \Lambda G$ is topologically free. 
    \end{thm}
    \begin{proof}
We first show there exists a Myrberg point that is not fixed by any parabolic isometry.  Indeed, let $P$ be the set of parabolic isometries in $G$, which fixes a Myrberg point. By Lemma \ref{lem: myrbergpoints}, since  $X$ is  geodesically complete, the fixed points of elements in $P$ are at most countable, but the Myrberg points are uncountable by Lemma \ref{lem:uncountableMyrberg}. So such a Myrberg point $z$ exists. 
Now, Lemma \ref{lem: specialrank1_general} shows that in this case there is still a rank-one element $h\in G$ such that the maximal elementary subgroup $E(h)=\langle h\rangle$. Then the same proof for the cases of hyperbolic and elliptic isometries in Theorem \ref{thm: topo free rank one} shows that $z$ is also not fixed by any non-trivial elements in $G$. In addition, recall the action $G\curvearrowright \Lambda G$ is minimal by Lemma \ref{lem:rankone-bigthree}, it is thus topologically free.
\end{proof}

It is a standard fact that if the action   $G\curvearrowright X$ is co-compact, then the limit set $\Lambda G$ is exactly the whole visual boundary $\del_\infty X$. As a summary of Proposition \ref{prop: pure inf on limit set}, Theorem \ref{thm: topo free rank one} and \ref{thm: topo free rank one2} , we have the following main result in this section.

\begin{thm}\label{topo free strong boundary on visual summary}
 Let $ G\curvearrowright X$ be a proper isometric  action of a  {non-elementary} group $G$ on a proper  $\cat$ space $X$ with a rank-one element and the elliptic radical $E(G)$ is trivial. Suppose
    \begin{enumerate}[label=(\roman*)]
        \item the action $\alpha$ is cocompact in which case $\Lambda G=\del_\infty X$; or
        \item the space $X$ is geodesic complete.
    \end{enumerate}
Then the topological action $G\curvearrowright \Lambda G$ is a topological free strong boundary action. 
\end{thm}

\subsection{Product actions}

In this subsection, we  extend the  topological freeness result to product actions of groups $G=G_1\times G_2$ on the visual boundary $X=X_1\times X_2$ of product spaces. 

Recall that $\del_\infty X$ is the join $\del_\infty X_1*\del_\infty X_2$ (see, e.g., \cite[Example II.8.11(6)]{B-H}). To be more specific. Let $\xi_1\in \del_\infty X_1$ and $\xi_2\in \del_\infty X_2$ be represented by the geodesic rays $c_1$ and $c_2$ and $\theta\in [0, \pi/2]$. Then $\del_\infty X$ contains points of the form $\xi_1\cos\theta+\xi_2\sin\theta$ represented by the geodesic ray $c: t\mapsto (c_1(t\cos\theta), c_2(t\sin\theta))$.

It is straightforward to see that $G\curvearrowright \del_\infty X$ is far from a minimal system as the copies of $\del_\infty X_i$ are proper $G$-invariant closed subsets in $\del_\infty X$. Moreover, the points in $\del_\infty X_i$ has large stabilizers, e.g., $G_1\leq \stab_G(x)$ for any $x\in \del_\infty X_2$. Nevertheless, these two copies are simply nowhere dense closed sets in $\del_\infty X$. Therefore, we still have the following result.

\begin{cor}\label{topo free product}
    Let $G_1\curvearrowright X_1$ and $G_2\curvearrowright X_2$ be two proper cocompact actions on proper $\cat$ spaces $(X_1, d_1)$ and $(X_2, d_2)$ with rank-one elements. Suppose each $G_i$ is non-elementary and the elliptic radical $E(G_i)$ is trivial. Then for the product action $G_1\times G_2\curvearrowright X_1\times X_2$, the induced topological action $G_1\times G_2\curvearrowright \del_{\infty} (X_1\times X_2)$ is topologically free.
\end{cor}
\begin{proof}
    Denote by $G=G_1\times G_2$ and $X=X_1\times X_2$ for simplicity. Note that $X$ is still a proper $\cat$ space and $\del_\infty X$ is the join $\del_\infty X_1*\del_\infty X_2$ as described above.

    Theorem \ref{thm: topo free rank one2} shows that $G_i\curvearrowright \del_\infty X_i$ is topologically free for each $i=1, 2$ as in this case $\Lambda G_i=\del_\infty X_i$. Let $\xi_i\in \del_\infty X_i$ be a (Myrberg) free point for the action $G_i\curvearrowright \del_\infty X_i$ represented by the geodesic $c_1, c_2$. Choose a $\theta\in (0, \pi/2)$ and define $\xi=\xi_1\cos\theta+\xi_2\sin \theta$, which is represented by the geodesic $c$ described above. We claim $\xi$ is a free point for the action $G\curvearrowright \del_\infty X$. Indeed, let $g=(g_1, g_2)$ be a non-trivial element in $G$. Without loss of generality, one assumes $g_1\neq 1$. Then for the geodesic $c$, note that
    \[g\cdot c(t)=(g_1\cdot c_1(t\cos\theta), g_2\cdot c_2(t\sin\theta))\]
    for any $t\in \R_{\geq 0}$ and observe
    \begin{align*}
        d^2_X(c(t), g\cdot c(t))&=d^2_1(c_1(t\cos\theta), g_1\cdot c(t\cos\theta))+d_2^2(c_2(t\sin\theta), g_2\cdot c_2(t\sin \theta)) \\
        &\geq d^2_1(c_1(t\cos\theta), g_1\cdot c(t\cos\theta)),
        \end{align*}
  which implies that $t\mapsto d_X(c(t), g\cdot c(t))$ is unbounded because $g_1[c_1]\neq [c_1]$ holds for the free point $\xi_1=[c_1]$ on $\del_\infty X_1$. 

  Finally, we show these free point are dense in $\partial_\infty X$. Indeed, first let $\xi=\xi_1\cos\theta+\xi_2\sin\theta\in \partial_\infty X$ be an arbitrary point with $0<\theta<\pi/2$. Then the neighborhood $N(\xi, L, \epsilon)$ at $\xi=[c]$ in $\del_\infty X$ is 
  \begin{align*}
N(\xi, L, \epsilon)=\{&[c']\in \del_\infty X: d(c(t), c'(t))<\epsilon \text{ for any }t<L\}\\
=\{&[c']=[c'_1]\cos\theta'+[c'_2]\sin\theta'\in \del_\infty X: d_1^2(c_1(t\cos\theta), c'_1(t\cos\theta'))\\&+d^2_2(c_1(t\cos\theta), c'_1(t\cos\theta')<\epsilon^2 \text{ for any } t<L\}.
  \end{align*}
Then for the open sets 
$N([c_1], L\cos \theta, \epsilon/2)$ and $N([c_2], L\sin \theta, \epsilon/2)$  on $\del_\infty X_1$ and $\del_\infty X_2$. Since $G_i\curvearrowright \del_\infty X_i$ is topologically free, there exists free points $\xi'_i\in \del_\infty X_i$ for $i=1, 2$ such that $\xi'_1=[c'_1]\in N([c_1], L\cos \theta, \epsilon/2)$ and $\xi'_2=[c'_2]\in N([c_2], L\sin \theta, \epsilon/2)$. Define $\xi'=(\cos\theta)\xi'_1+(\sin \theta)\xi'_2\in \del_\infty X$, which is a free point by the argument above. Then by definition of $c'_1$ and $c'_2$, observe
\[d^2_1(c_1(t\cos\theta ), c'_1(t\cos\theta))+d^2_2(c_2(t\sin \theta), c'_2(t\sin\theta))<\epsilon^2/2\]
for any $t<L$. This implies that $\xi'\in N(\xi, L, \epsilon)$. Finally, by the definition of the $\del_\infty X=\del_\infty X_1*\del_\infty X_2$, the copies of all $\del_\infty X_i$ are two meager closed sets and therefore, the set 
\[D=\{\xi=\xi_1\cos\theta+\xi_2\sin\theta: \xi_1\in \del_\infty X_1, \xi_2\in \del_\infty X_2, \theta\in(0, \pi/2) \}\]
is open dense in $\del_\infty X$. Thus, for any open set $O$ in $\del_\infty X$, the set $D\cap O$ is a non-empty open set and thus contains a free points as demonstrated above. Therefore, $G\curvearrowright \del_\infty X$ is topologically free.
\end{proof}

We provide applications to Coxeter groups acting on Davis complexes (see Subsection \ref{subsec Coxeter groups}). 

 \begin{cor}\label{topo free for Coxter}
 Let $W_S$ be an irreducible non-spherical non-affine Coxeter group. Then the boundary action $W_S\curvearrowright \partial_\infty \Sigma(W, S)$ is a minimal topologically free strong boundary action. On the other hand, if $W_S$ is reducible, let us write $W_S=W_{S_1}\times\dots W_{S_m}$, so that each factor $W_{S_i}$ is of neither spherical type nor affine type. Then the action $W_S\curvearrowright \del_\infty \Sigma(W, S)$  is still topologically free.
   \end{cor}
   \begin{proof}
   Note that $W_S\curvearrowright \Sigma(W, S)$ is a proper cocompact isometric action on the proper $\cat$ space $\Sigma(W, S)$ and the elliptic radical $E(W_S)$ is trivial by the irreducibility of $W_S$.
   The first part directly follows from Theorem \ref{topo free strong boundary on visual summary} based on Proposition \ref{Coxeter group rankone}.  The second part       is a straightforward consequence of the first part and Corollary \ref{topo free product}. 
   \end{proof}
We will investigate the combinatorial boundary actions of Coxeter groups in Section \ref{sec: Coxeter combinatorial boundary}.

\section{Boundaries of paraclique graphs}\label{sec: roller and NS bdry}

This section  studies several boundaries including graph boundary, combinatorial, and Roller boundary associated to a class of paraclique graphs recently introduced by Ciobanu and Genevois \cite{CG25}. The main result is that all these boundaries are naturally homeomorphic to horofunction boundary, provided that graph compactification is visual.  

As a warm-up, we first introduce the basics of $\cat$ cube complexes and its Roller boundaries, which will motivate the corresponding concepts and results in paraclique graphs.

\subsection{\texorpdfstring{$\cat$}\ \  cube complexes and Roller boundaries}\label{SSubRollerBoundary}
We refer to \cite{N-S, C-S, Le, Char} for more information    on $\cat$ cube complexes. 
\begin{defn}
A $\cat$ \textit{cube complex} is a simply connected cell complex whose cells are standard Euclidean cubes $[0, 1]^d$ ($d\ge 0$) of various dimensions glued isometrically along sides with the following addition property: the link of each $0$-cell (i.e. vertex) is a \textit{flag} complex. Note that a flag complex is a simplicial complex where any $n+1$ adjacent vertices belong to an $n$-simplex.
\end{defn}
By Gromov's link criterion, a $\cat$ {cube complex} $X$ is indeed a $\cat$ geodesic space equipped with the induced length metric from Euclidean cubes. We say a $\cat$ cube complex $X$ is \textit{finite dimensional} if there is a uniform upper bound on the dimension of cubes in $X$. In this finite-dimensional case, the complex $X$ is a complete metric space with respect to the length metric. A $\cat$ cube complex $X$ is said to be \textit{locally finite} if no point of $X$ is contained in infinitely many cubes. Note that a  $\cat$ cube complex $X$ is locally finite if and only if it is locally compact. Therefore, a locally finite, finite-dimensional $\cat$ cube complex is proper by Theorem \ref{Thm: Hopt-Rinow}. Except for the $\cat$-geometry, a $\cat$ {cube complex} could be fruitfully understood via the combinatorial metric on its 1-skeleton,  which is usually equipped with the usual $\ell_1$-metric (called the path metric or the combinatorial metric as well). For the finite-dimensional case, the $\ell_1$-metric and the usual $\cat$ metric on $X$  are quasi-isometric to each other by \cite[Lemma 2.2]{C-S}.


An important combinatorial feature of a $\cat$ cube complex is its hyperplane/halfspace structure.  Let $X$ be a $\cat$ cube complex. A \textit{midcube} of a cube $[0, 1]^d$, is the restriction of one coordinate of the cube to be $1/2$. A \textit{hyperplane} $\mathfrak{h}$ is a connected subspace of $X$ satisfying that for each cube $C$ in $X$, the intersection $\mathfrak{h}\cap C$ is either a midcube of $C$ or empty. For  an edge  $e$  in $X^1$, we say a hyperplane $\mathfrak{h}$ is dual to $e$ if $\mathfrak{h}\cap e\neq \emptyset$.  In general, $\mathfrak{h}$ separates $X$ into the  two components, called \textit{halfspaces}, denoted by $\cev{\mathfrak h}$ and $\vec{\mathfrak h}$. A hyperplane is called \textit{essential} if each of its associated halfspaces contains points arbitrarily far away from the hyperplane. A CAT(0) cube complex $X$ is called \textit{essential} if all of its hyperplanes are essential. We refer to  Figure \ref{fig: CAT0 cube complex} for illustrations of the above notions.

We say that two hyperplanes $ {\mathfrak h},  {\mathfrak k}$ are  \textit{$L$-well separated} for $L>0$ if the number of hyperplanes intersecting both of them is most $L$.  The $0$-well separated hyperplanes are also called \textit{strongly separated}. In particular, $ {\mathfrak h},  {\mathfrak k}$ are disjoint. Furthermore, we say that are \textit{super strongly separated} in \cite{FLM} if any two hyperplanes $ {\mathfrak h}',   {\mathfrak k}'$ intersecting $ {\mathfrak h},  {\mathfrak k}$ respectively must be disjoint. Thus, if $( {\mathfrak h},  {\mathfrak f})$ and $( {\mathfrak f},  {\mathfrak k})$ are  strongly separated, then $( {\mathfrak h},  {\mathfrak k})$ is super strongly separated. Two disjoint half spaces  are \textit{super strongly separated} if their bounding hyperplanes   are super strongly separated. These motivate the analogous notions in paraclique graphs given by Definition \ref{hyperplanesconfiguration}.

\begin{figure}[ht]
    \centering
\begin{tikzpicture}
      \node[label={left, yshift=0cm:}] at (1.2,0.8) {\small$\mathfrak{h_2}$};
      \node[label={above, yshift=0cm:}] at (7.75,1) {\small$\mathfrak{h_1}$};
      \node[label={left, yshift=0cm:}] at (-0.4,0) {\small$.\ .\ .$};
      \node[label={right, yshift=0cm:}] at (9,0.5) {\small$.\ .\ .$};

      \node[label={below, yshift=0cm:}] at (7.05,0) {\small$\cev{\mathfrak{h}}_1$};
      \node[label={below, yshift=0cm:}] at (8.5,0) {\small$\vec{\mathfrak{h}}_1$};
      \node[label={below, yshift=0cm:}] at (0.6,1.2) {\small$\uparrow$};
      \node[label={below, yshift=0cm:}] at (0.6,0.4) {\small$\downarrow$};
      \node[label={below, yshift=0cm:}] at (0.2,1.2) {\small$\cev{\mathfrak{h}}_2$};
      \node[label={below, yshift=0cm:}] at (0.2,0.4) {\small$\vec{\mathfrak{h}}_2$};

      \draw (0,0) -- (1.5,0);
      \draw (1.5,0) -- (3,0);
      \draw (1.5,0) -- (1.5,1.5);
      \draw (3,0) -- (3,1.5);
      \draw (1.5,1.5) -- (3,1.5);
      \draw (3,0) -- (4.5,0);
      \draw (4.5,0) -- (4.5,1.5);
      \draw (3,1.5) -- (4.5,1.5);
      \draw (4.5,0) -- (5.5,0.5);
      \draw (4.5,1.5) -- (5.5,2);
      \draw (3,1.5) -- (4,2);
      \draw (4,2) -- (5.5,2);
      \draw (5.5,2) -- (5.5,0.5);
      \draw (5.5,0.5) -- (7,0.5);
      \draw (7,0.5) -- (8.5,0.5);
      \draw (7,0.5) -- (7,2);
      \draw (5.5,2) -- (7,2);
      \draw[dotted] (4,2) -- (4, 0.5);
      \draw[dotted] (4,0.5) -- (5.5,0.5);
      \draw[dotted] (4,0.5) -- (3,0);

      \filldraw[draw=black, fill=gray!20](3,0.75)-- (4.5, 0.75) -- (5.5, 1.25) -- (4,1.25) -- (3, 0.75); 

      \draw (4.5, 0) -- (4.5, 1.5);
      \draw[gray] (1.5,0.75) -- (3,0.75);
      \draw[gray] (5.5,1.25) -- (7,1.25);

      \tikzset{enclosed/.style={draw, circle, inner sep=0pt, minimum size=.07cm, fill=black}}      
      \node[enclosed, label={right, yshift=.2cm:}] at (7.75,0.5) {};

\end{tikzpicture}

    \caption{A $\cat$ cube complex $X$. The vertex $\mathfrak{h_1}$ is an essential hyperplane that separates $X$ into two half-spaces: the left part $\cev{\mathfrak{h}}_1$ and the right part $\vec{\mathfrak{h}}_1$. On the other hand, the hyperplane $\mathfrak{h_2}$  is not essential.} 
    \label{fig: CAT0 cube complex}
\end{figure}
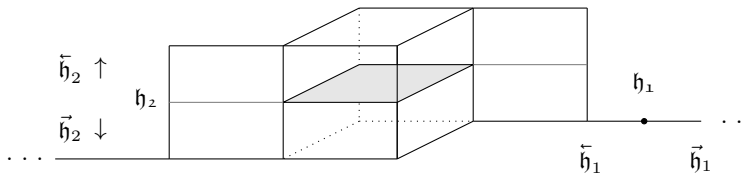

We study the symmetries on $\cat$ cube complexes, which are isometries preserving cubical structures. Let $X$ be a $\cat$ cube complex. Let $\aut(X)$ be the automorphism group consisting of all isometries that preserve the cubical structures. Then, $\aut(X)$ is a subgroup of $\opisom(X)$ and one may still regard $X$ as an $\opisom(X)$-space as in Subsection \ref{subsec}. The space $X$ is called \textit{cocompact} if the action $\aut(X)\curvearrowright X$ is cocompact. We say a group $G\leq \aut(X)$ acts \textit{essentially} on $X$ if no $G$-orbit remains in a bounded neighborhood of a halfspace of $X$.  Note that an essential action implies the underlying $\cat$ cube complex must be essential, but the converse is of course not true. However, one may restrict to the essential core of $X$ on which the action becomes essential. We refer to \cite[Section 3]{C-S} for  more relevant discussions. 

A $\cat$ cube complex $X$ is said to be \textit{irreducible} if it cannot be written as a nontrivial product of two $\cat$ cube complexes (i.e. no factor is singleton). Otherwise,  $X$ is \textit{reducible}. Let $n\in \N$. An $n$-dimensional \textit{flat} is an isometrically embedded copy of $n$-dimensional Euclidean space $\mathbb{E}^n$ (in the usual $\cat$ metric). An unbounded cocompact $\cat$ cube complex $X$ is said to be \textit{Euclidean} if $X$ contains a $\aut(X)$-invariant flat. Otherwise, we say $X$ is \textit{non-Euclidean}. It is called  \textit{strictly non-Euclidean}  if all irreducible factors are non-Euclidean.

\subsubsection*{\textbf{Roller boundary}}



From now on, assume that $X$ is a locally compact finite-dimensional $\cat$ cube complex.
Let us denote by ${\CH}$ the collection of all hyperplanes of $X$ and $\CS$ the set of all halfspaces. Following \cite{Roller}, we use {ultrafilters} to define the \textit{Roller compactification}. 

An \textit{orientation}  of hyperplanes is a map $\sigma: \CH\to \CS$ with the following properties
\begin{enumerate}
	\item For any hyperplane $\mathfrak{h}\in \CH$,     $\sigma( {\mathfrak h}) \in \{\cev {\mathfrak h},\vec {\mathfrak h}\}$ and
	\item For any two $\mathfrak{h}\ne \mathfrak{h}' \in \CH$,  $\sigma(\mathfrak{h})\cap \sigma(\mathfrak{h}')\ne\emptyset$.
\end{enumerate}
The image of $\sigma$ is called \textit{ultrafilter} in $\CS$. 
Each vertex $x\in X^0$ defines  an orientation $\sigma_x: \CH\to \CS$ with $x\in \sigma_x({\mathfrak h})$ for any $\mathfrak h\in \CH$ whose  image is called \textit{principal ultrafilter}. The collection $\CU(X)$ of all ultrafilters on $\CS$ is naturally identified as a   subset of $$\prod_{\mathfrak{h}\in {\CH}}\{\cev {\mathfrak h}, \vec {\mathfrak h}\}$$
which is a closed subset under product topology, and thus forms a compact metrizable space. The $\del_RX=\CU(X)\setminus X^0$ is called \textit{Roller boundary}. When  $X$ is locally finite, which is our main concern, $X^0$ is open and dense in $\CS$ and $\del_RX$ is  a compact  space. In general, $X^0$ might not be open. 

Let $G\le \aut(X)$. Then the action $G\curvearrowright X$ by automorphism induces the action $G$ on the collection $\CH$ of hyperplanes, which yields topological actions $G\curvearrowright \CU(X)$ and $G\curvearrowright \del_R(X)$ as $X^0$ is an invariant and open dense subset in $\CU(X)$.

We say that two boundary points $\xi,\eta \in \partial_\mathcal{R} X$ have \textit{finite symmetric difference} if the symmetric difference of their associated consistent orientations $U_{\xi} \Delta U_{\eta}$ is finite. 
According to \cite[Proposition 6.20]{FLM}, by an unpublished result of Bader and Guralnik, the Roller compactification is homeomorphic to the horofunction compactification of $X^0$. 
Moreover, \cite[Lemma 11.5]{YangConformal} shows that these two natural relations coincide, as stated below.

\begin{prop}\label{Roller to horofunction}
    There exists a canonical  homeomorphism 
    $$ \Phi : X^0 \cup \partial_{R} X \to X^0 \cup \partial_h X^0 $$ that restricts to the identity on $X^0$. 
    Furthermore, the finite symmetric difference relation on $\partial_{R} X$ corresponds precisely to the finite difference relation on $\partial_h X^0$.
\end{prop}

In \cite{N-S}, Nevo and Sageev studied a particular closed $G$-invariant subset $B(X)$ of $\del_RX$. We refer to it  the \textit{Nevo-Sageev boundary} in this paper. Consider the following set $\CU_{NT}(X)$ consisting all non-terminating ultrafilters:
\[\CU_{NT}(X)=\{\alpha\in\CU(X): h\in \alpha \Rightarrow \text{there exists }  h'\in \alpha\ \text{with }h'\subsetneq h\}\]
and define $B(X)=\overline{\CU_{NT}(X)}$ in $\CU(X)$. We remark that $B(X)$ is always non-empty if $X$ is essential and cocompact by \cite[Theorem 3.1]{N-S}. 

Unlike the visual boundary,  $B(X)$ has the following nice decomposition property. Namely, if $X$ is not irreducible and decomposes as $X=\prod_{i=1}^nX_i$, then $B(X)=\prod_{i=1}^nB(X_i)$ so that the dynamics on $B(X)$ could be reduced to the one on each factor.

\subsubsection*{\textbf{Rank-one isometries}}
Let $X$ be a locally compact finitely-dimensional $\cat$ cube complex so that $\aut(X)$ acts essentially on $X$ without fixed points  at the visual boundary. By Lemma \ref{lem:rankone is contracting}, a rank-one isometry on $\cat$ space is exactly a contracting isometry in Definition \ref{defn: contracting}. 
\begin{thm}\cite[Proposition 5.1, Theorem 6.3]{C-S}
Assume that $\aut(X)$ acts essentially on $X$ without fixed   points in $\partial_\infty X$. The following statements are equivalent:
\begin{enumerate}[label=(\roman*)]
    \item 
    $X$ is irreducible;
    \item 
    $X$ contains a pair  of strongly separated hyperplanes; 
    \item 
    Any group $G<\aut(X)$ without fixed points  in $\partial_\infty X$  contains rank-one elements in the $\cat$ metric. 
\end{enumerate}
\end{thm} 

\begin{rmk}
Assume that a non-elementary group $G<\isom(X)$ acts properly on a proper $\cat$ space $X$ with rank-one elements. Then $G$ fixes no   points in $\partial_\infty X$. Indeed, A non-elementary group $G$ contains at least two independent rank-one elements. By \cite[Lemma 3.20]{YangConformal}, any two independent rank-one elements have disjoint fixed points, so the conclusion follows.   Thus, we may assume the action  of $G$ on $X$ is essential, by restricting to the essential core of $X$  by \cite[Proposition 3.5]{C-S}, since the action   has no fixed points at the visual boundary.  
\end{rmk}

We are also interested in the contracting isometries on the  $1$-skeleton of $X$. It turns out the   notions of contracting isometries coincide on $X$ and on its $1$-skeleton. 
A similar  fact holds in Coxeter groups by Proposition \ref{Coxeter group rankone}. .
\begin{lem}\cite[Lemma 8.3]{GYang}\cite[Lemma 11.6]{YangConformal}\label{NS dynamics on cube complex}
Assume that $X$ is irreducible. Then 
\begin{enumerate}[label=(\roman*)]
    \item the set of contracting isometries in $\cat$ metric are exactly contracting in the  $\ell^1$-metric.
    \item 
    Assume that a group $G<\aut(X)$ acts essentially on $X$. Then $G$  contains a contracting isometry $g$ in the  $\ell^1$-metric so that their fixed $[\cdot]$-classes $[g^-], [g^+]$ are singletons in $\partial_h X$.  In particular, $g$ performs north-south dynamics on $\partial_h X$.
\end{enumerate}   
\end{lem}


\subsection{Paraclique graphs and their cubical-like geometry}\label{subsec paraclique}
In this subsection, we  follow  Ciobanu-Genevois \cite{CG25} closely to give an account of a more general class of graphs called paraclique graphs. 

 
Unless otherwise mentioned, assume that  $X$ is  a connected simplicial graph (i.e. without loops and multiple edges). For simplicity, assume that $X$ is countable. 

We say that a subgraph $Y\subset X$ is \textit{gated} if for any $x\in X$, there exists a  vertex in $Y$  called \textit{gate} of $x$, denoted by $\pi_Y(x)$, so that for any $y\in Y$ there exists a geodesic from $x$ to $y$ passing through $\pi_Y(x)$. By definition, the gate $\pi_Y(x)$ is necessarily unique and coincides with the shortest projection point of $x$ to $Y$ (see \textsection \ref{SSubRankone2}).  We say a subgraph $Y$ is \textit{convex} in $X$ if $Y$ contains every geodesic between any two points in it.  A gated graph is necessarily convex, but the converse may not be true. 

Let us first introduce  quasi-median graphs, which have been well-studied in graph theory (see \cite{BMW}) and whose cubical geometries are recently studied by Genevois \cite{Gen17}. 
\begin{defn}
A graph $X$ is called \textit{quasi-median} if it satisfies the following
\begin{enumerate}
    \item 
    the \textit{triangle condition}: for all vertices $o,x,y$ in $X$ with $x,y$ adjacent and $d(o,x)=d(o,y)$, there exists a common neighbor $z$ of $x,y$ so that $d(o,z)=d(o,x)-1$.  
    \item 
    the \textit{quadrangle condition}: for all vertices $o,x,y,z$ in $X$ with  $d(x,z)=d(y,z)=1$ and $d(o,x)=d(o,y)$, there exists a common neighbor $w$ of $x,y$ so that $d(o,z)=d(o,x)-2$. 
    \item 
    no two triangles share only one edge, and no two 4-circles share only two adjacent edges. 
\end{enumerate}        
\end{defn}
\begin{rmk} 

Median graphs could be realized as the one-skeleton of $\cat$ cube complexes, and quasi-median graphs without 3-cliques are exactly median graphs. The Cayley graphs of graph products are examples of quasi-median graphs  \cite{Gen17}. 

\end{rmk}

By a \textit{clique} we mean a maximal complete subgraph in $X$. A graph $X$ is called \textit{clique-gated} if every clique is gated. In \cite[Theorem 3.1]{HK96}, Hagauer and Klavžar showed that a graph $X$ is clique-gated if and only if it satisfies both the triangle condition and the property that no two triangles share exactly one edge. Equivalently, any two cliques in $X$ are either parallel or antipodal in the following sense.
\begin{figure}
    \centering

\tikzset{every picture/.style={line width=0.75pt}} 

\tikzset{every picture/.style={line width=0.75pt}} 

\begin{tikzpicture}[x=0.75pt,y=0.75pt,yscale=-1,xscale=1]

\draw   (148.5,64.67) -- (188.33,88) -- (108.67,88) -- cycle ;
\draw  [dash pattern={on 4.5pt off 4.5pt}] (147.83,166) -- (187.67,189.33) -- (108,189.33) -- cycle ;
\draw    (108.67,88) -- (108,189.33) ;
\draw  [dash pattern={on 4.5pt off 4.5pt}]  (148,65.67) -- (147.33,167) ;
\draw    (188.33,88) -- (187.67,189.33) ;
\draw    (108,189.33) -- (187.67,189.33) ;
\draw   (328.98,101.33) -- (289.31,66.75) -- (368.98,66.58) -- cycle ;
\draw   (327.79,155.33) -- (367.3,190.09) -- (287.64,189.9) -- cycle ;
\draw    (328.98,101.33) -- (327.67,156.67) ;

\draw (91.33,75.73) node [anchor=north west][inner sep=0.75pt]    {$u_{1}$};
\draw (92.67,181.4) node [anchor=north west][inner sep=0.75pt]    {$v_{1}$};
\draw (144,53.07) node [anchor=north west][inner sep=0.75pt]    {$u_{3}$};
\draw (187.67,76.07) node [anchor=north west][inner sep=0.75pt]    {$u_{2}$};
\draw (187.33,178.73) node [anchor=north west][inner sep=0.75pt]    {$v_{2}$};
\draw (148,154.4) node [anchor=north west][inner sep=0.75pt]    {$v_{3}$};
\draw (142.67,168.4) node [anchor=north west][inner sep=0.75pt]    {$\Delta _{1}$};
\draw (140.67,69.4) node [anchor=north west][inner sep=0.75pt]    {$\Delta _{2}$};
\draw (320,166.73) node [anchor=north west][inner sep=0.75pt]    {$\Delta _{1}$};
\draw (323.33,70.73) node [anchor=north west][inner sep=0.75pt]    {$\Delta _{2}$};
\draw (371.33,57.4) node [anchor=north west][inner sep=0.75pt]    {$u'$};
\draw (315.33,146.07) node [anchor=north west][inner sep=0.75pt]    {$v$};
\draw (274,185.4) node [anchor=north west][inner sep=0.75pt]    {$v'$};
\draw (316,96.4) node [anchor=north west][inner sep=0.75pt]    {$u$};

\end{tikzpicture}
    \caption{Parallel $3$-cliques (left) and antipodal cliques (right) in a paraclique graph. Antipodal cliques are not necessarily of same size.}
    \label{fig:defn parallel antipodal}
\end{figure}
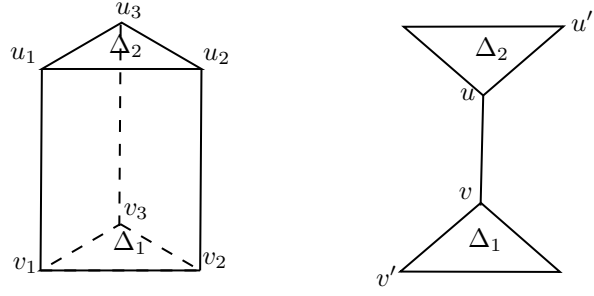
\begin{defn}\label{defn parallel antipodal}
Let  $\Delta_1,\Delta_2$  be two cliques of a graph $X$.
\begin{enumerate}[label=(\roman*)]
    \item $\Delta_1,\Delta_2$ are \textit{parallel} (write $\Delta_1\parallel\Delta_2$) if they are of same size, i.e. $\Delta_1=\{v_1,v_2,\cdots, v_k\}$ and $\Delta_2=\{u_1,u_2,\cdots, u_k\}$ for some $k\ge 1$, and 
\[
d(v_i,u_j)=
\begin{cases}
d(\Delta_1,\Delta_2), & i=j\\
d(\Delta_1,\Delta_2)+1, & i\ne j
\end{cases}
\]

\item 
$\Delta_1,\Delta_2$ are called \textit{antipodal} if there exists a unique pair of vertices $(u,v)\in \Delta_1\times \Delta_2$ so that $d(u,v)=d(\Delta_1,\Delta_2)$ and for other $u'\in \Delta_1\setminus u$ and $v'\in \Delta_2\setminus v$, 
\[
\begin{aligned}
d(u,v')&=d(u,v)+1=d(u',v)\\
d(u',v')&=d(u,v)+2   
\end{aligned}
\]
\end{enumerate}
See Figure \ref{fig:defn parallel antipodal} for illustration. 
\end{defn}

Paraclique graphs are introduced by Ciobanu-Genevois as a further generalization of quasi-median and mediangle graphs. The latter was recently introduced by Genevois in \cite[Definition 1.4]{Gen22} to generalize the Cayley graph of a finite rank Coxeter group (\cite[Proposition 3.24]{Gen22}). It is proved that mediangle graphs are paraclique (see \cite[Proposition 6.3]{CG25}). 
\begin{defn}\cite[Definition 2.7]{CG25}
A clique-gated graph $X$ is called \textit{paraclique} if the parallel relation between  cliques  is transitive:  if $\Delta_1$ is parallel to $\Delta_2$  and $\Delta_2$ is parallel to $\Delta_3$ then $\Delta_1$ is parallel to $\Delta_3$. The union of all the cliques in one parallelism class  defines a \textit{hyperplane} $ {\mathfrak h}$ in $X$. 
\end{defn}
\begin{rmk}
In a parallelism class, there is  a canonical bijection between the vertex sets of any two  cliques. That is, if $\phi_{1,2}:\Delta_1\to \Delta_2$ and $\phi_{2,3}:\Delta_2\to \Delta_3$ denote the corresponding bijections given by Definition \ref{defn parallel antipodal} for $\Delta_1\parallel \Delta_2$ and $\Delta_2\parallel\Delta_3$, then $\phi_{2,3}\phi_{1,2}:\Delta_1\to \Delta_3$ is exactly the  bijection given for $\Delta_1\parallel \Delta_3$. This follows from Proposition \ref{clique-sectors}(iii).
\end{rmk}


In what follows, unless mentioned otherwise,  $X$ is assumed to be a paraclique  graph.      

We say that a path $p$ \textit{crosses} a hyperplane $ {\mathfrak h}$ if $p$ contains an edge $e$ of some clique in $ {\mathfrak h}$.  Let $\mathcal H$ denote the set of all hyperplanes in $X$.
Let $X \dblsetminus  {\mathfrak h}$ denote  the graph obtained  by removing (the interiors of) edges of $ {\mathfrak h}$ in $X$.

\begin{prop}\cite[Proposition 2.9]{CG25}\label{clique-sectors}
Every component of $X \dblsetminus   {\mathfrak h}$  is convex. Moreover, for any clique $\Delta$ in $ {\mathfrak h}$, there is a one-to-one correspondence between $X \dblsetminus   {\mathfrak h}$ and the vertex set of $\Delta$ as follows:
\begin{enumerate}[label=(\roman*)]
    \item Each vertex $x$ in $\Delta$  belongs to a component which is exactly $\pi_\Delta^{-1}(x)$;
    \item Each component contains exactly one vertex in $\Delta$.
    \item 
    If $[x,y] \subset \Delta$ is an edge that is parallel to an edge $[x',y']$ in another clique $\Delta'$ of $ {\mathfrak h}$, then $\pi_\Delta^{-1}(x)=\pi_{\Delta'}^{-1}(x')$ and  $\pi_\Delta^{-1}(y)=\pi_{\Delta'}^{-1}(y')$.        
\end{enumerate}
\end{prop}
\begin{proof}    
The item \textit{(iii)} is not explicitly stated in \cite[Proposition 2.9]{CG25}, but follows from the definition \ref{defn parallel antipodal} of parallel cliques. Indeed, if $[x,y]$ is parallel to $[x',y']$, then $x'\in \pi_{\Delta}^{-1}(x)$, so  $\pi_\Delta^{-1}(x)=\pi_{\Delta'}^{-1}(x')$ follows from \textit{(ii)}. Similarly, this holds for $y,y'$. 
\end{proof}

Following \cite{CG25}, we refer to each component of $X \dblsetminus   {\mathfrak h}$ as a \textit{sector delimited} by $ {\mathfrak h}$ and denote all sectors by $\mathcal S( {\mathfrak h})$. It is clear that $\mathcal S( {\mathfrak h})$ forms a partition of the vertex set of $X$.   If $\mathcal S( {\mathfrak h})$ contains exactly two sectors,  each sectors are usually called   \textit{half-spaces}. Let $\mathcal S(\mathcal H)$ denote all the sectors delimited by hyperplanes. (In some literature,  the intersection of finitely many half-spaces are usual called sectors in the cube complex or Davis complex.) 

As in $\cat$ cube complexes, the geodesics in paraclique graph is characterized by the following same property. 
\begin{prop}\cite[Proposition 2.10]{CG25}\label{geodesicwalls}
A path is a geodesic in $X$ if and only if it crosses each hyperplane at most once. In particular, the distance between two vertices $x,y$ is the number of hyperplanes separating $x$ and $y$. 
\end{prop}

Denote $I(o,x)=\{y\in X: d(o,y)+d(y,x)=d(o,x)\}$ the \textit{interval} between $o$ and $x$. This is exactly the set of all vertices on some geodesic from $o$ to $x$.
By Proposition \ref{geodesicwalls}, any two geodesics between $o$ and $x$  cross the same family of hyperplanes which  are exactly the ones separating $o$ and $x$.  It also follows that if a geodesic  enters into a sector, it will not leave.

We now give a few elementary facts based on hyperplane separation considerations.  

By Proposition \ref{clique-sectors}, each edge $e=[x,y]$ in a clique of  $ {\mathfrak h}$ gives two distinct sectors of $ {\mathfrak h}$ which contains $x$ and $y$ respectively. Parallel edges $e=[x,y], e'=[x',y']$ in $ {\mathfrak h}$ define the same pair of distinct sectors containing $\{x,x'\}$ and $\{y,y'\}$ respectively. 

Further, we say that  a path $p$ crosses a hyperplane $ {\mathfrak h}$ \textit{through a pair of sectors} $\hat s_1,\hat s_2\in\mathcal S( {\mathfrak h})$ if $p$ contains an edge $e$ so that the sectors $\hat s_1,\hat s_2$ are determined by $e$. Any two geodesics between $o$ and $x$  cross the same family of hyperplanes through the same  pair of sectors. 

For further discussion, we introduce the following auxiliary notations.
\begin{enumerate}
    \item Let   ${\mathcal H}(p)$ denote the set  of hyperplanes (with repetition)  that  a path $p$ crosses. Then  $p$ is a geodesic if and only if $|{\mathcal H}(p)|$ is the length of $p$.

    \item 
    Let $p$ be an oriented geodesic path or geodesic ray. Let $\vv{\mathcal S}(p)$ denote the \emph{ordered} pairs of sectors $(\hat s_1,\hat s_2)$ delimited by $\mathfrak h\in {\mathcal H}(\alpha)$ so that  the oriented $p$  exits $\hat s_1$ and enters $\hat s_2$. 
    
    \item 
    Let $p$ be an oriented  geodesic ray.
    Denote by $ {\mathcal S}(p)$  the set of sectors $\hat s$ in $\mathcal S(\mathcal H)$  into which  $p$ eventually enters. That is, upon removal of a finite initial subpath, $p$ is contained in $\hat s$.
\end{enumerate}

We first derive a simple elementary fact from Proposition \ref{geodesicwalls}.

\begin{lem} 
Let $p$ and $q$ be two geodesic paths in $X$ so that ${\mathcal H}(p)$ and ${\mathcal H}(q)$ are disjoint. If the terminal endpoint of $p$ is the same as the initial endpoint of $q$ then the concatenation $p\cdot q$ is a geodesic path.
\end{lem} 
\begin{proof}
If not, $p\cdot q$ crosses at least twice   a hyperplane $\mathfrak h$ according to Proposition \ref{geodesicwalls}.  This contradicts the assumption ${\mathcal H}(p)\cap {\mathcal H}(q)=\emptyset$, as $\mathfrak h$ is crossed exactly once by $p$ and by $q$. 
\end{proof}

The next two lemmas explain typical situations where a connected argument using hyperplane separation could be applied. 
\begin{lem}\label{insamesector}
Consider a geodesic triangle in $X$ formed by $[x,y], [y,z]$ and $[x,z]$. If  $[x,y]$ and $[y,z]$ cross the same hyperplane $\mathfrak h$ through a same pair of sectors  delimited by $\mathfrak h$, then   $y$ and $z$ lie in the same  sector delimited by $\mathfrak h$.     
\end{lem}
\begin{proof}
This follows immediately from Proposition \ref{clique-sectors}. Indeed, let $[u,v]$ and $[u',v']$ be the two edges of $[x,y]$ and $[x,z]$ which are contained in two cliques $\Delta,\Delta'$ in $\mathfrak h$. Up to reversing edges, let us assume $[u,v]$ and $[u',v']$ are parallel. It is clear that $x\in \pi_\Delta^{-1}(u)=\pi_{\Delta'}^{-1}(u')$ and $\pi_\Delta^{-1}(v)=\pi_{\Delta'}^{-1}(v')$ by Proposition \ref{clique-sectors}. The conclusion follows as $y\in \pi_\Delta^{-1}(v),z\in\pi_{\Delta'}^{-1}(v')$.
\end{proof}

Let $p,q$ be two geodesics originating from the same point.
Then $\vv{\mathcal S}(p)\subseteq \vv{\mathcal S}(q)$ if and only if  ${\mathcal H}(p)\subseteq {\mathcal H}(q)$ and each edge of $p$ is parallel to an edge of $q$. See Definition \ref{defnequivgeodrays}.

\begin{lem}\label{concatenationisgeodesic} 
Let $p=[o,x]$ and $q=[o,y]$ be two geodesic path so that $\vv{\mathcal S}(p)\subseteq \vv{\mathcal S}(q)$. Then  for any choice of geodesic $[x,y]$, the concatenation $p\cdot [x,y]$ is a geodesic path.
\end{lem} 
\begin{proof}
Indeed,  if  $\tilde p=p\cdot [x,y]$   is not a geodesic, then $\tilde p$ crosses  some hyperplane $ {\mathfrak h}\in{\mathcal H}$ at least twice. As $p$ and $[x,y]$ are both geodesics and then each crosses $ {\mathfrak h}$    at most once by Proposition \ref{geodesicwalls}, $ {\mathfrak h}$  thus separates $x$ and $y$. On the other hand, since $ {\mathfrak h}\in {\mathcal H}(p)\subseteq {\mathcal H}(q)$ and $p,q$ cross  through the same pair of sectors delimited by $\mathfrak h$, we obtain that $x,y$ lie in the same sector delimited by $\mathfrak h$ by Lemma \ref{insamesector}. This is a contradiction, as $ {\mathfrak h}$  separates $x$ and $y$.  Thus, the proof is complete.  
\end{proof}

\begin{rmk}
The above two results  generalize the following facts in median graphs: 
\begin{enumerate}
    \item 
    If   $p=[o,x]$ and $q=[o,y]$ are two geodesic paths with ${\mathcal H}(p)\subseteq {\mathcal H}(q)$, then $p\cdot [x,y]$ is a geodesic.
\item 
If $p=[o,x],q=[o,y]$ cross $\mathfrak h$ via the same pair of sectors delimited by $\mathfrak h$, then $x,y$ belong to the same sector. 
\end{enumerate}
\end{rmk}

The following relations between hyperplanes  will be useful in further discussion.
\begin{defn}\label{hyperplanesconfiguration}
Let $\mathfrak h,\mathfrak k$ be two hyperplanes in $X$. 
\begin{enumerate}[label=(\roman*)]
    \item $\mathfrak h$ and $\mathfrak k$ \textit{transverse} (write $\mathfrak h\pitchfork \mathfrak k$) if no sector of $\mathfrak h$  (resp. $\mathfrak k$) is contained in some sector of $\mathfrak k$ (resp. $\mathfrak h$). 
    \item 
    $\mathfrak h$ and $\mathfrak k$ are \textit{nested} if there exists sectors $\hat a\in \mathcal S(\mathfrak h), \hat b\in \mathcal S(\mathfrak k)$ so that   any $\hat c\in \mathcal S(\mathfrak k)\setminus \hat b$ is contained in  $\hat a$ and any $\hat c\in \mathcal S(\mathfrak h)\setminus \hat a$ is contained in  $\hat b$.  
    \item 
    $\mathfrak h$ and $\mathfrak k$ are  \textit{$L$-well-separated} for some $L\ge 0$ if the number of hyperplanes intersecting both of them is most $L$.  The $0$-well-separated hyperplanes are  called \textit{strongly separated}.
\end{enumerate}

\end{defn}

We note that the transversality and nestedness are the only two configurations on every pair of hyperplanes (without the other  ones from set theoretical considerations).   

\begin{lem}\label{nest equiv conditions}
Let $\mathfrak h\ne \mathfrak k$ be two   hyperplanes  in $X$. Let $(x,y)\in \Delta_h\times \Delta_k$ be the unique vertices  between two antipodal  cliques $\Delta_h\subset \mathfrak h$ and $\Delta_k\subset \mathfrak k$ in Definition \ref{defn parallel antipodal}.  Fix  $\tilde x\ne x\in \Delta_h$ and  $\tilde y\ne y\in \Delta_k$. Then the following are equivalent:
\begin{enumerate}[label=(\roman*)]
    \item  $\mathfrak h$  projects to $y$. 
    \item  $\mathfrak k$  projects to $x$. 
    \item The sector $\pi^{-1}_{\Delta_h}(\tilde x)$ delimited by $\mathfrak h$  is disjoint with the sector  $\pi^{-1}_{\Delta_h}(\tilde y)$  delimited by $\mathfrak k$. 
    \item The sector $\pi^{-1}_{\Delta_h}(x)$ delimited by $\mathfrak h$  contains the sector $\pi^{-1}_{\Delta_k}(\tilde y)$   delimited by $\mathfrak k$.   
\end{enumerate}
\end{lem}
\begin{proof}
We first elaborate on \textit{(i)}: if $\mathfrak h$  projects to $y$, then it projects to the  vertex $y'$ of  any clique $\Delta_k'$ of $\mathfrak k$ that  is parallel to $y\in \Delta_k$. Indeed, $\pi^{-1}_{\Delta_k}(y)=\pi^{-1}_{\Delta_k'}(y')$ by Proposition \ref{clique-sectors}. 
Since $\mathfrak h$  projects to $y$,  $\mathfrak h$ is  contained in $\pi^{-1}_{\Delta_k}(y)$. Thus, $\mathfrak h$ is  contained in $\pi^{-1}_{\Delta_k'}(y')$ and $\mathfrak h$  projects to $y'$. 

We now prove \textit{(i)} $\Leftrightarrow$ \textit{(ii)}.  Assuming \textit{(i)}, we shall prove \textit{(ii)}; the other direction is symmetric. If not, assume that some clique $\Delta_k'$ of $\mathfrak k$ projects to $x'\in \Delta_h\setminus x$. Let $z'\in \Delta_k'$ be the corresponding point to some $z \in \Delta_k\setminus y$. Noting $z'\in \pi_{\Delta_h}^{-1}(x')$ and $z\in \pi_{\Delta_h}^{-1}(x)$, any geodesic $[z',z]$ crosses the hyperplane $\mathfrak h$, that is, contains an edge of a clique $\Delta_h'$ parallel to $[x',x]$. This implies that $\Delta_h'$ projects  to $z'\in \Delta_k$ or $z\in \Delta_k$.  This contradicts \textit{(i)} that $\Delta_h$ projects to $y$. Thus, every clique $\Delta_k'$ of $\mathfrak k$ projects to $x$, and \textit{(ii)} follows.

We  prove \textit{(i)} $\Rightarrow$ \textit{(iv)}. Indeed,    as $\hat t:=\pi^{-1}_{\Delta_k}(\tilde y)$ intersects $\pi^{-1}_{\Delta_h}(x)$, let us assume by contradiction that $\hat t$ intersects another sector  $\pi^{-1}_{\Delta_h}(x')$ for some $x'\ne x\in \Delta_h$. By convexity, $\hat t$ contains some an edge $e$ parallel to $[x,x']$ of $\mathfrak h$. As $e$ projects to $z$, the clique containing $e$ does so. This contradicts (i) that   $\mathfrak h$ projects to $y$. Thus,  \textit{(i)} $\Rightarrow$ \textit{(iv)} follows.

It is clear that \textit{(iv)} $\Leftrightarrow$ \textit{(iii)}. It thus remains to prove \textit{(iii)} $\Rightarrow$ \textit{(i)}. By contradiction, assume that some clique $\Delta_h'$ of $\mathfrak h$ projects to some $y'\ne y\in \Delta_k$. Set $\hat t=\pi_{\Delta_k}^{-1}(y')$ and  $\hat s:=\pi_{\Delta_h}^{-1}(x')$ for some $x'\ne x\in \Delta_h$. Note that the sector $\hat s$ intersects $\pi_{\Delta_k}^{-1}(y)$ as $x'$ projects to $y$, and also intersects  $\hat t$ as $x'$ projects to $y'$. By convexity, $\hat s$  contains an edge $e$ of $\mathfrak k$ that is parallel to $[y,y']$. As $\hat t$ contains one vertex of $e$, this contradicts the assumption $\hat s\cap \hat t=\emptyset$.   The proof of \textit{(iii)} $\Rightarrow$ \textit{(i)} is complete.
\end{proof}

\begin{lem}\label{transversenest}
Let $\mathfrak h,\mathfrak k$ be two distinct hyperplanes in $X$. Then
\begin{enumerate}[label=(\roman*)]
    \item If $\mathfrak h$  does not transverse $\mathfrak k$, then $\mathfrak h$ and $\mathfrak k$ are nested. 
    \item 
    If $\mathfrak h$  transverses $\mathfrak k$, then any sector delimited by  $\mathfrak h$ intersects every sector delimited by  $\mathfrak k$.   
\end{enumerate}
\end{lem}
\begin{proof}
\textit{(i)}  Let $(x,y)\in \Delta_h\times \Delta_k$ be the unique vertices  between two antipodal  cliques $\Delta_h\subset \mathfrak h$ and $\Delta_k\subset \mathfrak k$ in Definition  \ref{defn parallel antipodal}. 
By Proposition \ref{clique-sectors}, write the sectors $\hat a:=\pi_{\Delta_h}^{-1}(x)$ delimited by $\mathfrak h$ and $\hat b=\pi_{\Delta_k}^{-1}(y)$ delimited by $\mathfrak k$. If $\mathfrak h$  does not transverse $\mathfrak k$,  then a sector  delimited by $\mathfrak h$   is contained in the sector $\hat b$.  By Lemma \ref{nest equiv conditions}, $\mathfrak h$ projects to $y$ and we then derive that $\hat b$ contains the sector $\pi_{\Delta_h}^{-1}(\tilde x)$ for every $\tilde x\ne x\in \Delta_h$. Similarly, $\hat a:=\pi_{\Delta_h}^{-1}(x)$ contains the sector $\pi_{\Delta_k}^{-1}(\tilde y)$ for every $\tilde y\ne x\in \Delta_k$. This verifies that $\mathfrak h$ and $\mathfrak k$ are nested.

\textit{(ii)} If  a sector delimited by $\mathfrak h$   is disjoint with some sector delimited by $\mathfrak k$, then  $\mathfrak h$  transverses $\mathfrak k$ by Lemma \ref{nest equiv conditions}.
\end{proof}

\begin{lem}\label{fernoslemma}
Let $\hat s\in \mathcal S({\mathfrak h})$ and $\hat t_1\ne \hat t_2 \in \mathcal S({\mathfrak k})$. Then $\hat t_1 \cap \hat s$ and $\hat t_2 \cap  \hat s$ are both nonempty if and only if ${\mathfrak h}\pitchfork {\mathfrak k}$ or ${\mathfrak k}\subset \hat s$. 
\end{lem}
\begin{proof}
For the $\Leftarrow$ direction:  if ${\mathfrak k}\subset \hat s$ then  $\hat t_1 \cap \hat s\ne\emptyset$ and $\hat t_2 \cap  \hat s\ne\emptyset$. Otherwise, if ${\mathfrak h}\pitchfork {\mathfrak k}$, then a sector of ${\mathfrak h}$ intersects every sector of ${\mathfrak l}$ by Lemma \ref{transversenest}. 

For the $\Rightarrow$ direction: let us assume that ${\mathfrak h}$ and ${\mathfrak k}$ are nested. Since $\hat t_1 \cap \hat s\ne\emptyset$ and $\hat t_2 \cap  \hat s\ne\emptyset$,  we deduce  that $\hat s$ gives  the sector $\hat a$ in Definition \ref{hyperplanesconfiguration}(ii): except one sector, all sectors delimited by ${\mathfrak k}$ is contained in $\hat s$. Of course, ${\mathfrak k}\subset \hat s$ follows.
\end{proof}
 
\subsection{Graph, combinatorial and Roller   compactifications} \label{subsec paraclique boundary}
This subsection  discusses several boundaries associated to paraclique graphs: 
\begin{enumerate}
\item
Klisse's graph boundary \cite{Kli20}, 
\item
Genevois' combinatorial boundary \cite{Gen20},
\item
Roller boundary \cite{Roller}, \cite[Section 7]{Gen22}.
\end{enumerate} 
The main result of this subsection is that all these boundaries are homeomorphic to the horofunction boundary when the graph boundary is visual.   

\subsubsection*{\textbf{Combinatorial compactification}}
Let $\alpha$ be an oriented geodesic in a paraclique graph $X$. Recall ${\mathcal H}(\alpha)$ denote the set of   hyperplanes that $\alpha$  crosses. Let $\vv{\mathcal S}(\alpha)$ denote the set of   sector \emph{ordered} pairs $(\hat s_1,\hat s_2)$ delimited by $\mathfrak h\in {\mathcal H}(\alpha)$ so that  the oriented $\alpha$  exits $\hat s_1$ and enters $\hat s_2$. 

\begin{defn}\label{defnequivgeodrays}
Let $\alpha,\beta$ be two oriented geodesic paths originating from $o\in X$. We say that $\alpha,\beta$ are \textit{equivalent} (write $\alpha\sim\beta$) if $\vv{\mathcal S}(\alpha)=\vv {\mathcal S}(\beta)$. That is to say, they cross the same set of hyperplanes through the same pair of sectors. More precisely, if $\alpha$ crosses an edge $e$ in a hyperplane $ {\mathfrak h}$, then $\beta$ crosses an edge $f$ in $ {\mathfrak h}$ so that $e$ and $f$ are parallel; the same holds for $\beta$ and $\alpha$.

\end{defn}


The \textit{combinatorial compactification} $\overline{(X,o)}_c$  of $X$ is defined as the set of all equivalent classes of \textit{oriented} geodesic paths from $o$. The vertices $x\in X$ could be seen as the union of geodesic segments $[o,x]$. So the equivalent classes of oriented geodesic segments are the same as the vertex set $X^0$. The equivalent classes of oriented geodesic rays is denoted by $\partial_{c} X$. We equip $\overline{(X,o)}_c=X^0\cup \partial_{c} X$  with the following metric. 

Let $\alpha,\beta$ be two non-equivalent geodesic paths from $o$, so the symmetric difference $\vv S(\alpha)\Delta \vv S(\beta)$ is non-empty. Define the distance  $$\delta(\alpha,\beta):=2^{-n}$$ where $n=\min\{d(o,\hat s_2):(\hat s_1,\hat s_2)\in \vv{\mathcal S}(\alpha)\Delta \vv{\mathcal S}(\beta)\}$.  

Note that $\partial_{c} X$ might not be compact, and we shall show  it is, when the graph compactification is visual defined below.  
\begin{lem}\label{localuniformconvergence}
Let $\alpha_n\to \alpha_\infty$ in the combinatorial boundary. Then there exist $\alpha_n'\sim \alpha_n$ and $\alpha_\infty'\sim \alpha_\infty$  so that $\alpha_n'$ converges to $\alpha_\infty'$ locally uniformly. That is, for any finite set $K$ in $X$,  $\alpha_n'\cap K=\alpha_\infty'\cap K$ for all but finitely many $n$.   
\end{lem}
\begin{rmk}
The converse is not true: $\alpha_n\to \alpha_\infty$ locally uniformly does not imply $\alpha_n\to \alpha_\infty$ in the combinatorial boundary. For example, consider the ladder graph with vertex set $\{(0,n),(1,n):n\in \mathbb Z\}$. Then the sequence of  geodesic segments  $[(0,0),(0,n)][(0,n), (1,n)]$ converges locally uniformly to the geodesic ray $\{(0,n):n\in \mathbb N\}$, but does not converge in $\partial_cX$.   
\end{rmk}
\begin{proof}[Proof of Lemma \ref{localuniformconvergence}]
Let $p$ be any finite initial segment of $\alpha_\infty$  ending at $p_+$.
As $\alpha_n\to \alpha_\infty$, we have $\vv{\mathcal S}(p)$ is contained in $\vv{\mathcal S}(\alpha_n)$ for all large $n\gg 0$. By Lemma \ref{concatenationisgeodesic}, the path $p\cdot [p_+,w]$ is a geodesic for all but finitely many $w\in \alpha_n$. So we could replace any finite initial segment of $\alpha_n$ with $p$ so that $\alpha_n$ start with $p$. As $p$ is arbitrary, we do the modification on $\alpha_n$ to produce $\tilde \alpha_n$ in the same equivalent class. The uniform convergence limit denoted as $\tilde\alpha_\infty$ of $\tilde \alpha_n$ is clearly also in $\alpha_\infty$, so the conclusion follows.       
\end{proof}

\subsubsection*{\textbf{Roller compactification}}
Analogous to the Roller boundary for $\cat$ cube complexes,  we define the  Roller boundary of a paraclique graph $X$ following an idea of Genevois \cite[Section 7]{Gen22}. We define an assignment 
$$
\begin{aligned}
\sigma:\quad &{\mathcal H} \longrightarrow 2^{\mathcal S( {\mathcal H})}\\  
 & {\mathfrak h}\longmapsto \mathcal S( {\mathfrak h})
\end{aligned}
$$
which shall be refereed to as an \textit{orientation} of hyperplanes ${\mathcal H}$. Namely, we choose one sector $\sigma( {\mathfrak h})\in \mathcal S( {\mathfrak h})$ for each hyperplane $ {\mathfrak h}\in {\mathcal H}$ so that  any finitely many chosen sectors intersect: for any $ {\mathfrak h}_1,\cdots,  {\mathfrak h}_n\in {\mathcal H}$ we have $\cap_{i=1}^n \sigma( {\mathfrak h}_i)\ne\emptyset$. 

Each vertex  $x\in X$ determines an orientation called \textit{principal} orientation by choosing the sector of each  $ {\mathfrak h}$ to contain $x$.
All orientations of hyperplanes form a closed subset denoted as $\overline X_{\mathcal R}$ in the product $\prod_{\mathfrak h\in \mathcal H}\mathcal S(\mathfrak h)$ with compact topology, which is called  the \textit{Roller compactification} of $X$. The \textit{Roller boundary}  $\partial_{R} X$ consists of  non-principal orientations.   

We can define the boundary of each sector $\hat s$ delimited by a hyperplane $\mathfrak h$, denoted as $\partial \hat s$. Namely, $\xi\in \partial_R X$ is a boundary point of the sector $\hat s$  if $\hat s$ appears in the image of the map $\sigma_\xi$ on $\mathfrak h$. In this terms, we have partitions $\overline X_{R}=\sqcup \{\hat s\cup \partial \hat s \in \mathcal S(\mathfrak h)\}$ for each $\mathfrak h$. 

If $\mathfrak h$ separates $X$ into two sectors, we then denote them by $\cev {\mathfrak h}$ and $\vec {\mathfrak h}$. In this case, $\partial_R X= \partial\cev {\mathfrak h}\sqcup \partial\vec {\mathfrak h}$.  

\subsubsection*{\textbf{Graph compactification}}
Following Klisse \cite{Kli20}, we define a compactification of  any  rooted (connected) graph $(X,o)$ with a basepoint $o\in X$, based on a partial order on $X$ defined below. 

A partial order on a set is a binary relation $\leq$ that is \textit{reflexive}, \textit{antisymmetric} and \textit{transitive}. A set with a partial order is called a \textit{partially ordered set (poset)}. The \textit{join} of a subset $Y$, if exists, is the least upper bound  of $Y$ denoted by $\lor Y$ so that $y\leq \lor Y$ for every $y\in Y$ and if $y\leq z$ for any $y\in Y$ then $z\leq \lor Y$. Similarly, the \textit{meet} of a subset $Y$, if exists, is the greatest lower bound  of $Y$ denoted by $\land Y$ so that $y\geq \land Y$ for every $y\in Y$ and if $y\geq z$ for any $y\in Y$ then $z\geq \land Y$. A poset is called a \textit{complete meet-semilattice} if any non-empty set has a meet.

We now define a \textit{graph order} on the rooted graph $(X,o)$. 
For $x,y\in X$, we declare $x\leq_o y$ if $x$ lies on some geodesic from $o$ to $y$; otherwise, $x\nleq_o y$, that is $d(o,y)>d(o,x)+d(x,y)$. Let $\mathbf x=(x_n)$ be a sequence of vertices in $X$. We extend the order by defining $x\leq_o \mathbf x$ if $x\leq_o x_n$ for all large enough $n$. Similarly, $x\nleq_o \mathbf x$ if  $x\nleq_o x_n$  holds for all large enough $n$. We say that $(x_n)$ \textit{$o$-converges} if given any $x\in X$, either $x\leq_o \mathbf x$ or $x \nleq_o \mathbf x$. If in addition, $\sup_{y\in X, y\leq_o \mathbf x} d(o,y)=\infty$, we say that $\mathbf x$ \textit{$o$-converges infinity}. 

\begin{defn}
Let $\mathbf x=(x_n)$ and $\mathbf y=(y_n)$ be two sequences that $o$-converge. We say that $\mathbf x$ and $\mathbf y$ are \textit{equivalent} (write $\mathbf x\sim \mathbf y$) if for any $x\in X$, we have $x\leq_o \mathbf x \Leftrightarrow x\leq_o \mathbf y$. 
\end{defn}

The compactification $\overline{(X,o)}$ is defined as follows. As a set, $\overline{(X,o)}$ consists of  all equivalent classes of  $o$-converging sequences. Any constant sequence in $X$ $o$-converges, so $X$ is contained in $\overline{(X,o)}$.  The boundary $\partial(X,o)$ is the subset of equivalent classes of sequences that $o$-converges to infinity.

We now define a subbase for the topology   on $\overline{(X,o)}$, which consists of the following family of subsets:   $$\mathcal U_x=\{z\in \overline{(X,o)}: x\leq_o z\} \bigand \mathcal U_x^c=\{z\in \overline{(X,o)}: x\nleq_o z\}$$
Equivalently, we may endow the topology in the following way. Every equivalent class $[\mathbf x]\in \overline{(X,o)}$ defines a map $\sigma_{\mathbf x}: y\in X\to \{0,1\}$ as follows:
$$
\sigma_{\mathbf x}(y)=\begin{cases}
1, & y\leq_o \mathbf x\\
0, & y\nleq_o \mathbf x
\end{cases}
$$
By definition of equivalence,  $[\mathbf x]\mapsto \sigma_{\mathbf x}$ is a well-defined injective map. 
\begin{rmk}\label{graphcompactificationiscompact}
If we view $\overline{(X,o)}$ as a  subset of $2^X$,  then it is  a closed subset in $2^X$ under product topology.    Indeed, let $\sigma_{\mathbf x_n}\in \overline{(X,o)}$ converges to $\sigma_\infty$ in $2^X$, which by definition  assigns to each $y\in X$  a value in $\{0,1\}$. We need to find a  sequence of vertices  $z_n$ in $X$ so that  $\sigma_\infty=\sigma_{(z_n)}$. This requires to run a Cantor's argument as follows. 

Let $Y$ be the subset of $y\in X$ with $\sigma_\infty(y)=1$, and list $X=\{y_1,y_2,\cdots, y_k,\cdots\}$.  Given $y_1\in X$, there exists $n_1$ so that either $y_1\leq_o \mathbf x_n$ or $y_1\nleq_o \mathbf x_n$ for any $n\ge n_1$. If $y_1\in Y$ (in the former case), let $z_1$ be any vertex in $\mathbf x_{n_1}$ so that $y_1\leq_o z_1$;  otherwise we do nothing.  Now for any $k\ge 1$ there exists $n_k>n_{k-1}$ so that either  $y_i\leq_o \mathbf x_n$ or $y_i\nleq_o \mathbf x_n$ for all $n\ge n_k$ and for all $1\le i\le k$. If $y_k\in Y$, let $z_k$ be any vertex in $\mathbf x_{n_k}$ so that $y_k\leq_o z_k$; otherwise we do nothing. In this way, we find a sequence of $z_n\in X$ so that, setting $\mathbf z:=(z_n)$, for any $y\in Y$, $y\leq_o \mathbf z$ and for any $y\in X\setminus Y$, $y\nleq_o \mathbf z$.  This verifies that $Y$ is the support of $\sigma_{\mathbf z}$, so  $\sigma_\infty=\sigma_{\mathbf z}$ follows.  
\end{rmk}

\begin{lem}\label{equivalent bounded o-converging sequence}
Let $\mathbf{x}=(x_n)$ and $\mathbf{y}=(y_n)$ be two $o$-converging sequences in $\overline{(X,o)}\setminus \partial(X,o)$. Then $\mathbf{x} \sim \mathbf{y}$ if and only if there exists a bounded subset $I$ of $X$ so that $I=\cap_{n\ge m} I(o,x_n)=\cap_{n\ge m} I(o,y_n)$ for all large $m\gg 0$.      
\end{lem}

\begin{proof}
The direction $\Rightarrow$ follows, since $\sigma_{\mathbf{x}}=\sigma_{\mathbf{y}}$ has the same bounded support $I$ for $\mathbf x \notin \partial(X,o)$. The other direction is by the same reasoning.   
\end{proof}

It is straightforward to verify that the subspace topology from $2^X$ is the same as the above topology on $\overline{(X,o)}$. If $X$ is countable, which is our standing assumption in this section,  $\overline{(X,o)}$ could be metrizable. Remark \ref{graphcompactificationiscompact} provides an alternative proof of the following result, which was proved using functional analysis considerations.
 
\begin{lem}\cite[Lemma 2.3]{Kli20}
The space  $\overline{(X,o)}$ is compact. If $X$ is a countable and locally finite, then $\overline{(X,o)}$ is metrizable, and $X$ is an open and dense subset in $\overline{(X,o)}$, and $\partial(X,o)=\overline{(X,o)}\setminus X$.     
\end{lem}

In the sequel,  if $o$ is understood in context, we shall write $x\leq y$ or $x\nleq y$ for $y\in \overline{(X,o)}$,   $\partial X=\partial(X,o),$ and $\overline X = \overline{(X,o)}$. 

\begin{defn}\label{def visual graph boundary}
We say that the graph compactification $\overline{(X,o)}$ is \textit{visual} if every equivalent class in $\overline{(X,o)}$ is represented by the vertex set of a (possibly finite) geodesic path.    
\end{defn}
\begin{rmk}
By definition, if $\mathbf {x}$ is a $o$-converging sequence to infinity, then there exists a geodesic ray $\gamma$ so that the vertex set of $\gamma$ is equivalent to $\mathbf {x}$. Otherwise,  there exists a vertex $v$ so that $\mathbf {x}$ is equivalent to the constant sequence $x$.  By Lemme \ref{equivalent bounded o-converging sequence}, if $X$ is locally finite, then the set $I$ is a finite set: indeed it is the interval set $I(o,x)$.

\end{rmk}

The following characterizes when the graph compactification is  visual. 
\begin{lem}\cite[Proposition 2.9]{Kli20}\label{finitejoin=geodesic}
Let $(X,o,\leq)$ be a connected rooted graph with the graph order.
Then the graph compactification $\overline{(X,o)}$ is {visual}  if and only if there are only finitely many $\leq$-minimal elements in $\mathcal U_x\cap \mathcal U_y$    for every $x,y\in X$.
In particular, if one of the above holds, then $\partial X=\overline X\setminus X$.
\end{lem}

Note that, if the join of any two elements $x,y$ exists, then $\mathcal U_x\cap \mathcal U_y=\mathcal U_{x \lor y}$. So if the graph order defines a complete meet-semilattice, then the graph compactification is visual.

The graph order on the Cayley graph of a Coxeter group is called weak order, which   defines the complete meet-semilattice by \cite[Theorem 3.2.1]{BB}. Thus, Coxeter groups have visual graph compactification  (\cite[Example 2.11]{Kli20}). We next verify the graph order on quasi-median graphs is a complete meet-semilattice. The general case for a  paraclique graph is left open.   
\begin{lem}\label{completemeetinquasimedian}
Let $(X,o,\leq)$ be a rooted quasi-median graph with the graph order. Then the meet of any non-empty subset $Y$ in $X$ exists. In particular, the join of any $\leq$-bounded set $Y$ exists.
\end{lem}
\begin{proof}
We first prove that the meet of any two elements exists. Given  $x,y\in X$, let us denote by $\mathcal S(o,\{x,y\})$  the set of the sectors containing $\{x,y\}$    delimited by hyperplanes separating $o$ and $\{x,y\}$. If $\mathcal S(o,\{x,y\})$ is empty,  we define the meet $x\land y=o$.  Let us now assume it is non-empty.  

Each sector in a quasi-median graph is gated (\cite[Corollary 2.22]{Gen17}), so the finite intersection $A=\cap \mathcal S(o,\{x,y\})$ of those sectors is a \emph{non-empty} gated set by \cite[Proposition 2.8]{Gen17}. We claim that the gate $\pi_A(o)$ of $o$ to $A$ is the meet of $x$ and $y$. 

Indeed, since $A$ is a gated set containing $x,y$, we obtain $\pi_A(o)\leq x$ and $\pi_A(o)\leq y$.  We now need show that $\pi_A(o)$ is the greatest lower bound on $x$ and $y$. That is, if $z\leq x$ an $z\leq y$ for some $z\in X$, then $z\leq \pi_A(o)$. Note that $z$ lies on a geodesic $[o,x]$ and on a geodesic  $[o,y]$. By replacing the subpaths from $o$ to $z$, we may assume that $[o,z]\subset [o,x]\cap[o,y]$. This implies that any hyperplane separating $o$ and $z$ must separate $o$ and $\{x,y\}$. Let $\mathfrak h$ be the hyperplane crossing the last edge $[w,z]$ of $[o,z]$. Then there exists a sector $S$ delimited by $\mathfrak h$ which  contains $\{x,y,z\}$ but not $o$. That is, $S\in \mathcal S(o,\{x,y\})$, so we obtain $z=\pi_S(o)$. 
By \cite[Corollary 2.40]{Gen17},  $\pi_A = \pi_A \cdot \pi_S$ holds   for gated subsets $A\subseteq S$. Thus, $z\leq \pi_A(o)$. Therefore, the meet $x\land y=\pi_A(o)$ exists.

By a standard argument, the existence of the meet for any non-empty set follows from that for two elements. In fact, let $A$ be a non-empty set. Let $x_0\in A$. If $x_0$ lies on a geodesic from $o$ to any $y\in A$, then $\land A=x_0$. Otherwise, there exists $y\in A$ so that $x_0$ is not on a geodesic $[o,y]$. Set $x_1=x_0\land y$. Since the distance $d(o,x_1)<d(o,x_0)$ strictly decreases, this process must be terminating  in finite steps and we arrive  at a vertex $x_n$ on $[o,x_0]$. By construction, $x_n$ is the meet $\land A$. The proof is complete.
\end{proof}

If $X$ is quasi-median and $\mathcal U_x\cap \mathcal U_y$ is non-empty, then $\mathcal U_x\cap \mathcal U_y=\mathcal U_{x \lor y}$ contains a unique minimal element by Lemma \ref{completemeetinquasimedian}. As a corollary of Lemma \ref{finitejoin=geodesic}, we obtain.
\begin{lem}\label{quasimedianisvisual}
The graph compactification $\overline{(X,o)}$ of a quasi-median graph $X$ at any root $o\in X$ is visual.     
\end{lem}

In the remainder of this subsection, we shall establish the homeomorphisms between the above compactifications, provided that the graph one is visual.

We first note the following elementary fact, which says that equivalence of geodesic rays is the same as the   equivalence,  in graph order, of the vertex sets on geodesic rays.

\begin{lem}\label{partialorder=samewalls}
Assume that $X$ is a paraclique graph.
Let $\mathbf{x}=(x_n)$ and $\mathbf{y}=(y_n)$ be the   vertex set on  two geodesic rays $\alpha,\beta$ respectively. Then $\alpha,\beta$ are {equivalent} if and only if each $x_n$ is  on a geodesic from $o$ to $y_m$ for all but finitely many $m$.  In particular, $\alpha\sim \beta \Leftrightarrow \mathbf x\sim \mathbf y$.  
\end{lem}
\begin{proof}
We first prove the direction $\Rightarrow$. Indeed, by assumption, $\vv{\mathcal S}(\alpha)=\vv{\mathcal S}(\beta)$, this implies that given $x_n$, $\vv{\mathcal S}([o,x_n])$ is a subset of $\vv{\mathcal S}([o,y_m])$ for all but finitely many $m$. Moreover, recalling  sectors are convex and a geodesic entering a sector will not exit it,  so if $[o,x_n]$ crosses a hyperplane ${\mathfrak h}$ through a pair of sectors, then  $[o,y_m]$ does so. Hence, $[o,x_n][x_n,y_m]$ is a geodesic  by Lemma \ref{concatenationisgeodesic} and the direction $\Rightarrow$ follows.

For the  direction $\Leftarrow$, given $x_n$,   $[o,x_n][x_n,y_m]$ is a geodesic for all $m\gg 0$. Then the hyperplane $[o,x_n]$ crosses must be crossed by $[o,y_m]$, and  $[o,y_n]$ enters the sectors $[o,x_n]$ enters into. Thus, $\vv\mathcal S([o,x_n])$ is a subset of $\vv{\mathcal S}([o,y_m])\subset \vv{\mathcal S}(\beta)$ for all $m\gg 0$. Letting $n\to \infty$, we obtain  $\vv{\mathcal S}(\alpha)\subset \vv{\mathcal S}(\beta)$. Similarly, we have $\vv{\mathcal S}(\beta)\subset \vv{\mathcal S}(\alpha)$. Thus, $\vv{\mathcal S}(\alpha)= \vv{\mathcal S}(\beta)$. 
\end{proof}

\begin{prop}\label{graphbdry=combdry}
Let $X$ be a  paraclique graph so that the graph compactification $\overline{(X,o)}$ is visual. Then the identification $x\in X^0\mapsto x\in X^0$  extends a homeomorphism from the graph compactification  $X^0\cup \partial X$  to the combinatorial compactification $X^0\cup \partial_c X$.\end{prop} 
\begin{proof}
By assumption, every boundary point $\mathbf x\in \partial X$ is represented by a geodesic ray $\alpha$. By Lemma \ref{partialorder=samewalls} the map assigning $\mathbf x \mapsto [\alpha]$  from $\partial X$ to $\partial_c X$ is well-defined and injective. The surjectivity is clear, as every geodesic ray defines a $o$-converging sequence.  It remains to prove the continuity. 

Let  $\mathbf x_n\in \partial X\to \mathbf x_\infty\in\partial X$ in the graph compactification. This means, for any finite set of points $z\in X$, $z\leq_o \mathbf x_n \Leftrightarrow z\leq_o \mathbf x_\infty$. Assume that $\mathbf x_n,\mathbf x_\infty$ are the corresponding vertex sets on  geodesic rays $\alpha_n, \alpha_\infty$. To prove $\alpha_n\to \alpha_\infty$ in the combinatorial compactification, it is better to argue by contradiction. If there exists two distinct pair of sectors $(\hat s,\hat t_1)$ and $(\hat s,\hat t_2)$ delimited by a hyperplane $\mathfrak h$ so that $\alpha_n$ exits $\hat s$ and enters $\hat t_1$, but $\alpha_\infty$ exits $\hat s$ and enters $\hat t_2\ne \hat t_1$. Let $z\in \alpha_\infty$ be the first vertex in the sector $\hat t_2$. By Proposition \ref{geodesicwalls}, $z$ is not on $\alpha_n$; otherwise the path $[o,z][z,w]$ would cross $\mathfrak h$ at least twice for infinitely many $w\in \alpha_n$. This contracts the convergence of $\mathbf x_n\to \mathbf x_\infty$. The continuity is proved and the proof is complete.
\end{proof}

\begin{prop}
Let $X$ be a  paraclique graph so that the graph compactification $\overline{(X,o)}$ is visual. Then the identification $x\in X^0\mapsto x\in X^0$ extends to a homeomorphism from the combinatorial compactification $X^0\cup \partial X$  to the  Roller  compactification  $X^0\cup \partial_R X$.    
\end{prop}
\begin{rmk}\label{rmk: genevois roller}
The homeomorphism between the combinatorial  and Roller boundaries is proved by Genevois \cite[Proposition A.2]{Gen20} for median graphs (i.e. $\cat$ cube complexes). The above result for mediangle graphs is anticipated by him in \cite[Section 7]{Gen22} without the above assumption. It is not clear to us whether Lemma \ref{completemeetinquasimedian} holds in mediangle graphs.    
\end{rmk}
\begin{proof} 
Let $\alpha$ be a geodesic path or ray issuing from $o$. If $ {\mathcal S}(\alpha)$ denotes the set of sectors into which  $\alpha$ eventually enters, then $ {\mathcal S}(\alpha)$ represents a point in the Roller compactification. The assignment $\pi([\alpha])={\mathcal S}(\alpha)$  defines a well-defined map from combinatorial compactification to Roller compactification. 

We now prove the continuity of $\pi$. Let $x_n\to [\alpha]$ in the combinatorial compactification. By Lemma \ref{localuniformconvergence}, we may assume that $[o,x_n]$ converges to $\alpha$ locally uniformly. It is then clear that $ {\mathcal S}([o,x_n])$ converges pointwise to $ {\mathcal S}(\alpha)$. This continuity follows.

At last,   the subjectivity follows from continuity of $\pi$. Indeed, for any $\xi\in \partial_R X$, we have $x_n\in X\to \xi$. Up to taking subsequence, assume that $x_n\to[\alpha]$ in the combinatorial compactification. By the continuity, $\pi(x_n)\to \pi([\alpha])$. Since the Roller compactification is metrizable, we see that $\pi([\alpha])=\xi$. In particular, there exists a geodesic ray $\alpha$ so that  ${\mathcal H}(\alpha)=\xi$.
\end{proof}

In \cite[Theorem 3.5]{Kli20}, Klisse proved that the graph boundary of a Coxeter group is visual and homeomorphic to the horofunction boundary. We now generalize this fact. The proof presented here  differs from his argument in several points, due to absence of  group actions on the graph.  
\begin{prop}
Let $X$ be a paraclique graph so that the graph compactification $\overline{(X,o)}$ is visual. Then the identification $x\in X^0\mapsto x\in X^0$ extends a homeomorphism from the graph compactification $X^0\cup \partial X$ to the horofunction compactification $X^0\cup \partial_h X$.    
\end{prop}
\begin{proof}
Let $\alpha,\beta$ be two equivalent geodesic rays from the basepoint $o$. We need to prove that they define the same horofunction boundary point. 

First of all, the vertex set $\mathbf x=(x_n)$ on $\alpha$ defines a horofunction $b_{\mathbf x}: X\to\mathbb R$ in a standard way $$\forall z\in X,\quad b_{\mathbf x}(z)=\lim_{n\to\infty} d(z,x_n)-d(o,x_n)$$ Similarly, we have the horofunction $b_{\mathbf y}(z)$ defined by the vertex set $\mathbf y=(y_n)$ on $\beta$. It suffices to prove that $b_{\mathbf x}(z)=b_{\mathbf y}(z)$ for any $z\in X$. 

By \cite[Lemma E.2]{B-O},  there exists a geodesic ray from $z$ flowing into $\beta$. That is, there exists $m_0$ so that $[z,y_{m_0}]\cdot [y_{m_0},y_m]$ is a geodesic for any $m\ge m_0$.  
Now by Lemma \ref{partialorder=samewalls}, since $\alpha\sim\beta$, any vertex $x_n$ on $\alpha$  lies on some geodesic $[o,y_m]$ for all $m\gg 0$ and thus  $d(o,x_n)+d(x_n,y_m)=d(o,y_m)$. Recall that $[z,y_{m_0}]\cdot [y_{m_0},y_m]$ is a geodesic for any $m\ge m_0$. This implies $d(z,x_n)+d(x_n,y_m)=d(z,y_m)$ for $m\gg n$, so we obtain  
$$
d(z,x_n)-d(o,x_n)=d(z,y_m)-d(o,y_m)
$$
which yields $b_{\mathbf x}(z)=b_{\mathbf y}(z)$ by taking the limit. 

\textbf{Injectivity}. We follow Klisse's argument in \cite[Theorem 3.5]{Kli20}. Let $\mathbf x\ne \mathbf y$ in $\partial X$. As $\mathbf x\ne \mathbf y$, there exists $z\in X$  so that $z\leq \mathbf x$ but $z\nleq \mathbf y$. On the one hand,  $z\leq \mathbf x$ implies that $z$ lies on $[o,x_n]$ for all $n\gg 0$, so   $b_{\mathbf x}(z)=-d(o,z)$. On the other hand,  $z\nleq  \mathbf y$ implies that $d(o,z)+d(z,y_n)>d(o,y_n)$ for all $n\gg 0$, so $b_{\mathbf y}(z)\ne -d(o,z)$.

\textbf{Continuity}. Let $x_n\in X$ be a sequence of points that converges to a point $\xi\in \partial X$. By assumption, $\xi$ is represented by a geodesic ray $\alpha$. By Lemma \ref{localuniformconvergence}, since $\partial X\cong \partial_c X$, we may choose a sequence of geodesics  $[o,x_n]$, which converges locally uniformly to $\alpha$. 

We need to show that for any $z\in X$, $b_{x_n}(z)\to b_\xi(z)$. Let us first consider that $z$ is a neighbor to $o$: $d(z,o)=1$. 

As $x_n\to \xi$ in $\partial X$, there exists $n_0$  (depending on $z$) so that  $z\leq_o x_n \Leftrightarrow z\leq_o \xi$ for all $n\ge n_0$. By \cite[Lemma E.2]{B-O},  there is a geodesic ray from  $z$ flowing into $\alpha$ at a vertex $u$,  so $u$ lies on $[o,x_n]$ for all $n\gg 1$. If $z\leq_o x_n$, then $b_{x_n}(z)=-1=b_{\xi}(z)$. Otherwise,   $z\nleq_o x_n$. We then have two cases: either $d(z,x_n)=d(o,x_n)+1$ or $d(z,x_n)=d(o,x_n)$. In the former case, $b_{x_n}(z)=1=b_{\xi}(z)$. In the latter case, since a paraclique graph (a clique-gated graph) satisfies the triangle condition by \cite[Theorem 3.1]{HK96}, there exists a common neighbor to $o$ and $z$ on the geodesic $[o,x_n]$, this implies $b_{x_n}(z)=0=b_{\xi}(z)$. In each case,  $b_{x_n}(z)=b_{\xi}(z)$ for all $n\ge n_0$.

The general case follows inductively. By induction,  we proceed to prove that $b_{x_n}(w)\to b_\xi(w)$ for any $w\in [o,z]$. Therefore, $b_{x_n}\to b_\xi$ pointwise.

\textbf{Surjectivity}. Once the map $\pi: \partial X\to \partial_h X$ is well-defined, the surjectivity follows from  the continuity. Given $\xi\in\partial_hX$, let $x_n\in X$ tend to  $\xi$. Up to taking a subsequence, we assume that $x_n\to \mathbf x$ for some $\mathbf x\in \partial X$. Then the continuity implies   $\pi(\mathbf x)=\xi$. The proof is complete.
\end{proof}

Summarizing the above discussion, we prove.
\begin{thm}\label{allboundaryaresame}
Let $X$ be a paraclique graph so that the graph compactification is visual. Then the graph boundary, combinatorial boundary and Roller boundary are homeomorphic to the horofunction boundary: $$\partial X \cong \partial_c X \cong \partial_R X\cong  \partial_h X$$ 
\end{thm}
We emphasize that the graph,   Roller, and horofunction compactifications are defined  for possibly non-locally finite graphs, so the corresponding boundaries could be non-compact.

\subsection{Finite symmetric difference partition on the Roller boundary}\label{subsec paraclique boundary partitions}
According to \textsection \ref{subsec horofunction compactfication},  the finite difference partition $[\cdot]$ is defined on the horofunction boundary $\partial_h X$.  We now equip the  Roller boundary $\partial_R X$ with a finite symmetric difference partition, which shall be  shown to be identical to the finite difference partition when $\partial_h X\cong \partial_R X$.

\subsubsection*{\textbf{Finite symmetric difference  relation on $\partial_R X$}}
Two orientations $\sigma,\sigma'\in \partial_R X$ are \textit{equivalent} if $\sigma(\mathfrak h)=\sigma'(\mathfrak h)$ for all but finitely many $\mathfrak h\in\mathcal H$. That is to say, $\sigma,\sigma'$ differ on at most finitely many sectors.
It is possible that two geodesic rays cross the same set of hyperplanes but differ on finitely many sectors. The equivalent class denoted  by $[\cdot]$ defines the \textit{finite symmetric relation} on the Roller boundary. 
 We say that a $[\cdot]$-class $[\xi]$ for $\xi\in \partial_R X$ is \textit{minimal} if it consists of only one point. 

Recall that $\mathcal S(\alpha)$ denotes the set of sectors that the geodesic ray $\alpha$ \emph{eventually} enters into.
 
\begin{lem} \label{finitedifferencesamefinitesymmetric}
Let $X$ be a paraclique graph so that the graph compactification is visual. 
Let $\alpha,\beta$ be  two geodesic rays in $X$ from the same initial point $o$ and ending at  points $\xi, \eta\in \partial_h X$ respectively. Denote by $b_\xi$ and $b_\eta$ be the associated horofunctions. Then
\begin{enumerate}[label=(\roman*)]
    \item If  $\|b_\xi-b_\eta\|<\infty$, then the symmetric difference $\mathcal S(\alpha)\Delta \mathcal S(\beta)$ is finite.
    \item 
    If the symmetric difference $\mathcal S(\alpha)\Delta \mathcal S(\beta)$ is finite, then $\|b_\xi-b_\eta\|<\infty$.
\end{enumerate}
\end{lem}
\begin{proof}
\textit{(i)} Assume that $|b_\xi(x)-b_\eta(x)|\le K$ for any $x\in X$.  Assume to the contrary that ${\mathcal S}(\alpha)\Delta{\mathcal S}(\beta)$ is infinite. That is, there are infinitely many distinct sectors $\hat s_i$ delimited by ${\mathfrak h}_i$ $(n\ge 1)$, which $\alpha$ enters but $\beta$ does not.  It is possible that $\{{\mathfrak h}_i: i\ge 1\}$ may contain repetitions: $\alpha, \beta$ the same ${\mathfrak h}_i$ but enter into distinct sectors.  

Write  $(y_n)$ for the vertex set on $\beta$ and then $b_\eta(x)=\lim_{n\to\infty} d(x,y_n)-d(o,y_n)$. Given any $x\in \alpha$, we have $b_\xi(x)=d(o,x)$ and then  for all $y_n$ on $\beta$ with $n\gg 0$, $$d(o,x)+d(x,y_n)-d(o,y_n)\le K$$ 
By Lemma \ref{geodesicwalls}, $d(o,y_n)$ is the number of   hyperplanes  separating $o$ and $y_n$. Such a hyperplane separates  either \begin{enumerate}
    \item $o$ and $x$: $o$ and $x$ lie in distinct sectors; or 
    \item  $x$ and $y_n$: $x$ and $y_n$ lie in distinct sectors; or 
    \item  $o$, $x$ and $y_n$: $o,x,y_n$ are contained in three distinct sectors.
\end{enumerate}  
The above inequality implies that there are at most $K$ hyperplanes, separating $o$, $x$ and $y_n$, of those separating $o$ and $y_n$.   

By assumption of $|{\mathcal S}(\alpha)\Delta{\mathcal S}(\beta)|=\infty$, for some large $n$, we can choose $x\in \alpha$ so that $\mathfrak h_1,\cdots,\mathfrak h_{2K+1}$  separate $o$ and $x$, but $[o,y_n]$ does not enter into the associated sectors $\hat s_i$. Note that a hyperplane in $\mathcal H([o,x])$ separates either   $o,y_n$ or  $o, x,y_n$. Thus, there are at least $K+1$ sectors $\hat s_i$ delimited by some $\mathfrak h_i$ (say, $1\le i\le K+1$, for definiteness) containing $x$ but not $y_n$. This implies that $\mathfrak h_1, \cdots, \mathfrak h_{K+1}$ does not separate $o,y_n$, and $x,y_n$. 
Since $d(o,x)+d(x,y_n)$ is the total number of hyperplanes in $\mathcal H([o,x])$ and of those in $\mathcal H([x,y_n])$, we see $d(o,x)+d(x,y_n)-d(o,y_n)>K+1$. This is a contradiction.

\textit{(ii)}  If the symmetric difference ${\mathcal S}(\alpha)\Delta{\mathcal S}(\beta)$ is at most $K$, we shall prove that $\alpha$ and $\beta$ have Hausdorff distance at most $K$. Indeed, let $x\in \alpha$ and $y\in \beta$ so that $d(x,y)=d(x,\beta)$. If ${\mathfrak h}$ is a hyperplane separating $x,y$, it  separates either, $o,x$ but not $o,y$, or $o,y$ but not $o,x$, or both  $o,x$ and $o,y$. In the last case, $x,y$ lie in distinct sectors. In summary, the sector delimited by  ${\mathfrak h}$ lies in ${\mathcal S}(\beta)\setminus {\mathcal S}(\alpha)$, or ${\mathcal S}(\alpha)\setminus {\mathcal S}(\beta)$. Hence, $d(x,y)\le |{\mathcal S}(\alpha)\Delta{\mathcal S}(\beta)|\le K$.
\end{proof}

Two hyperplanes are strongly separated if no hyperplane transverses both (Definition \ref{hyperplanesconfiguration}). Let $\partial \hat s$ denote the boundary of a sector $\hat s$ delimited by some ${\mathfrak h}$ in $\partial_R X$. 
The following generalizes \cite[Corollary 7.5]{Fer}, which shall be used to determine which $[\cdot]$-classes are minimal on boundaries of Coxeter groups.  
\begin{lem}\label{minimalcriterion}
Assume that $\{{\mathfrak h}_n:n\ge 1\}$ are pairwise strongly separated hyperplanes.   Let $\hat s_n\in \mathcal S({\mathfrak h}_n)$ be an infinite descending chain of  sectors delimited by some ${\mathfrak h}_n\in \mathcal H$. Then the intersection $\cap_{n\ge 1} (\partial \hat s_n\cup \hat s_n)$ is a singleton in the Roller boundary $\partial_R X$. 
\end{lem}
\begin{proof}
Since every finite intersection of these sectors $\hat s_n\in \mathcal S({\mathfrak h}_n)$ is non-empty, and $\overline X=X^0\cup\partial_R X$ is compact, $\cap_{n=1}^\infty (\hat s_n\cup \partial \hat s_n)$  is nonempty.

Assume to the contrary that $\xi,\eta\in \cap_{n=1}^\infty \partial \hat s_n$ are distinct. That is, $\xi$ and $\eta$ differ on some hyperplane ${\mathfrak k}$ and lie in the distinct sectors $\hat t_1\ne \hat t_2\in \mathcal S(\mathfrak k)$, so  $\xi\in \partial \hat  t_1$ and $\eta\in \partial \hat  t_2$. We derive from $\xi,\eta\in \cap_{n\ge 1} \partial \hat  s_n$ that $\hat t_1 \cap \hat s_n\ne \emptyset$ and $\hat t_2 \cap \hat s_n\ne \emptyset$ for each $n\ge 1$. 

By Lemma \ref{fernoslemma}, for any $\hat s\in \mathcal S({\mathfrak h}), \hat t_1\ne \hat t_2 \in \mathcal S({\mathfrak k})$ we have that $\hat t_1 \cap \hat s\ne \emptyset$ and $\hat t_2 \cap  \hat s\ne \emptyset$ if and only if ${\mathfrak h}\pitchfork {\mathfrak k}$ or ${\mathfrak k}\subset \hat s$. By strong separation, it follows that ${\mathfrak k} \subset \hat s_n$ for all $n$ sufficiently large. But this is impossible since $\hat  s_n$ is descending and there are finitely many sectors separating any two. Therefore, no such ${\mathfrak k}$ exists and $\xi =\eta$.
\end{proof}


\begin{cor}\label{squeezingraysareminimal}
Assume that a geodesic ray $\gamma$ intersects an infinite sequence of strongly separated hyperplanes pairs $(\mathfrak h_n,\mathfrak k_n)$. Then $\gamma$ intersects any infinite descending chain of strongly separated hyperplanes. In particular, the endpoint of $\gamma$ is minimal.    
\end{cor}
\begin{proof}
By definition, if $(\mathfrak h_n,\mathfrak k_n)$ and  $(\mathfrak h_{m},\mathfrak k_{m})$ for $m\ne n$ are both strongly separated, then $(\mathfrak h_n,\mathfrak h_m)$ strongly separated. Thus, the sequence  $\{\mathfrak h_n\}$ is pairwise strongly separated, and forms a desired descending chain.     
\end{proof}

\section{Applications to Coxeter groups}\label{sec: Coxeter combinatorial boundary}

We first  introduce basic notions related to Coxeter groups and refer the reader to, e.g.,  \cite{Dav} for more details, and then apply the results in the preceding section. 
\subsection{Coxeter groups}\label{subsec Coxeter groups}
A \textit{Coxeter group} $W$ is a group generated by a set $S$ subject to the following set of relations
\[W=\la S: (st)^{m_{st}}=e\text{ for any } s,t\in S\ra,\]
where $m_{ss}=1$, and $m_{st}\in \mathbb N_{\ge 2}\cup\infty$ for $s\ne t$, and  $m_{st}=\infty$ means there is no relation between $s$ and $t$. Let $\mathcal R$ denote the set of the conjugates $wsw^{-1}$ of $s\in S$ with $w\in W$ called \textit{reflections}. The pair $(W,S)$ is called a \textit{Coxeter system}. In this paper, we always assume that $S$ is finite, so  $W$ is finitely generated. 

Given $T\subset S$, the subgroup $W_T$ generated by $T$ is called a \textit{special subgroup} or a \textit{standard parabolic subgroup}. A subgroup $P\leq W$ is called \textit{parabolic} if it is a conjugate of a standard parabolic subgroup. In this term,  $W_S=W$.

The \textit{Coxeter diagram} for $(W, S)$ is a (non-directed) labeled graph where the vertex set is $S$ and edge set are all pairs $\{s, t\}$ such that $m_{st}\geq 3$. If $m_{st}\geq 4$, one labels the edge $\{s, t\}$ by $m_{st}$. The Coxeter group $(W, S)$ is said to be \textit{irreducible} if its Coxeter diagram is connected. Denote by $S_1,\dots, S_n$ the connected component of $S$ in the Coxeter diagram. It is straightforward to see that $W_S\simeq W_{S_1}\times W_{S_2}\times\dots\times W_{S_n}$. Therefore, $W_S$ is irreducible if and only if $W_S$ cannot be written as non-trivial direct products of special subgroups. Following \cite{Pa}, a Coxeter group $(W,S)$  is said to be of \textit{spherical} type if it is finite; it is of \textit{affine} type if $W=R\rtimes W_0$ for a finitely generated free abelian group $R$ and a finite Coxeter group $W_0$. 

\begin{lem}\cite[Proposition 4.3]{Pa}\label{irreducible Coxeter group E(W) trivial}
Assume that $(W,S)$ is irreducible and non-spherical. Then $W$ has no nontrivial  normal finite subgroup.  
\end{lem}

\subsubsection*{\textbf{Wall structure}}
Let $X(W,S)$ denote the Cayley graph of $W$ with respect to $S$. We denote by $d_X$  the combinatorial metric on $X(W,S)$. The Davis complex  $\Sigma(W,S)$ is a contractive polyhedron complex built from $X$ as one-skeleton, where cells are given by Coxeter polytopes. The Euclidean metric on Coxeter polytopes induces a $\cat$ metric denoted by $d_\Sigma$ on $\Sigma(W,S)$ on which $W$ acts by isometry and geometrically (see \cite[Chapter 12]{Dav}). By construction, the Davis complex $\Sigma(W,S)$ is a finite neighborhood of its one-skeleton  $X(W,S)$, so they are   quasi-isometric.
 
By \cite{CG25},  $X(W,S)$ is a paraclique graph (see \textsection\ref{subsec paraclique}).  Let $\mathcal H$ denote the set of hyperplanes which are the union of edges ($2-$cliques) in a parallel class.   There is one-one correspondence between $\mathfrak h\in \mathcal H$ and the reflections  $\mathcal R$:  $wsw^{-1}$  swaps the two endpoints of $[w,ws]$ and of all its parallel edges in the corresponding hyperplane. Each hyperplane $\mathfrak h$ separates $X(W,S)$ into two closed convex components denoted as $\cev{\mathfrak h}$ and $ \vec{\mathfrak h}$. Conversely, each hyperplane is preserved similarly by some reflection.  

There is a natural one-to-one correspondence between hyperplanes  in  $X(W,S)$ and a family of   walls in $\Sigma(W,S)$ which are the fixed point sets of  reflections $r\in \mathcal R$.  Namely, the set of edges in  $X(W,S)$  that are crossed   by  a wall forms  a hyperplane, and conversely, the middle points of edges of a hyperplane in  $X(W,S)$   are contained in a wall. Thus, two notions of hyperplanes are contained in a uniform finite neighborhood of the other in $\Sigma(W,S)$.  If there is no ambiguity, we  denote  by $\mathcal H$ the set of walls on $\Sigma(W,S)$ and we will use $\mathfrak h$ to denote hyperplanes in either space.   Abusing language, two hyperplanes in $\Sigma(W,S)$ are called $L_0$-well separated (resp. strongly separated) if the corresponding ones in $X(W,S)$  are $L_0$-well separated (resp. strongly separated) in Definition \ref{hyperplanesconfiguration}.

\subsubsection*{\textbf{Combinatorial boundaries}}
By Klisse \cite{Kli20} (more generally, Theorem \ref{allboundaryaresame}), Coxeter groups $(W,S)$ have homeomorphic boundaries of combinatorial nature: $$\partial X(W,S)\cong \partial_c X(W,S)\cong \partial_R X(W,S)\cong \partial_h X(W,S)$$  
By Proposition \ref{finitedifferencesamefinitesymmetric}, the finite difference relation on $\partial_h X(W,S)$ is exactly finite symmetric difference relation on $\partial_R X(W,S)$. The latter is clearly identical  to  the following block  relation using infinite reduced words as in \cite{LT15}.  

We say that a word over $S$ is  \textit{reduced} if it is a geodesic word in $X(W,S)$, and an infinite word $\mathbf i$ is   \textit{reduced} if any finite subword is reduced.  Let $Inv(\mathbf i)$ denote the set of hyperplanes  that it crosses. Two infinite reduced words $\mathbf i,\mathbf j$ are \textit{equivalent} if $Inv(\mathbf i)=Inv(\mathbf j)$. They are in the same \textit{block} denoted as $B[\mathbf i]$ if the symmetric difference  $Inv(\mathbf i)\Delta Inv(\mathbf i)$ is finite.  In a geometric terms, the set of infinite reduced words are exactly the set of geodesic rays issuing from the group identity.     So two equivalent infinite reduced words give two equivalent geodesic rays in the combinatorial boundary, and vice versa. Hence,  blocks of infinite reduced words induces the finite symmetric difference partition on the Roller boundary of $X(W,S)$.

\subsubsection*{\textbf{Contracting isometries}}

The existence of rank-one elements on the Davis complex has been characterized by Caprace and Fujiwara \cite{C-F}.  Recently, Ciobanu and Genevois proved that rank-one elements are exactly the contracting isometries on the Cayley graph (\cite[Theorem 5.1]{CG25}).

\begin{prop}\cite{C-F} \cite[Theorem 5.1]{CG25}\label{Coxeter group rankone}
   Let $(W,S)$ be a non-virtually cyclic Coxeter group. Then 
   \begin{enumerate}[label=(\roman*)]
       \item 
       $W$ contains a pair of independent rank-one elements on  $\Sigma(W, S)$.
       \item
       $W$ is non-spherical and non-affine. 
   \end{enumerate}
   Moreover, the set of rank-one elements are exactly the contracting isometries on $X(W,S)$.
\end{prop}

By Proposition \ref{Coxeter group rankone}, there will be no ambiguity to  talk about contracting isometries  in Coxeter groups. 
Let  $h$ be a rank-one isometry  on $\Sigma(W,S)$, which is also contracting on the Cayley graph $X(W,S)$. Let $\ax(h)$ denote a $\cat$ axis in $\Sigma(W,S)$ which is a bi-infinite geodesic in $\mathrm{Min}(h)$.
Let $\widetilde\ax(h):=E(h)\cdot o$ be a combinatorial quasi-geodesic in $X(W,S)$ for some basepoint $o\in X(W,S)$. 

Sometimes, it would be convenient to have a  $h$-invariant combinatorial axis $\widetilde\ax(h)$ in $X(W,S)$.  The following lemma explains this could be always done by taking appropriate power.
\begin{lem}\label{lem: combinatorial axis}
Let  $h$ be a contracting isometry  on $X(W,S)$. Then there exists an integer $n_0>0$ such that $h^{n_0}$ preserves a  combinatorial bi-infinite geodesic in $X(W,S)$.  
\end{lem}
\begin{proof}
By definition, $n\mapsto h^no$ is a contracting quasi-geodesic $\alpha$ in $X(W,S)$. By the contracting property of $\alpha$, any bi-infinite geodesic having finite Hausdorff distance with $\alpha$ is contained in the $R$-neighborhood of $\alpha$ for some fixed $R>0$. We then use an argument of Delzant \cite{Del96} to construct the desired axis. To this end, let us endow a total order on the generators of $S$. We say a combinatorial geodesic is \textit{special} if the labeling word of every subsegment is lexicographical minimal with respect to this order.  Let $\mathcal L$ be the set of special bi-infinite geodesics $\gamma$ with   finite Hausdorff distance  to $\alpha$. Note that $\mathcal L$ is left invariant under $h$, since $h$ preserves $\alpha$ and the $h$-image of a special geodesic is special. If $N$ denotes the number of elements in an $R$-ball, there are at most $N^2$ special geodesics connecting points in two $R$-balls. Thus, $|\mathcal L|\le N^2$, so for $n_0=(N^2)!$, $h^{n_0}$ fixes some combinatorial  geodesic in $\mathcal L$.  
\end{proof}

Note that  $\ax(h)$ and $\widetilde\ax(h)$ have finite $\cat$ Hausdorff distance in $\Sigma(W,S)$. Denote by $\pi_{\ax(h)}$ and $\pi_{\widetilde\ax(h)}$ be the shortest projection to $\ax(h)$ in $\Sigma(W,S)$ and $X(W,S)$ in either metric. 

The following result  is crucial in proving  that $h$ is a contracting isometry on $X(W,S)$, which we shall use later on. 
\begin{lem}\cite[Claims 5.3 and 5.4]{CG25}\label{well-separated hyperplanes}
Let  $h$ be a rank-one isometry  on $\Sigma(W,S)$. Then there are $L_0,D>0$ with the following property. The axis $\ax(h)$ intersects two disjoint and $L_0$-well-separated hyperplanes $\mathfrak k,\mathfrak k'$ so that their $\cat$ shortest projections $\pi_{\ax(h)}(\mathfrak k), \pi_{\ax(h)}(\mathfrak k')$ have diameter bounded above by $D$. 
\end{lem}

The following is the  main technical result of this subsection.
\begin{lem}\label{projectionsareclose}
There exists $R>0$ depending on $h$ so that for any $x\in X(W,S)$, $$\pi_{\ax(h)}(x)\subset N_R(\pi_{\widetilde\ax(h)}(x)) \bigand \pi_{\widetilde\ax(h)}(x)\subset N_R(\pi_{\ax(h)}(x))$$    
where the neighborhoods are took with respect to the  metric $d_X$ and $d_\Sigma$.
\end{lem}


\begin{proof}
By definition, the projection of any point to a contracting subset has uniformly bounded  diameter. Thus there exists $C>0$ so that for any $x\in X(W,S)$, $\pi_{\ax(h)}(x)$ and $\pi_{\widetilde\ax(h)}(x)$ have diameter bounded by $C$ in either metric. Since $d_\Sigma$ and $d_X$ are quasi-isometric,  it suffices to find  some $R>0$ so that $\pi_{\ax(h)}(x)$ is contained in an $R$-neighborhood of $\pi_{\widetilde\ax(h)}(x)$ in $d_X$-distance. 

First of all, since $h$ acts by translation on $\ax(h)$, we may choose  by Lemma \ref{well-separated hyperplanes} an infinite periodic sequence of  pairwise disjoint and $L_0$-well-separated hyperplanes $\mathfrak k_n$ ($n\in \mathbb Z$) so that $$h^l \cdot \mathfrak k_n=\mathfrak k_{n+1}$$ for some $l>0$. Here the hyperplanes $\mathfrak k_n$ are understood in the Davis complex.  Let $z_n=\mathfrak k_n\cap \ax(h)$ be the intersection point (since a $\cat$ geodesic $\ax(h)$ intersects a $\cat$ hyperplane in one point, if at all).  Since the sequence is periodic and $d_\Sigma$, $d_X$ are quasi-isometric, there exists a constant $R_0$ depending on $l$ so that  $$\max\{d_X(z_n,z_{n+1}),d_\Sigma(z_n,z_{n+1})\}<R_0$$

\begin{figure}
    \centering

\tikzset{every picture/.style={line width=0.75pt}} 

\begin{tikzpicture}[x=0.75pt,y=0.75pt,yscale=-1,xscale=1]

\draw    (68,172) -- (445,172.99) ;
\draw [shift={(447,173)}, rotate = 180.15] [color={rgb, 255:red, 0; green, 0; blue, 0 }  ][line width=0.75]    (10.93,-3.29) .. controls (6.95,-1.4) and (3.31,-0.3) .. (0,0) .. controls (3.31,0.3) and (6.95,1.4) .. (10.93,3.29)   ;
\draw    (117,72) -- (119,257) ;
\draw    (198,71) -- (200,256) ;
\draw    (273,74) -- (275,259) ;
\draw    (345,73) -- (347,258) ;
\draw    (147,101) .. controls (164,186) and (372,149) .. (374,98) ;
\draw [shift={(374,98)}, rotate = 272.25] [color={rgb, 255:red, 0; green, 0; blue, 0 }  ][fill={rgb, 255:red, 0; green, 0; blue, 0 }  ][line width=0.75]      (0, 0) circle [x radius= 3.35, y radius= 3.35]   ;
\draw [shift={(147,101)}, rotate = 78.69] [color={rgb, 255:red, 0; green, 0; blue, 0 }  ][fill={rgb, 255:red, 0; green, 0; blue, 0 }  ][line width=0.75]      (0, 0) circle [x radius= 3.35, y radius= 3.35]   ;

\draw (139,84.4) node [anchor=north west][inner sep=0.75pt]    {$x$};
\draw (376,82.4) node [anchor=north west][inner sep=0.75pt]    {$y$};
\draw (119,71.4) node [anchor=north west][inner sep=0.75pt]    {$\mathfrak k_{0}$};
\draw (200,72.4) node [anchor=north west][inner sep=0.75pt]    {$\mathfrak k_{1}$};
\draw (276,72.4) node [anchor=north west][inner sep=0.75pt]    {$\mathfrak k_{2}$};
\draw (348,73.4) node [anchor=north west][inner sep=0.75pt]    {$\mathfrak k_{3}$};
\draw (100,173.4) node [anchor=north west][inner sep=0.75pt]    {$z_{0}$};
\draw (183,174.4) node [anchor=north west][inner sep=0.75pt]    {$z_{1}$};
\draw (257,173.4) node [anchor=north west][inner sep=0.75pt]    {$z_{2}$};
\draw (329,173.4) node [anchor=north west][inner sep=0.75pt]    {$z_{3}$};
\draw (199,135.4) node [anchor=north west][inner sep=0.75pt]    {$w_{1}$};
\draw (274,135.4) node [anchor=north west][inner sep=0.75pt]    {$w_{2}$};
\draw (346,128.4) node [anchor=north west][inner sep=0.75pt]    {$w_{3}$};
\draw (379,177.4) node [anchor=north west][inner sep=0.75pt]    {$Ax( h)$};

\end{tikzpicture}
    \caption{Proof of Lemma \ref{projectionsareclose}}
    \label{fig:projectionsareclose}
\end{figure}
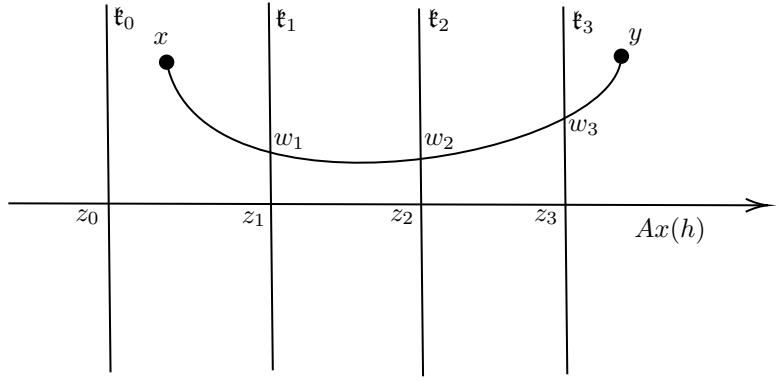

\begin{Claim}\label{clm: bdd projection}
For each $i\in \mathbb Z$, the $\cat$ projection of $\vec{\mathfrak k}_i\cap \cev{\mathfrak k}_{i+1}$  to $\ax(h)$  has diameter at most $(D+2R_0)$.     
\end{Claim}
\begin{proof}[Proof of the Claim \ref{clm: bdd projection}]
Let us only prove the case $i=0$; the other cases are the same.
Indeed, by Lemma \ref{well-separated hyperplanes},  $\mathfrak k_0$ (resp. $\mathfrak k_1$)  projects to $\ax(h)$ as a bounded set of diameter at most $D$, so the $\cat$ projection of the halfspace $\vec{\mathfrak k}_0$ to $\ax(h)$ is contained in the $D$-neighborhood of the positive half-ray $[z_0,\ax(h)^+]_{\ax(h)}$. Similarly, the $\cat$  projection of $\cev{\mathfrak k}_1$ to $\ax(h)$ is contained in the $D$-neighborhood of $[\ax(h)^-,z_1]_{\ax(h)}$.   Recalling  $d_X(z_0,z_{1})\le R_0$ from above, the $\cat$ projection of $\vec{\mathfrak k}_0\cap \cev{\mathfrak k}_1$ has diameter at most $(2R_0+D)$.  
\end{proof}

By the periodicity of $\{\mathfrak k_n:n\in \mathbb Z\}$, up to $h$-translation we may assume  that $x$ lies between the two hyperplanes $\mathfrak k_0$ and $\mathfrak k_1$. That is, $x\in \vec{\mathfrak k}_0\cap \cev{\mathfrak k}_1$. See Figure \ref{fig:projectionsareclose} for illustration.  

Let $z\in \pi_{{\ax}(h)}(x)$ be a $\cat$ projection point. By the Claim \ref{clm: bdd projection},  since $x\in \vec{\mathfrak k}_0\cap \cev{\mathfrak k}_1$,  $z$  is contained in the $(D+2R_0)$-neighborhood of $z_1$ in $\cat$ metric. For notational simplicity, since $d_X$ and $d_\Sigma$ are quasi-isometric, we may assume $d_X(z,z_1)\le 2R_0+D$ by increasing the constant.

From now on, it would be better to understand the hyperplanes $\mathfrak k_n$  in the Cayley graph $X(W,S)$. The hyperplanes are the union of parallel edges in the sense of paraclique graphs, so every combinatorial geodesic  intersects $\mathfrak k_n$ in exactly two vertices.  

We will use the following key fact. Similar consideration has appeared in \cite[Lemma 6.1]{C-S}. 
\begin{Claim}\label{clm: intersect separated hyperplanes}
Let $[x,y]$ be a combinatorial  geodesic in $X(W,S)$ which intersects $L_0$-well separated hyperplanes $\mathfrak k_1,\mathfrak k_2,\mathfrak k_3$. Choose three points  $w_i\in \mathfrak k_i\cap [x,y]$ with $1\le i\le 3$. Then $$d_X(z_2,w_2)\le R_1:=L_0+2R_0.$$      
\end{Claim} 
\begin{proof}[Proof of the Claim \ref{clm: intersect separated hyperplanes}]
Indeed, by Proposition \ref{geodesicwalls}, the distance $k:=d_X(z_2,w_2)$ equals  the number of  hyperplanes   separating $z_2$ and $w_2$. According to the $L_0$-well separated hyperplanes assumption, at most    $L_0$  of those hyperplanes   transverse $\mathfrak k_1,\mathfrak k_3$. Thus, there are at least $(k-L_0)$  of the remaining ones  which    transverse only one of them. Let $\mathfrak h$ denote  such a hyperplane and assume $\mathfrak h$  does not transverse  $\mathfrak k_1$; the other case is symmetric.   By Proposition \ref{geodesicwalls}, a hyperplane cannot be crossed twice by the combinatorial geodesics $[w_1,w_3]$ and $[z_1,z_3]$. If $\mathfrak h$ does not transverse the first half $[w_1,w_2]$ of $[w_1,w_3]$, then it must transverse the second  half $[z_2,z_3]$ of $[z_1,z_3]$; otherwise, it must transverse the first half $[z_1,z_2]$.  Hence, the number of those remaining hyperplanes $\mathfrak h$ is bounded above by $d_X(z_1,z_3)\le 2R_0$. Since this number is lower bounded by $k-L_0$ as above, we deduce that  $d_X(z_2,w_2)=k\le (2R_0+L_0)=R_1$. 
\end{proof}

Let us conclude the proof. Let $y\in \pi_{\widetilde{\ax}(h)}(x)$ be a shortest projection in combinatorial metric.  Assume   first that a combinatorial geodesic $[x,y]$ crosses at most $2$ well-separated hyperplanes from $\{\mathfrak k_i: i\in \mathbb Z\}$. For definiteness,  assume that $[x,y]$ crosses $\mathfrak k_1, \mathfrak k_2$, so $[x,y]$ is contained in the union $(\vec{\mathfrak k}_0\cap \cev{\mathfrak k}_1)\cup (\vec{\mathfrak k}_1\cap \cev{\mathfrak k}_2) \cup(\vec{\mathfrak k}_2\cap \cev{\mathfrak k}_3)$. Then,  $d_\Sigma(z,y)\le D+2R_0$ by   Claim \ref{clm: bdd projection}.  Otherwise,  $[x,y]$ crosses at least $3$ well-separated hyperplanes. Let us  say that it crosses $\mathfrak k_1,\mathfrak k_2,\mathfrak k_3$. Thus, the above Claim \ref{clm: intersect separated hyperplanes} implies $d_X(z_2,w_2)\le R_1$. By $d_X$-shortest projection, we have $d_X(y,w_2)\le d_X(w_2,z_2)$  and thus $d_X(y,z_2)\le d_X(y,w_2)+ d_X(w_2,z_2)\le 2R_1$. By assumption, $d_X(z_1,z_2)\le R_0$. Hence, in this case,  $d_X(y,z)\le 2R_0+2R_1+D$.  Setting $R=2R_0+2R_1+D$ completes the proof.
\end{proof}
 
\subsection{Lam-Thomas partition on the Tits boundary}
In this subsection, we recall a partition of the Tits boundary $\partial_T \Sigma(W,S)$ of the Davis complex $\Sigma(W,S)$ defined by Lam and Thomas \cite{LT15}. 

First, we introduce some terms.
A hyperplane $\mathfrak k$ separates the Davis complex  $\Sigma(W,S)$ into two \textit{closed} $\cat$ convex components denoted by $\cev{\mathfrak k}$ and $\vec{\mathfrak k}$. That is, $\cev{\mathfrak k}\cup \vec{\mathfrak k}=\Sigma(W,S)$ and $\cev{\mathfrak k}\cap \vec{\mathfrak k}=\mathfrak k$. We say $\xi, \zeta$ are separated by $\mathfrak k$ if $\xi \in \partial\cev{\mathfrak k}\setminus \partial\mathfrak k$ and $\zeta \in \partial\vec{\mathfrak k}\setminus \partial\mathfrak k$.

Let $\partial {\mathfrak k}$, $\partial \cev{\mathfrak k}$ $\partial \vec{\mathfrak k}$  denote their  boundary in  $\partial_\infty\Sigma(W,S)$, which are respectively the set of accumulation points of  ${\mathfrak k},\cev{\mathfrak k},\vec{\mathfrak k}$ respectively in the  visual compactification. Accordingly,  $\partial_\infty\Sigma(W,S)=\partial \cev{\mathfrak k}\cup \partial \vec{\mathfrak k}$ and  $\partial \cev{\mathfrak k}\cap \partial \vec{\mathfrak k}=\partial \mathfrak k$.

\begin{defn}\cite[Definition 4.3]{LT15}
Two points $\xi, \zeta\in \partial_T \Sigma(W,S)$ are \textit{equivalent} (written as $\xi\sim \zeta$) if for every hyperplane $\mathfrak k$, either $(\xi \in \partial \cev{\mathfrak k} \Leftrightarrow \zeta\in \partial \vec{\mathfrak k})$ or $(\xi \in \partial \vec{\mathfrak k} \Leftrightarrow \zeta\in \partial \cev{\mathfrak k})$.  In other words, $\xi\sim \zeta$ if and only if they are not separated by any hyperplane.  Denote by $\mathcal C(\xi)$  the equivalent class.
\end{defn}
\begin{rmk}\label{rmk defn Wxi}
We first clarify a bit the definition, and then introduce an important group $W(\xi)$ associated to each $\mathcal C(\xi)$.  
\begin{enumerate}
    \item The underlying set of the  Tits boundary (with a finer topology)  for a $\cat$ space is the same as that of visual boundary. Most times we are  concerned only with the barely  set without explicit topology (with exception in Proposition \ref{LamThomasPartition}(2)). 
   
    \item 
    By definition, if $\xi\sim  \zeta$, then $\xi$ and $\zeta$ are contained simultaneously in the same set of boundaries of hyperplanes (which may be an empty set). Thus,  if $W(\xi)$ is the subgroup   generated by the set of reflections in $W$ about hyperplanes which contain $\xi$ in their boundaries, then $W(\xi)=W(\zeta)$. If $\xi$ is not contained in any boundary of hyperplanes, we define $W(\xi)=\{1\}$.
\end{enumerate}
   
\end{rmk}

In \cite{LT15}, a one-to-one correspondence is set up between the partition  $\{C(\xi): \xi \in \partial_T \Sigma(W,S)\}$ and the blocks $B[\mathbf i]$ of infinite reduced words in $W$.  Namely, let   the group identity $o\in W$ be the basepoint.  If $\xi\sim \zeta$, the $\cat$ geodesic ray $[o,\xi]$ in $\Sigma(W,S)$ crosses the same infinite sequence of walls as  $[o,\zeta]$ and thus defines an equivalent class of infinite reduced words $\mathbf i(\xi)$ (i.e. $Inv(\mathbf i(\xi))=Inv(\mathbf i(\zeta))$). Conversely, given $\mathbf i$, let $\partial_T \Sigma(\mathbf i)$ denote  the set of boundary points $\xi\in \partial_T \Sigma(W,S)$ so that the hyperplanes crossed by the $\cat$ geodesic ray $\gamma=[o,\xi]$ are exactly $Inv(\mathbf i)$.  Two  words $\mathbf i\sim \mathbf j$ in the same block define the same set $\partial_T \Sigma(\mathbf i)=\partial_T \Sigma(\mathbf j)$.  Lam-Thomas proved that the maps $\xi\mapsto B[\mathbf i(\xi)]$ and $\mathbf i\mapsto \partial_T \Sigma(\mathbf i)$ induce the inverse of the other on the partitions on $\partial_h X(W,S)$ and $\partial_T \Sigma (W,S)$.

Moreover, since $\partial_T \Sigma(\mathbf i)$ coincides $\mathcal C(\xi)$ for $\xi\in \partial_T \Sigma(\mathbf i)$, Lam-Thomas \cite{LT15} further proved that the elements in block $B[\mathbf i]$ is bijective to the elements of the subgroup $W(\xi)$.

For further reference, we summarize the above discussion as   follows.
\begin{prop}\label{LamThomasPartition}
There exists one-to-one correspondence between the partition on $\partial_h X(W,S)$ and the Lam-Thomas's partition on $\partial_T \Sigma (W,S)$:
$$
B[\mathbf i] \longleftrightarrow \partial_T \Sigma(\mathbf i)
$$
with the following properties
\begin{enumerate}[label=(\roman*)]
    \item 
    $\partial_T \Sigma(\mathbf i)$ coincides $\mathcal C(\xi)$ for $\xi\in \partial_T \Sigma(\mathbf i)$.
    \item 
    $\mathcal C(\xi)$ is a path-connected and totally geodesic subset in $\partial_T \Sigma(\mathbf i)$.
    \item 
    The block $B[\mathbf i]$ is in bijection with the set of elements in $W(\xi)$. 
\end{enumerate}      
\end{prop}

The finite difference partition is generally non-trivial on the horofunction boundary. That is, certain boundary points have non-singleton equivalent class. For instance, it is easy to see that a direct product of two nontrivial  groups (with the union of  generating sets in two factors) has nontrivial finite difference partition on the horofunction boundary. Thus, a reducible Coxeter group always has nontrivial finite difference partition. Using Klisse's result, we can characterize the case for irreducible Coxeter groups.

\begin{lem}\label{charactrize nontrivial finite difference}
Assume that  $W$ is irreducible. Then the finite difference partition is trivial if and only if  $W$ is a hyperbolic group with the Gromov boundary homeomorphic to the graph  boundary.  
\end{lem}
\begin{proof}
If the finite difference partition on $\partial_h X(W,S)$ is trivial,  then  the horofunction  boundary is \textit{small at the infinity} in the following sense: if $x_n \in W$ tends to $\xi\in \partial_h X(W,S)$ with $\sup d(x_n,y_n)<\infty$, then $y_n\to \xi$. By \cite[Theorem 3.14]{Kli20}, the graph  boundary is small at the infinity if and only if $W$ is a hyperbolic group with Gromov boundary homeomorphic to graph  boundary.  Thus, the $\Rightarrow$ direction follows.   The other direction follows from the fact that the horofunction boundary modding out the finite difference partition recovers the Gromov boundary. Thus, the finite difference partition is trivial. 
\end{proof}

\subsection{North-south dynamics for contracting isometries in irreducible Coxeter groups} 

We first prove that  certain  contracting isometries in irreducible non-affine  Coxeter groups admit north-south dynamics on the horofunction boundary (Definition \ref{NSDynamicDef}).
\begin{lem}\label{NorthSouthDynamicsCoxeter}
Let $(W,S)$ be an  irreducible non-spherical non-affine Coxeter group. Then there exists a contracting isometry $h\in W$ on $X(W,S)$ so that
\begin{enumerate}[label=(\roman*)]
    \item Any $\cat$ axis  $\ax(h)$ crosses a pair of  strongly separated hyperplanes in $\Sigma(W,S)$. 
    \item The $[\cdot]$-classes of fixed points $[h^+], [h^-]$ in $\partial_h X(W,S)$ are  singletons.
\end{enumerate} In particular,  $h$ admits north-south dynamics on the horofunction boundary.  
\end{lem}
Recall that the hyperplanes in $X(W,S)$ and in $\Sigma(W,S)$ are naturally identified so that they have a uniform finite Hausdorff distance in $\Sigma(W,S)$. Two hyperplanes in $\Sigma(W,S)$ are called $L_0$-well separated (resp. strongly separated) if the corresponding ones in $X(W,S)$  are $L_0$-well separated (resp. strongly separated) in Definition \ref{hyperplanesconfiguration}.

\begin{proof}
By Proposition \ref{Coxeter group rankone}, $W$ contains contracting elements, so it is acylindrically hyperbolic by a result of Sisto \cite{Sisto2}. Since $E(W)$ is trivial by Lemma \ref{irreducible Coxeter group E(W) trivial}, there exists  a contracting isometry $h$ so that the maximal elementary group $E(h)$ is cyclic (\cite[Corollary 5.7]{Hull}). See the proof of Lemma \ref{lem: specialrank1} for more details. 

We shall prove that any such contracting isometry   $h$ with cyclic $E(h)$ are desired ones. Let $\ax(h)$ be the $\cat$ axis of $h$ in $\Sigma(W,S)$.   We claim that $\ax(h)$ crosses a pair of  strongly separated hyperplanes. Here hyperplanes are understood in the Davis complex $\Sigma(W,S)$.   

Indeed, let $(\mathfrak h_1,\mathfrak h_1')$ be a pair of $L_0$-well separated hyperplanes given  by Lemma \ref{well-separated hyperplanes}. Set $\mathfrak h_n:=h^{-n}\mathfrak h_1$ and $\mathfrak h_n':=h^{n}\mathfrak h_1'$ for $n\ge 1$. Then  we may extract  from $\{(\mathfrak h_n,\mathfrak h_n'):n\in\mathbb Z\}$  an infinite   sequence of distinct $L_0$-well separated hyperplanes pairs still denoted as $(\mathfrak h_n,\mathfrak h_n')$  so that  $\ax(h)$ crosses $\mathfrak h_n,\mathfrak h_n'$ and $d_\Sigma(\mathfrak h_n, \mathfrak h_n')\to \infty$. Arguing by contradiction and up to taking a further subsequence, we may assume that   $(\mathfrak h_n,\mathfrak h_n')$ are not strongly separated for each $n$.  That is, there exists a hyperplane $\mathfrak k_n$  for each $n\ge 1$ that intersects $\mathfrak h_n,\mathfrak h_n'$. Pick up  intersection points $x_n\in \mathfrak k_n\cap \mathfrak h_n, y_n\in  \mathfrak k_n\cap\mathfrak h_n'$, and denote by $x_n',y_n'$ the corresponding shortest $\cat$ projections on $\ax(h)$. Since the projection of $\mathfrak h_n$ to $\ax(h)$ has diameter at most $D$ by Lemma \ref{well-separated hyperplanes},  $d_\Sigma(\mathfrak h_n, \mathfrak h_n')\to \infty$ implies $d_\Sigma(x_n',y_n')\to\infty$.  Since the hyperplane $\mathfrak k_n$ is convex, the $\cat$ geodesics $[x_n,y_n]$ are contained in  $\mathfrak k_n$. 

Let  $C>0$ denote the contracting constant of $\ax(h)$.  By the contracting property in Lemma \ref{char contracting property},  since $d_\Sigma(x_n',y_n')\to\infty$, we see that  $[x_n,y_n]$ intersects the $2C$-neighborhood of the projections $x_n',y_n'$, and thus of   $N_{2C}(\ax(h))$.  Applying the $h$-translation on $\ax(h)$, we may assume that these hyperplanes $\mathfrak k_n$ intersect the $2C$-ball around a fixed point (e.g., the basepoint $o$). Since $X$ is locally finite there are only finitely many hyperplanes intersecting the $2C$-ball. Hence,  $\{\mathfrak k_n:n\ge 1\}$ is a finite set, so it contains  a hyperplane denoted as  $\mathfrak k$ which must intersect $N_{2C}(x_n')$ and $N_{2C}(y_n')$. Note that $\ax(h)$ is contracting and thus Morse by Lemma \ref{lem:Morse}. Since $d_\Sigma(x_n',y_n')\to\infty$, we deduce from the Morse property of $\ax(h)$ that  $\ax(h)$ lies in the $R$-neighborhood of the hyperplane $\mathfrak k$ for some $R$ depending on $C$. 

To conclude the proof of \textit{(i)}, we use the following specific property of Coxeter groups: each hyperplane in $\Sigma(W,S)$ is fixed pointwise by a reflection $r$ (and vice versa). Thus, $r\ax(h)$ has finite Hausdorff distance at most $2R$ with $\ax(h)$, which implies that $r$ belongs to $E(h)$. This gives a contradiction, as the infinite  cyclic subgroup $E(h)$ is torsion-free. Hence,  $\ax(h)$ crosses a pair of  strongly separated hyperplanes with bounded projection to $\ax(h)$.

We next prove that  $[h^+]$ is minimal. By taking a high power of $h$, we may assume that $h$ leaves invariant a combinatorial bi-infinite geodesic $\widetilde{\ax}(h)$ by Lemma \ref{lem: combinatorial axis}. Since  $\widetilde{\ax}(h)$ is contained in a finite $\cat$ neighborhood of $\ax(h)$, we deduce the same conclusion that $\widetilde{\ax}(h)$ crosses a pair of  strongly separated (combinatorial) hyperplanes in $X(W,S)$ with bounded  projection to $\widetilde{\ax}(h)$. Since $h$ acts by translation on $\widetilde{\ax}(h)$, $\widetilde{\ax}(h)$ enters into an infinite descending chain of strongly separated half-spaces bounded by those hyperplanes.

Indeed, let $\xi\in [h^+]$ be an accumulation point of $h^n o$. By the above discussion, $\xi$ is contained in an infinite descending chain of strongly separated half-spaces. Note that any two points  in $[h^+]$ being viewed as orientations of half-spaces have finite symmetric difference on half-spaces, we see that any two $\eta,\xi\in [h^+]$ are contained in an infinite descending chain of strongly separated half-spaces. By Lemma \ref{minimalcriterion}, $\eta=\xi$, so $[h^+]$ is minimal.   

The ``in particular" statement follows from Lemma \ref{NSDynamics on Horofunction boundary} where the north-south dynamics is proved relative to the $[\cdot]$-classes $[h^+],[h^-]$.
\end{proof}
\begin{cor}\label{AssumpACoxeter}
In the setup of Lemma \ref{NorthSouthDynamicsCoxeter}, let $x_n$ be a sequence of points so that their projection $\pi_{\widetilde{\ax}(h)}(x_n)$ to the axis $\widetilde{\ax}(h)$ gets unbounded. Then any accumulation point of $x_n$ is either $h^+$ or $h^-$ (depending on $x_n$ on the positive ray or the negative ray).     
\end{cor}
\begin{proof}
Assume that $\pi_{\widetilde{\ax}(h)}(x_n)$ is  on a positive ray $\gamma$ of $\widetilde{\ax}(h)$.  We are going to prove that $x_n\to h^+$. By the contracting property of $\widetilde{\ax}(h)$, $y_n\in \pi_{\widetilde{\ax}(h)}(x_n)$ is $C$-close to $[o,x_n]$ for a fixed constant $C$. Note that $\gamma$ crosses an infinite descending chain of strongly separated hyperplanes $\mathfrak k_n$, so $[o,y_n]$ also crosses an unbounded number of those   hyperplanes as $n\to\infty$. Let $\hat s_n$ be the half-spaces delimited by $\mathfrak k_n$ into which $\gamma$ enters. Hence, for any fixed $m\ge 1$, all but finitely many $y_n$ and $x_n$ are contained in $\hat s_m$, so $x_n$ tends to the intersection $\cap_{n\ge 1} \hat s_n$ which is exactly $h^+$ by Lemma \ref{minimalcriterion}.       
\end{proof}


The following lemma  shall be used to prove the double density of fixed point pairs of contracting elements.
\begin{lem}\label{doublelimitcriterion}
Let $h, k$ be two independent  contracting isometries on $X(W,S)$ so that each of their axes crosses two strongly separated hyperplanes. Then for any $n\gg 0$, $g_n:=h^nk^n$ is  a contracting element so that 
\begin{enumerate}[label=(\roman*)]
    \item 
    the axis of $g_n$ crosses two strongly separated hyperplanes;
    \item
    the fixed $[\cdot]$-classes $[g_n^+]$ and $[g_n^-]$ are singletons;

    \item 
    $g_n^+$ tends to $h^+$ and $g_n^-$ tends to $h^-$ as $n\to\infty$. 
\end{enumerate}
\end{lem}
\begin{proof}
The contracting property of $g_n$ is proved in \cite[Lemma 3.13]{YangConformal} by showing that for all $n\gg 0$, the path $$\gamma_n=\cup_{i\in \mathbb Z} g_n^i[o,h^no]h^n[o,k^no]$$ is $C$-contracting for a constant $C$ independent of $n$. Write $\gamma:=\gamma_n$ in the proof for sake of simplicity.  Once \textit{(ii)} is proved, \textit{(iii)} also follows from \cite[Lemma 3.13]{YangConformal}, where the convergence is proved upon taking $[\cdot]$-classes. By Lemma \ref{minimalcriterion}, \textit{(ii)} actually follows from \textit{(i)}. So our goal is to prove \textit{(i)}.

Since $\gamma$ is Morse by Lemma \ref{lem:Morse},  there exists   some $R>0$ depending on $C$ so that any combinatorial geodesic  with endpoints on $\gamma$ is contained in the $R$-neighborhood of $\gamma$. By a limiting argument using Ascoli-Arzela Lemma, we  produce a bi-infinite geodesic  $\alpha$ in the $R$-neighborhood of $\gamma$. 

Note the following fact: Let  $\alpha,\beta$ be two geodesic segments with endpoints at most $R$-apart. If $\alpha$ crosses $N$ distinct hyperplanes, then at least $(N-2R)$ of those hyperplanes  are crossed by $\beta$.  Indeed, each hyperplane   separates either $\alpha^-,\beta^-$, or $\alpha^+,\beta^+$, or $\beta^-,\beta^+$. Thus, at least $(N-2R)$ of them separates $\beta^-,\beta^+$.

Consequently, if we choose $n\gg 0$ sufficiently large, $[o,h^no]$ crosses at least $(2R+2)$ pairwise strongly separated hyperplanes. By the above discussion, $\alpha$  must cross two strongly separated hyperplanes. By construction,  $\gamma$ contains infinitely many copies of $[o,h^no]$, so we see that $\alpha$  crosses an bi-infinite sequence of strongly separated hyperplanes. By Lemma \ref{minimalcriterion}, this implies that the endpoints of $\alpha$ are both minimal. Note that the endpoints of $\alpha$ are in the same $[\cdot]$-classes as the those of $\gamma=\gamma_n$, which are $[g_n^-]$ and $[g_n^+]$, so the assertion (ii) follows.  
\end{proof} 

Under the assumption on $W$ in Lemma \ref{NorthSouthDynamicsCoxeter}, 
let $\mathcal C$ denote the set  of contracting isometries in $W$  on $X(W,S)$  so that its axis crosses a pair of  strongly separated hyperplanes.

\begin{cor}
Let $(W,S)$ be an infinite irreducible non-affine Coxeter group.  Then $\mathcal C$  contains infinitely many pairwise independent elements, and each element in $\mathcal C$ admits north-south dynamics on $\partial_h X(W,S)$.      
\end{cor}

\subsection{Minimal actions on  boundaries for irreducible Coxeter groups}
Let $(W,S)$ be an  irreducible non-spherical  non-affine Coxeter group.  Recall that $\mathcal C$ denote the set  of contracting isometries  on $X(W,S)$  so that its axis crosses at least two strongly separated hyperplanes. 

The main result of this subsection is as follows.
\begin{thm}\label{MinimalBdryActionCoxeter}
Let $(W,S)$ be an infinite irreducible non-affine Coxeter group.  Then 
\begin{enumerate}[label=(\roman*)]
  
    \item 
    The horofunction boundary $\partial_h X(W,S)$ is the unique and minimal $W$-invariant closed subset  and it is a perfect set in the sense that there is no isolated points. 
    
    \item 
    For any  $h\in \mathcal C$, the  set $\{gh^\pm: g\in W\}$     is dense in $\partial_h X(W,S)$.  
    \item 
    The fixed point  pairs of all  elements in $\mathcal C$ are dense in $\partial_h X(W,S)\times \partial_h X(W,S)$.
\end{enumerate}
\end{thm}

Before moving on, let us present   examples so that certain contracting elements may not have minimal fixed points and the limit set is a proper subset of the whole boundary.
\begin{eg}\label{example with non-minimal fixed points}
Let $\mathbb F_2$ be a free group of rank 2. We construct a $\cat$ cube complex $X$ so that the action on $X$ is geometric and essential. The space $X$ is the universal cover of a thickening $Y$ of a rose. Namely, let $S$ be the flat annulus with a natural cube complex structure that is obtained from the rectangle $[-2,2]\times [-1,1]$ by identifying $\{-2\}\times  [-1,1]$ with $\{2\}\times  [-1,1]$. Let $Y$ be the union of two copies $S_1,S_2$ of $S$ by gluing the four-square $[-1,1]\times [-1,1]$ in $S_1$ and $S_2$. Since $\mathbb F_2$ is hyperbolic, so  $X$ is hyperbolic on which every nontrivial element is contracting. 

Let $s_1,s_2\in \mathbb F_2$ be the element associated with $S_1,S_2$, so $\mathbb F_2=\langle s_1,s_2\rangle$.  The universal covering $X$ of $Y$ is a $\cat$ cube complex, where $S_1, S_2$ lift to flat strips $\tilde S_1\cong \tilde S_2\cong \mathbb R\times [-1,1]$. Note that $s_1, s_2$ preserve $\tilde S_1, \tilde S_2$ by translation, and three parallel lines in $\tilde S_i$ with $i=1,2$ define an equivalent class of three distinct points in $\partial_h X$. Thus, fixed points of  $s_1,s_2$ are not singletons. On the other hand, it is easy to check that all nontrivial elements in $\mathbb F_2$ that are not conjugates of $s_1,s_2$ have singleton fixed points, so have north-south dynamics on $\partial_h X$. Hence,  $\partial_h X$ contains a  unique minimal $\mathbb F_2$-invariant closed subset denoted by $\Lambda \mathbb F_2$. However, the  points defined by the middle lines in $\tilde S_1,\tilde S_2$ are isolated points in   $\partial_h X$, so are not contained in $\Lambda \mathbb F_2$. See Figure \ref{fig: ex5.14} for the illustration.

Such construction could be performed for any geometric action $G$ on a $\cat$ cube complex $X$. Namely, consider the carrier of a hyperplane $\mathfrak h$, and thicken it to be $[-1,1]\times \mathfrak h$, and do it equivariantly for every $G$-translate of $\mathfrak h$.  The resulting space has isolated points in the horofunction boundary as above, and certain contracting elements would fail to have minimal fixed points.   
\end{eg}

\begin{figure}[ht]
    \centering
\begin{tikzpicture}
      \tikzset{enclosed/.style={draw, circle, inner sep=0pt, minimum size=.1cm, fill=black}}
      
      \node at (-1,-0.5) {$\tilde{S}_1$};
      \node at (7,0) {$\cdots$};
      \node at (-0.5,0) {$\cdots$};
      \node at (0,2) {$\cdots$};
      \node at (0,-2) {$\cdots$};
      \node at (6,2) {$\cdots$};
      \node at (6,-2) {$\cdots$};

      \node at (0.8,3.5) {$\tilde{S}_2$};
      \node at (1.5,3.2) {$\cdots$};
      \node at (1.5,-3.2) {$\cdots$};
      \node at (4.5,3.2) {$\cdots$};
      \node at (4.5,-3.2) {$\cdots$};

 \begin{scope}[blend mode=darken]     
      \fill[red!10] (0,0.5) -- (6.5,0.5) -- (6.5,-0.5) -- (0,-0.5);
      \fill[red!10] (0.5,2.5) -- (2.5,2.5) -- (2.5,1.5) -- (0.5,1.5);
      \fill[red!10] (0.5,-2.5) -- (2.5,-2.5) -- (2.5,-1.5) -- (0.5,-1.5);
      \fill[red!10] (3.5,2.5) -- (5.5,2.5) -- (5.5,1.5) -- (3.5,1.5);
      \fill[red!10] (3.5,-2.5) -- (5.5,-2.5) -- (5.5,-1.5) -- (3.5,-1.5);

      \fill[blue!10] (1,-3) -- (1,3) -- (2,3) -- (2,-3);
      \fill[blue!10] (4,-3) -- (4,3) -- (5,3) -- (5,-3);
\end{scope}
      
      \draw (0,0) -- (6.5,0) node[midway, sloped, above] {};
      \draw (0,0.5) -- (6.5,0.5) node[midway, sloped, above] {};
      \draw (0,-0.5) -- (6.5,-0.5) node[midway, sloped, above] {};

      \draw (1.5,3) -- (1.5,-3) node[midway, sloped, above] {};
      \draw (1,3) -- (1,-3) node[midway, sloped, above] {};
      \draw (2,3) -- (2,-3) node[midway, sloped, above] {};

      \draw (0.5,2) -- (2.5,2) node[midway, sloped, above] {};
      \draw (0.5,1.5) -- (2.5,1.5) node[midway, sloped, above] {};
      \draw (0.5,2.5) -- (2.5,2.5) node[midway, sloped, above] {};

      \draw (0.5,-2) -- (2.5,-2) node[midway, sloped, above] {};
      \draw (0.5,-1.5) -- (2.5,-1.5) node[midway, sloped, above] {};
      \draw (0.5,-2.5) -- (2.5,-2.5) node[midway, sloped, above] {};

      \draw (5,-3) -- (5,3) node[midway, sloped, above] {};
      \draw (4.5,-3) -- (4.5,3) node[midway, sloped, above] {};
      \draw (4,-3) -- (4,3) node[midway, sloped, above] {};

      \draw (3.5,2) -- (5.5,2) node[midway, sloped, above] {};
      \draw (3.5,1.5) -- (5.5,1.5) node[midway, sloped, above] {};
      \draw (3.5,2.5) -- (5.5,2.5) node[midway, sloped, above] {};

      \draw (3.5,-2) -- (5.5,-2) node[midway, sloped, above] {};
      \draw (3.5,-1.5) -- (5.5,-1.5) node[midway, sloped, above] {};
      \draw (3.5,-2.5) -- (5.5,-2.5) node[midway, sloped, above] {};

      \draw (3,-0.5) -- (3,0.5) node[midway, sloped, above] {};
      \draw (6,-0.5) -- (6,0.5) node[midway, sloped, above] {};

      \draw (1,1) -- (2,1) node[midway, sloped, above] {};
      \draw (1,-1) -- (2,-1) node[midway, sloped, above] {};
      \draw (4,1) -- (5,1) node[midway, sloped, above] {};
      \draw (4,-1) -- (5,-1) node[midway, sloped, above] {};

\end{tikzpicture}
    \vspace*{-0.5cm}
    \caption{Examples of $\cat$ cube complex $X$ with proper limit set}
    \label{fig: ex5.14}
\end{figure}
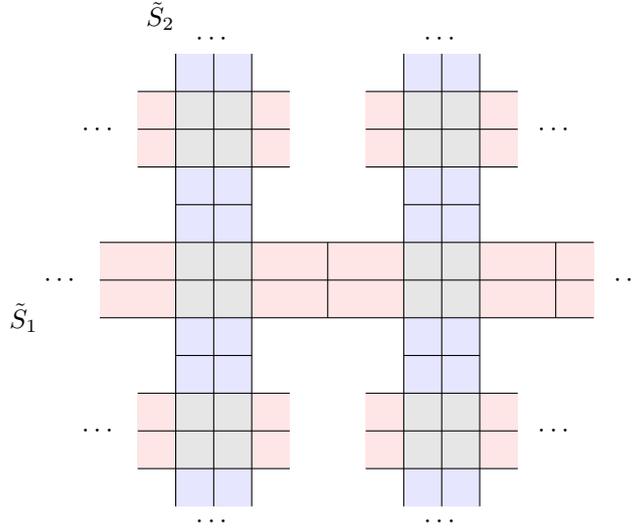

We first  verify the assertion (i) in Theorem \ref{MinimalBdryActionCoxeter}.

\begin{lem}\label{MinimalLimitSetCoxeter}
Assume that $(W,S)$ is irreducible, non-spherical and non-affine.
Then there is a unique  minimal $W$-invariant closed subset denoted as $\Lambda_h (W,S)$ in the horofunction boundary.   
\end{lem}
\begin{proof}
Let $h$ be a contracting isometry with minimal fixed points.
Let $\Lambda$ be the topological closure of the fixed points of all conjugates $ghg^{-1}$. By \cite[Lemma 3.30]{YangConformal}, $\Lambda$ is the minimal $G$-invariant closed subset. For completeness, we recall the proof.  

First of all, any $G$-invariant closed subset $A$ contains at least
three points. Otherwise, $G$ would fix a point or a pair of points, which contradicts the North–South dynamics of conjugates of $h$, as the set of their fixed points is infinite. Now, as $[h^+] = \{h^+\}$ is minimal, we can thus choose $x \in A \setminus h^\pm$. By North–South dynamics for $h$, $h^nx$ converges to $h^+$. This implies $h^+\in A$, since $A$ is closed and $G$-invariant.
Thus, $A$ contains $\Lambda$ and the minimal action follows.
\end{proof}

\begin{lem}\label{MinimalLimitSetIsWholeBdryCoxeter}
In the setup of Lemma \ref{MinimalLimitSetCoxeter},  $\Lambda_h(W,S)=\partial_hX(W,S)$. 
\end{lem}
\begin{proof}
Let $\Lambda_h (W)$ denote the limit set in $\partial_hX(W,S)$ of the vertex set  $W$ of $X(W,S)$, which consists of accumulations points in the boundary. Since the action of $W$ on $X(W,S)$ is cocompact, we have $\Lambda_h (W)=\partial_hX(W,S)$. Indeed, any boundary point is represented by a geodesic ray, so the vertex set on the ray tends to it.  

Fix  a contracting isometry  $h\in \mathcal C$  on $X(W,S)$ so that $h$ has  the minimal fixed points $h^-,h^-$ in $\partial_hX(W,S)$. In the proof of Lemma \ref{MinimalLimitSetCoxeter}, $\Lambda_h(W,S)$ is the topological closure of  $\{gh^-,gh^+:g\in W\}$. 
By Lemma \ref{unique limitset on Horofunction boundary}, the $[\cdot]$-closure of $\Lambda_h(W,S)$ (i.e. the union of $[\cdot]$-classes over $\Lambda_h(W,S)$) is exactly $\Lambda_h (W)=\partial_hX(W,S)$. That is, every point $\xi$ in $\partial_hX(W,S)$ is contained in a $[\cdot]$-class of some $\xi'\in \Lambda_h(W,S)$. 
Then there exists $g_n\in W$ so that $g_nx\to \xi'$ for some $x\in \{h^-,h^+\}$.

If $[\xi']$ is singleton, then there is nothing to do, as $\xi=\xi'$. Let us assume that $[\xi]$ is non-singleton. Let  $W(\xi')$ be the subgroup of $W$ generated by the set of reflections about hyperplanes which contain $\xi'$ in their boundaries. Note that $W(\xi')=W(\xi)$ by Remark \ref{rmk defn Wxi}.  By Proposition \ref{LamThomasPartition}, the points of $[\xi']$ are one-to-one correspondence to the elements of $W(\xi')$. Say, $h\in W(\xi')$ is the corresponding element for  $\xi\in [\xi']$ so that $h\xi'=\xi$. Since $g_nx\to \xi'$, we obtain $hg_nx\to \xi$. This implies that every point of $\partial_h X(W,S)$ is an accumulation point of $\Lambda_h(W,S)$. The proof is complete.  
\end{proof}
 
\begin{proof}[Proof of Theorem \ref{MinimalBdryActionCoxeter}]
The assertion (i) has been proved, and then (ii) is proved by Lemma \ref{MinimalLimitSetCoxeter}.   The double density (iii)  follows from Lemma \ref{doublelimitcriterion}. 
\end{proof}




\subsection{Myrberg limit sets in the horofunction  boundary}
In this subsection, we study the Myrberg limit sets in the horofunction  boundary for Coxeter groups. 

In \cite[Section 4]{YangConformal}, Myrberg limit points are defined for any discrete non-elementary group with contracting isometries on a general metric space. The definition is almost same as the one given in  Definition \ref{defn:Myrberg} for $\cat$ spaces: the only difference is to add $[\cdot]$-closure to the convergence. We would not get into the details here. Rather, in the current setup, we could make the same  definition as the one for $\cat$ spaces.    
\begin{defn}
A  point $z\in \partial_h X(W,S)$ is called \textit{Myrberg point} if for any point $w\in X$, the orbit $G(z,w)=\{(gz,gw): g\in G\}$ is dense in the set $\partial_h X(W,S)\times \partial_h X(W,S)$  in the following sense:
\begin{itemize}
    \item for any $(a,b)\in \partial_h X(W,S)\times \partial_h X(W,S)$ there exists $g_n\in G$ so that $g_nz\to a$ and $g_nw\to b$ in the horofunction compactification.
\end{itemize}  
We denote by $\partial_h^{Myr} X(W,S)$ the set of Myrberg limit points in $\partial_h X(W,S)$.    
\end{defn}

Recall that $\mathcal C$ is the set of contracting elements so that their axis crosses at least two strongly separated hyperplanes in Theorem \ref{MinimalBdryActionCoxeter}. The next definition is  the same as Definition \ref{defn: recurrence with arbitrary accuracy} for $\cat$ spaces, except the additional assumption that $h$ is contained in $\mathcal C$. 

\begin{defn}\label{defn: recurrence with arbitrary accuracy coxeter}
Given $h\in \mathcal C$, we say that a geodesic ray $\gamma$ is \textit{recurrent to $h$ with arbitrary accuracy}, if there exists a constant $R$ depending only on  $\mathrm{Ax}(h)$ with the following property. For any large $L\ge 1$, the geodesic ray $\gamma$ contains a segment of length $L$ in  $N_R(g\mathrm{Ax}(h))$ for some $g\in G$. 
\end{defn}

\begin{lem}\label{MyrbergCharCoxeter}
A point $z\in \partial_h X(W,S)$ is Myrberg if and only if for any contracting  isometry $h\in\mathcal C$, any geodesic ray $[o,z]$ is recurrent to $h$ with arbitrary accuracy.     
\end{lem}
\begin{proof}
The proof is almost identical to the proof of Lemma \ref{lem: characterizemyrberg}, modulo the above definitions and Theorem \ref{MinimalBdryActionCoxeter}. The main difference to spell out is the convergence now in the horofunction compactification instead of the cone topology. We shall use Roller compactification due to the homeomorphism $\partial_R X(W,S)\cong \partial_hX(W,S)$ by Theorem \ref{allboundaryaresame}.  

The $\Rightarrow$ direction is the same as the one of Lemma \ref{lem: characterizemyrberg}, which uses only the metrizability, so  works for Roller compactification. As for the $\Leftarrow$ direction,    assume that any geodesic ray $\gamma=[o,z]$ is recurrent to $h$ with arbitrary accuracy. Then there exists $g_n\in W$ so that $g_n\gamma$ intersects a fixed ball and $g_n\gamma\cap N_R(\ax(h))$ goes unbounded for a fixed $R>0$. We need to show that $g_no,g_nz$ converge to $h^-,h^+$ respectively. This follows from Corollary \ref{AssumpACoxeter}. The lemma is proved. 
\end{proof}

We shall consider the Myrberg limit points in the visual boundary $\partial_\infty \Sigma(W,S)$ for the action of $W$ on the Davis complex. Let  $\partial_\infty^{Myr} \Sigma(W,S)$ denote Myrberg limit points in  $\partial_\infty \Sigma(W,S)$.   

\begin{lem}\label{MyrbergCharCoxeter Davis}
A point $z\in \partial_\infty \Sigma(W,S)$ is Myrberg if and only if for any rank-one  isometry $h\in\mathcal C$, any geodesic ray $[o,z]$  is recurrent to $h$ with arbitrary accuracy.     
\end{lem}
\begin{proof}
By Remark \ref{rmk: sub-class rank one for Myrberg}, it suffices to show that the fixed point pairs of  all $h\in\mathcal C$ are dense in $\partial_\infty \Sigma(W,S)$. 
Indeed, since the action $W\curvearrowright \Sigma(W,S)$ is geometric, the limit set $\Lambda W$ in the visual boundary $\partial_\infty \Sigma(W,S)$ is the whole boundary. To prove double density, let $\xi\ne \eta\in \partial_\infty \Sigma(W,S)$.  Let us fix a rank-one isometry $h\in\mathcal C$ with fixed points $h^-,h^+$ in $\partial_\infty \Sigma(W,S)$. By Lemma \ref{lem:rankone-bigthree},  the set $\{gh^-,gh^+: g\in W\}$ is dense in $\partial_\infty \Sigma(W,S)$. Pick two sequence of elements $g_n, f_n\in W$ so that $g_nh^-\to \xi$ and $f_nh^+\to \eta$. By Lemma \ref{doublelimitcriterion}, denoting $g=g_n$ and $f=f_n$, the two fixed points of $f^mg^m$ tends to $(f_nh^+,g_nh^-)$ as $n\to\infty$, and, in addition, $f^mg^m$ belongs to $\mathcal C$. Hence, we could find a sequence of integers $\{m_n\}_n$ so that the two fixed points of $f_n^{m_n}g_n^{m_n}$  tend to $(\xi,\eta)$ as $n\to\infty$. The proof is complete.  
\end{proof}

Recall the Tits boundary has the same underlying set with the visual boundary. We shall also understand the Myrberg limit set $\partial_\infty^{Myr}\Sigma(W,S)$ in the Tits boundary $\partial_T\Sigma(W,S)$.

\begin{cor}\label{MyrbergisMinimal}
Myrberg points in $\partial_T\Sigma(W,S)$ (resp. $\partial X_h(W,S)$) are minimal in Lam-Thomas partition on $\partial_T\Sigma(W,S)$ (resp. finite difference partition on $\partial X_h(W,S)$).    
\end{cor}
\begin{proof}
The minimality for Myrberg points in $\partial_T\Sigma(W,S)$ follows from Lemma \ref{MyrbergVisual} which says that Myrberg points is visual from any other boundary point, while the Lam-Thomas partition is Tits geodesically connected by Proposition \ref{LamThomasPartition}. 

The case for Myrberg points in $\partial X_h(W,S)$ follows from Lemma \ref{squeezingraysareminimal}: indeed, by Lemma \ref{MyrbergCharCoxeter}, for any  Myrberg point $\xi\in \partial X_h(W,S)$, there is some geodesic ray $\gamma$ ending at $\xi$, which intersects unboundedly the fixed neighborhoods  of infinite many axes of  contracting isometries $h_n\in \mathcal C$. Recall that a contracting isometry in $\mathcal C$ admits an axis that crosses infinitely many strongly separated hyperplanes. Thus, $\gamma$  crosses infinitely many pairwise strongly separated hyperplanes, so by Lemma \ref{squeezingraysareminimal}, the end point of $\gamma$ is minimal. 
\end{proof}

\subsection{Homeomorphic Myrberg limit sets}

The following is the main result of this subsection.
\begin{prop}\label{homeomorphicMyrbergset}
There is a $W$-equivariant homeomorphism between the Myrberg limit set $\partial_h^{Myr} X(W,S)$ in $\partial_hX(W,S)$ and the Myrberg limit set $\partial_\infty^{Myr} \Sigma(W,S)$ in $\partial_\infty\Sigma(W,S)$.    
\end{prop}

In the preceding subsection, we defined the Myrberg limit set in the horofunction boundary, in a very similar way  as the one in the visual boundary.
The difficulty of Proposition \ref{homeomorphicMyrbergset}  lies in that, on one hand, the Myrberg limit points are characterized in a very metric term as in Lemma \ref{MyrbergCharCoxeter} and Lemma \ref{MyrbergCharCoxeter Davis}. On the other hand, the $\cat$ metric on $\Sigma(W,S)$ are very different from the combinatorial metric on $X(W,S)$. To resolve this ``inconsistency", we use the fact that the two shortest projections to the axis of contracting isometries agree up to a bounded error by Lemma \ref{projectionsareclose}.  We now give the full details.

\begin{proof}[Proof of Proposition \ref{homeomorphicMyrbergset}]
We first build a map $\Phi: \partial_h^{Myr}X(W,S)\to \partial_\infty^{Myr}X(W,S)$.

Let $\alpha$ be a combinatorial geodesic ray in $X(W,S)$ ending at a Myrberg limit point $\xi\in \partial_h^{Myr}X(W,S)$. By Lemma \ref{MyrbergCharCoxeter},  $\alpha$ is recurrent to  any contracting isometry $h \in \mathcal C$ on the Cayley graph $X(W,S)$. By definition,  $\alpha$ intersects an infinite sequence of $N_R(g_n\ax(h))$ for $g_n\in W$ with diameter tending to $\infty$. For each  intersection, let $x_n,y_n$ be the corresponding entry and exit points of $\alpha$. It is easy to verify that  $x_n$ (resp. $y_n$) are uniformly close to the combinatorial projections of the previous $N_R(g_{n-1}\ax(h))$ $o$ (resp. the next  $N_R(g_{n+1}\ax(h))$)  to $N_R(g_n\ax(h))$.  

By Lemma \ref{projectionsareclose}, the $\cat$ projections of $g_n\ax(h)$ to the previous $g_{n-1}\ax(h)$ and next $g_{n+1}\ax(h)$ are uniformly close to $y_n$ and $x_{n+1}$. We then construct an admissible ray $\beta$ as in Definition \ref{AdmDef} by connecting these  shortest $\cat$ projections between $g_n\ax(h)$ and $g_{n-1}\ax(h)$. By Proposition \ref{admisProp} and Remark \ref{rmk CAT fellow travel}, any $\cat$ geodesic with endpoints on $\beta$ fellow travels along $\beta$.   Using Ascoli-Arzela Lemma, this implies that a subsequence of the $\cat$ geodesics $[o,x_n]$ tends to  a  limiting $\cat$ geodesic ray denoted by $\alpha'$. Moreover, the intersection of $\alpha'$ with $N_R(\widetilde{\ax}(h))$ in the Davis complex tends to infinity as $n\to\infty$. This means that  $\alpha'$ is recurrent to any contracting isometry $h\in\mathcal C$, so it tends to a Myrberg point denoted by $\Phi(\xi)$ in $\partial_T \Sigma(W,S)$ by Lemma \ref{MyrbergCharCoxeter Davis}.  

The above argument could be reversible to define a map $\Psi: \partial_\infty^{Myr}X(W,S)\to \partial_h^{Myr}X(W,S)$: if $\alpha$ is a $\cat$ geodesic ray in $\Sigma(W,S)$ ending at a Myrberg limit point $\xi\in \partial_\infty^{Myr}\Sigma(W,S)$. By Lemma \ref{MyrbergCharCoxeter},  $\alpha$ is recurrent to  any contracting isometry $h \in \mathcal C$ on the  Davis complex $\Sigma(W,S)$. Using Lemma \ref{projectionsareclose}, we could construct an admissible ray $\beta$ as above by connecting the  shortest projections between $g_n\ax(h)$ and $g_{n-1}\ax(h)$. The admissible ray $\beta$ fellow travels geodesics along  $g_n\ax(h)$ by Proposition \ref{admisProp}. (There is a bit difference from the above:  the fellow travel here is in a weak sense of Definition \ref{Fellow}, rather than the one in Remark \ref{rmk CAT fellow travel}.) So a limiting argument using Arzela-Ascoli Lemma to produce a combinatorial geodesic ray $\alpha'$  in $X(W,S)$ ending at a Myrberg limit point denoted as $\Psi(\xi)\in \partial_hX(W,S)$. By construction, the geodesic ray $\alpha'$ passes through the sequence of contracting axes as $\alpha$ does. Thus, $\alpha'$  tends to a Myrberg point denoted by $\Psi(\xi)$ in $\partial_h^{Myr} \Sigma(W,S)$ by Lemma \ref{MyrbergCharCoxeter}.  

Consequently, we see that the two maps $\Phi,\Psi$ are inverses to each other. Thus, $\Phi$ is a bijective map. The continuity could be checked readily and is left to interested reader. 
\end{proof}

We summarize the above discussion as follows.
\begin{thm}\label{minimal action and Myrberg sets}
Let $(W,S)$ be an irreducible nonspherical non-affine Coxeter group. Then the actions of $W$  on the visual boundary of the Davis complex and the horofunction boundary of the Cayley graph are   minimal, topologically free, strong boundary actions. Moreover, the corresponding Myrberg limit sets $\partial_\infty^{Myr} \Sigma(W,S)$ and $\partial_h^{Myr} X(W,S)$ are $W$-equivariant homeomorphic to each other.      
\end{thm}
\begin{proof}
The minimal action on the visual boundary  follows from Lemma \ref{lem:rankone-bigthree} (which does not require $W$ to be irreducible, only the existence of rank-one elements), while the one on Roller boundary follows from  Theorem \ref{MinimalBdryActionCoxeter}.  The topologically free action on visual boundary follows from Theorem \ref{thm: topo free rank one}. The one for the horofunction boundary is a consequence of the fact that the two Myrberg limit sets are  homeomorphic by Proposition \ref{homeomorphicMyrbergset}. Since the actions are minimal actions, the strong boundary actions are consequences of north-south dynamics (see Lemmas \ref{lem: rankone northsouth} and \ref{NorthSouthDynamicsCoxeter}).
\end{proof}

For $\cat$ cube complexes, one could prove  a similar statement as Proposition \ref{homeomorphicMyrbergset}. We here take a shortcut using the existing results in literature to  deduce only the topological free action. 

\begin{thm}\label{topo free Roller}
Let $G\curvearrowright X$ be a proper essential  isometric action of a non-elementary group on a proper irreducible $\cat$ cube complex $X$ with a rank-one element. If $X$ is not geodesically complete in the $\cat$ metric, assume further that the action is co-compact.    Suppose the elliptic radical $E(G)$ is trivial. Then {there exists a} Myrberg point $z\in \Lambda_R G\subset \partial_R X$   for the  action $ G\curvearrowright \Lambda_R G$ {so that $z$} is a free point in the sense that $\stab_G(z)=\{e\}$. Therefore, the  action $ G\curvearrowright \Lambda_R G$ is topologically free.    
\end{thm}
\begin{proof}
Recall that a boundary point $\xi\in \del_RX$  is called \textit{squeezing} if there exists $r>0$ so that $\xi$ viewed as an ultrafilter of half spaces contains a sequence of pairs of super strongly separated half spaces at a distance at most $r$. 
First, note that a Myrberg point is a squeezing point by \cite[Lemma 11.5]{YangConformal}. There, a Myrberg point is proved to be a regular point, but the same proof shows it is also squeezing. 
By \cite[Lemma 6.10]{FLM}, there is an $\mathrm{Aut}(X)$-equivariant bijection between the squeezing points in $\partial_\infty X$ and the squeezing points in $\partial_R X$. 
So the conclusion follows from Theorem \ref{thm: topo free rank one}.
\end{proof}

If $X$ is strictly non-Euclidean (i.e. each factor is non-Euclidean), the following important theorem was established in \cite{N-S} for proper cocompact actions.

\begin{thm}\cite[Theorem 5.1, Theorem 5.8]{N-S}\label{thm: contratible set and boundary}
Let $X$ be a locally finite, finite dimensional essential strictly non-Euclidean $\cat$ cube complex admitting a proper cocompact action of $G\leq \aut(X)$. Then $B(X)$ is a $G$-boundary.
\end{thm}

\begin{rmk}\label{comparing bdry}
Let $X$ be a locally finite, finite-dimensional, irreducible, non-Euclidean $\cat$ cube complex and $G\leq \aut(X)$ acts properly.
\begin{enumerate}
    \item\label{comparing bdry-minimality}  Another two boundaries denoted by $R(X)$ and $S(X)$ for $X$ called \textit{regular boundary} and \textit{strong separate boundary} were introduced in \cite{Fer} and \cite{K-S} respectively.  Note that all squeezing points are regular and thus belong to $S(X)$ by \cite[Remark 6.7]{FLM}, which implies $\Lambda_RG\subset S(X)$. In addition, $S(X)$ is a $G$-invariant closed subset of $B(X)$. Therefore,  if the action $G\curvearrowright X$ is cocompact, then $\Lambda_R G=S(X)=B(X)$ by Theorem \ref{thm: contratible set and boundary} as $G$-boundary actions are minimal.
 \item\label{comparing bdry-new proof} Under an additional mild assumption that $X$ is geodesically complete or the action $G\curvearrowright X$ is cocompact,  Theorem \ref{topo free Roller} also provides a new approach to the second part of \cite[Proposition 1.1]{K-S} on the topological freeness on $S(X)$, which is one of the key ingredients of \cite{K-S}. Indeed, under the assumption of \cite[Proposition 1.1]{K-S}, the first part of \cite[Proposition 1.1]{K-S} has demonstrated that $G\curvearrowright S(X)$ is minimal, which implies that $\Lambda_R G=S(X)$. Then Rank rigidity theorem for $\cat$ cube complex recorded in \cite[Theorem 6.3]{C-S} yields a rank-one isometry in $G$. Thus, Theorem \ref{topo free Roller} shows that the action $G\curvearrowright S(X)$ is topologically free.
    \end{enumerate}
 \end{rmk}

 Now we have the following stronger version of Theorem \ref{thm: contratible set and boundary}

 \begin{thm}\label{strong boundary cube complex}
 Let $X$ be a locally finite essential irreducible non-Euclidean finite dimensional $\cat$ cube complex admitting a proper cocompact action of $G\leq \aut(X)$. Suppose the elliptic radical $E(G)$ is trivial. Then  $G\curvearrowright B(X)$ is topologically free, topologically amenable strong boundary action.
 \end{thm}
 \begin{proof}
Since the action is cocompact and the cube complex $X$ is essential,  the action $G\curvearrowright X$ by automorphism is essential. Now because $X$ is irreducible, rank rigidity theorem for $\cat$ cube complex recorded in \cite[Theorem 6.3]{C-S} implies that $G$ contains a contracting isometry $g$ whose axis crosses two strongly separated hyperplanes. By Lemma \ref{lem:rankone is contracting}, $g\in G$ is a rank-one isometry. By Lemma \ref{NS dynamics on cube complex},  $g$   has north-south dynamics on $\Lambda_R G=B(X)$, on which the restricted action is also minimal by Theorem \ref{thm: contratible set and boundary}. Thus, it follows from Proposition \ref{prop: north-south visual} that the action $G\curvearrowright B(X)$ is a strong boundary action. In addition, the topological freeness follows directly from Theorem \ref{topo free Roller} and Remark \ref{comparing bdry}\eqref{comparing bdry-minimality}.

 For topological amenability of $G\curvearrowright B(X)$, since $G\curvearrowright X$ is proper, then stabilizer $\stab_G(x)$ of any $x\in X$ is finite. Then \cite[Theorem 5.11]{Du}(see also \cite[Remark P. 1492]{G-N})
     shows that $G\curvearrowright \del_RX$ is topological amenable. Then since $B(X)$ is a $G$-invariant closed subset of $\del_R X$, the action $G\curvearrowright B(X)$ is still topological amenable by Remark \ref{subspace topological amenable}. \end{proof}

\section{Applications to \texorpdfstring{$C^*$}\ -algebras}  \label{sec: application}
We refer to \cite[Section 2.5, Section 4.1]{B-O} for several canonical constructions such as \textit{reduced group $C^\ast$-algebra} $C^\ast_r(G)$ for a discrete group $G$ and \textit{reduced crossed product $C^\ast$-algebra} $C(Z)\rtimes_r G$ for a topological dynamical system $G\curvearrowright Z$.  

\subsection{Applications to reduced group \texorpdfstring{$C^*$}\ -algebras}
Combining \cite[Theorem 1.1]{O} with Theorem \ref{topo free strong boundary on visual summary}, one has the following result on $C^*$-selflessness and $C^*$-simplicity.

\begin{thm}\label{C star selfless}
 Let $G$ be a (non-elementary) group that admits a proper isometric action $G\curvearrowright X$ on a proper  $\cat$ space $X$ with a rank-one element and the elliptic radical $E(G)$ is trivial. Suppose
    \begin{enumerate}[label=(\roman*)]
        \item the action $G\curvearrowright X$ is cocompact; or
        \item the space $X$ is geodesic complete.
    \end{enumerate}
Then $G$ is $C^*$-selfless, i.e., the reduced group $C^*$-algebra $C_r^*(G)$ is selfless in the sense of \cite{Ro}.
\end{thm}

Moreover, since $C^*$-selflessness implies $C^*$-simplicity, Theorem \ref{C star selfless} above provide a new proof on the $C^\ast$-simplicity of groups acting nicely on $\cat$ spaces, which form a large subclass of acylindrically hyperbolic groups with the trivial elliptic radical.  It was shown in \cite[Proposition 1.1]{K-S} that the $C^\ast$-simplicity of a group $G$ acting nicely on a non-Euclidean irreducible finite dimensional $\cat$ cube complex $X$ can be determined by the topological freeness of its action on the so-called strong separate boundary $S(X)$. Note that in the case on the $\cat$ cube complex in \cite{K-S}, the group $G$ automatically contains a rank-one isometry by the rank rigidity theorem in \cite[Theorem A]{K-S} and the explanation before \cite[Corollary B]{K-S}. Therefore, using the visual boundary, Theorem \ref{C star selfless} extends this result to some extent.

\subsection{Application to reduced crossed product \texorpdfstring{$C^*$}\ -algebras}
In this subsection, we mainly establish several regularity properties of the reduced crossed product $C^*$-algebras. We refer to \cite{B-O} as a standard reference for various structural properties such as \textit{nuclearity} and \textit{exactness}, respectively (see \cite[Definition 2.3.1, 2.3.2]{B-O}) and to \cite{Kir}, \cite{K-Rord}, \cite{Kir-Rord} and \cite{Phillips} for the definition of \textit{pure infiniteness} for $C^\ast$-algebras and only make the following remark to emphasize its paradoxical nature.

\begin{rmk}\label{rmk: defn of pure inf}
    A very useful characterization of pure infiniteness was provided in \cite{K-Rord} that a $C^\ast$-algebra $A$ is purely infinite if and
only if every non-zero positive element $a$ in $A$ is properly infinite in the sense of $a\oplus a\precsim a$, where $a\oplus a$ denotes the diagonal matrix with entry $a$ in $M_2(A)$ and ``$\precsim$'' is the \textit{Cuntz subequivalence relation}. 
\end{rmk}

A $C^\ast$-algebra $A$ is said to be a \textit{Kirchberg algebra} if it is separable simple nuclear and purely infinite. The \textit{Uniform Coefficient Theorem}, written as the UCT for simplicity, provides a relationship between KK-groups and the K-groups of $C^*$-algebras. 
It is still a major open question in the classification theory of $C^*$-algebras asking whether all separable nuclear $C^*$-algebras satisfy the UCT. Some partial results are known. See, e.g., Remark \ref{rmk: equivalence of dynamical properties} below.
The celebrated classification theorem by Kirchberg and Phillips (see, e.g., \cite{Kir} and \cite{Phillips}) asserts that all Kirchberg algebras satisfying the UCT can be classified by their K-theory.

A discrete group $G$ is said to be  \textit{exact} if the reduced $C^\ast$-algebra $C^\ast_r(G)$ is exact. It is well-known that $G$ is exact, if and only if, $G$ admits a topological amenable action in the sense of Definition \ref{topo amenable action}, which is further equivalent to  that $G$ has Yu's property A. For the crossed product $C^\ast$-algebra $C(Z)\rtimes_r G$ with exact acting group $G$,  the fact that $C(Z)$ is nuclear and thus exact (see e.g., \cite[Proposition 2.4.2]{B-O}) and \cite[Theorem 10.2.9]{B-O} will yield the following useful observation.

\begin{prop}\label{exact crossed product}
Let $G$ be a countable discrete exact group and $G\curvearrowright Z$ a topological action on a compact Hausdorff space $Z$. Then the $C^\ast$-algebra $A=C(Z)\rtimes_r G$ is exact.
\end{prop}

The following remark collects some other correspondence between $C^\ast$-structure properties of $C(Z)\rtimes_r G$ and dynamical properties of $G\curvearrowright Z$.

\begin{rmk}\label{rmk: equivalence of dynamical properties}
First, it is straightforward to see if $Z$ is a compact metrizable then $C(Z)\rtimes_r G$ is unital and separable. It is well known that if the action $G\curvearrowright Z$ is topologically free and minimal then the reduced crossed product $C(Z)\rtimes_r G$ is simple (see \cite{A-S}) and it is also known that the crossed product $C(Z)\rtimes_r G$ is nuclear if and only if the action $G\curvearrowright Z$ is (topological) amenable (see \cite{B-O}). It is also essentially due to Archbold and Spielberg \cite{A-S} that $C(Z)\rtimes_r G$ is simple and nuclear if and only if the action is minimal, topologically free, and topologically amenable. Moreover, by the classical result of \cite{Tu}, a crossed product $C(Z)\rtimes_r G$ satisfies the UCT if the action is topological amenable. Moreover, it follows from \cite[Theorem 5]{L-S} that the reduced crossed product $C^*$-algebras of minimal topologically free strong boundary actions are simple and purely infinite. This result has been extended to more general situations in \cite{J-R}, \cite{M2} and \cite{M}.
\end{rmk}

Then, combining Remark \ref{rmk: equivalence of dynamical properties} and Theorem \ref{topo free strong boundary on visual summary}, one has the following.

\begin{thm}\label{pure infiniteness from visual}
    Let $G\curvearrowright X$ be a proper isometric  action of a  {non-elementary} group $G$ on a proper  $\cat$ space $X$ with a rank-one element and the elliptic radical $E(G)$ is trivial. Suppose
    \begin{enumerate}[label=(\roman*)]
        \item the action $G\curvearrowright X$ is cocompact; or
        \item the space $X$ is geodesic complete.
    \end{enumerate}
    Then the crossed product $C^\ast$-algebra $A=C(\Lambda G)\rtimes_r G$ for the induced action $G\curvearrowright \Lambda G$ is unital simple separable and purely infinite. Moreover, in the case that $G\curvearrowright X$ is cocompact, the limit set $\Lambda G=\del_\infty X$.
    \end{thm}

As applications, we have the following Theorem for Coxeter groups.

\begin{cor}\label{pure inf irreducible Coxeter}
    Let $W_S$ be an irreducible non-spherical non-affine Coxeter group. The following is true.
    \begin{enumerate}[label=(\roman*)]
        \item\label{coxeter davis} The crossed product $C^\ast$-algebra $A=C(\partial_\infty \Sigma(W, S))\rtimes_r W_S$ of the visual boundary action of the Davis complex $\Sigma(W, S)$ is an exact unital simple separable purely infinite $C^\ast$-algebra.    
        \item The crossed product $C^*$-algebra $B=C(\partial_h X(W, S))\rtimes_r W_S$  is a unital Kirchberg algebra satisfying the UCT and thus classifiable by the K-theory. 
        \end{enumerate}
\end{cor}
\begin{proof}
It follows from Theorem \ref{minimal action and Myrberg sets} and Remark \ref{rmk: equivalence of dynamical properties} that both $A$ and $B$ are unital simple separable purely infinite $C^*$-algebras.
The exactness of $A$ follows from Lemma \ref{exact crossed product} because Coxeter group $W_S$ is known to be exact by \cite{D-J}. The nuclearity of $B$ follows from the topological amenability of $W_S\curvearrowright \Lambda_R(W, S)$ by \cite[Theorem 1.1]{Lec} and  Remark \ref{subspace topological amenable}. Finally, it follows from \cite{Tu} that $B$, as a nuclear reduced crossed product $C^*$-algebra, satisfies the UCT. Then $B$ is a unital Kirchberg algebra satisfying the UCT and thus classifiable by the K-theory.
  \end{proof}

In addition, the first part of Corollary \ref{pure inf irreducible Coxeter} can be generalized to the visual boundary action of discrete subgroup of automorphism groups of locally finite buildings of type $(W, S)$, where $(W, S)$ is an irreducible non-spherical and non-affine Coxeter group. We refer to \cite{A-B}, \cite{Dav}, \cite{C-F} and \cite{Lec} for the backgrounds of buildings and the relation to Coxeter groups. We simply recall that any building $\CB$ can be also equipped with a $\cat$ metric (see \cite[Chapter 18]{Dav}) and with this metric $\CB$ is proper if $\CB$ is locally finite whose definition could be found in the paragraph above Definition 2.2.6 in \cite{Lec}. Thus, we have the following as a further generalization of Corollary \ref{pure inf irreducible Coxeter} since $\Sigma(W, S)$ itself is a locally finite building of type $(W, S)$.

\begin{cor}\label{pure inf building}
       Let $\CB$ be a locally finite building of type $(W,S)$ in which $(W, S)$ is an irreducible non-spherical non-affine Coxeter group. Let $G< \aut(\CB)$ be a non-elementary group acting properly and cocompactly on $\CB$. Suppose $G$ contains a rank-one isometry and $E(G)$ is trivial. Then the boundary action $G\curvearrowright \partial_\infty\CB$ is a minimal topologically free strong boundary action and the crossed product $C^\ast$-algebra $A=C(\partial_\infty \CB)\rtimes_r G$ is a unital simple separable purely infinite exact $C^\ast$-algebra. 
       \end{cor}
       \begin{proof}
          Note that the acting group $G$ is exact by \cite[Corollary 1.2]{Lec}. Then the result follows from Theorem \ref{pure infiniteness from visual}.    
           \end{proof}

\begin{rmk}\label{coxeter and building}
    \begin{enumerate}
        \item It is indicated in the paragraph before \cite[Theorem A]{C-S} that rank rigidity theorem holds for buildings, which is established in \cite{C-F}. This means if the building $\CB$ in Corollary \ref{pure inf building}  is additionally to be geodesically complete, then $G$ contains a rank-one isometry automatically.
        \item\label{topo amen} It seems that there is no definitive answer on whether the $C^\ast$-algebra $A$ in Corollary \ref{pure inf irreducible Coxeter}\ref{coxeter davis} and Corollary \ref{pure inf building} is nuclear or not. For example, when $W_S$ is hyperbolic, which holds exactly when $W_S$ contains no subgroup $\Z^2$  (see, e.g., \cite[Chapter 12]{Dav}), the boundary $\del_\infty\Sigma(W, S)$ are exactly the Gromov boundary and the action on it is known to be topological amenable and the $C^\ast$-algebra $A$ in Corollary \ref{pure inf irreducible Coxeter}\ref{coxeter davis} is thus nuclear (see, e.g., \cite[Proposition 3.2]{A-D}). However, we recall a standard fact that topological amenability of an action $G\curvearrowright Z$ implies that any stabilizer $\stab_G(z)$ is amenable for any $z\in Z$. Therefore, for the actions $W_S\curvearrowright \del_\infty \Sigma(W, S)$ and $G\curvearrowright \del_\infty \CB$ in Corollary \ref{pure inf irreducible Coxeter}\ref{coxeter davis} and \ref{pure inf building}, if there exists a point in the boundary whose stabilizer is non-amenable, then the action is not topological amenable. This, in particular, implies that the $C^\ast$-algebra $A$ in Corollary \ref{pure inf irreducible Coxeter} and \ref{pure inf building} are not nuclear by Remark \ref{rmk: equivalence of dynamical properties}. 
    \end{enumerate}
\end{rmk}

We then have the following applications to actions on the boundaries $B(X)$ of a $\cat$ cube complexes.

\begin{thm}\label{pure infinite irreducible B(X)}
    Let $X$ be a locally finite essential irreducible non-Euclidean finite dimensional $\cat$ cube complex admitting a proper cocompact action of $G\leq \aut(X)$. Suppose the elliptic radical $E(G)$ is trivial. Then  the $C^\ast$-algebra $A=C(B(X))\rtimes_r G$ is a unital Kirchberg $C^\ast$-algebra satisfying the UCT and thus classifiable by the K-theory. 
    \end{thm}
    \begin{proof}
This is a straightforward application of Theorem  \ref{strong boundary cube complex} and Remark \ref{rmk: equivalence of dynamical properties}.
    \end{proof}

\section{Acknowledgements}
The authors would like to thank Jean L\'ecureux for the responses to questions regarding the topological amenability of actions, which led to Remark \ref{coxeter and building}\eqref{topo amen}, for raising questions on topological freeness for Coxeter group actions on horofunction boundaries, and for bringing the authors' attention to \cite{Du} and \cite{Kli20}. Additionally, the authors are grateful to Tianyi Zheng for the helpful conversations during the early stage of the project.

\bibliographystyle{alpha}

\end{document}